\documentclass[a4paper,oneside,11pt]{article}

\usepackage{amsmath,amsfonts,amscd,amssymb}
\usepackage{longtable,geometry}
\usepackage[english]{babel}
\usepackage[utf8]{inputenc}
\usepackage[active]{srcltx}
\usepackage[T1]{fontenc}
\usepackage{graphicx}
\usepackage{pstricks}
\usepackage{bbm}
\usepackage{mathtools}

\usepackage{MnSymbol}
\usepackage{stmaryrd}
\usepackage{nicefrac}
\usepackage{calrsfs}
\usepackage{enumitem}

\usepackage{tocloft}

\usepackage{rotating}

\usepackage{xcolor}
\usepackage{framed}
\usepackage{hyperref}
\usepackage{cleveref}

\colorlet{shadecolor}{blue!15}

\newtheorem{theorem}{Theorem}[section]
\newtheorem{conj}{Conjecture}

\newtheorem{corollary}[theorem]{Corollary}
\newtheorem{lemma}[theorem]{Lemma}
\newtheorem{proposition}[theorem]{Proposition}
\newtheorem{definition}[theorem]{Definition}

\newtheorem{remark}[theorem]{Remark}
\newtheorem{fact}[theorem]{Fact}

\newtheorem{question}[conj]{Question}
  
\newcommand{\be}[1]{\begin{equation}\label{#1}}
\newcommand{\ee}{\end{equation}}
\numberwithin{equation}{section}

\newcommand{\ba}[1]{\begin{align}\label{#1}}
\newcommand{\ea}{\end{align}}
\numberwithin{equation}{section}

\newcommand{\ben}{\begin{equation*}}
\newcommand{\een}{\end{equation*}}
\numberwithin{equation}{section}

\newenvironment{proof}[1][\relax]
  {\paragraph{Proof\ifx#1\relax\else~of #1\fi}}%
  {~\hfill$\square$\par\bigskip}

\newcommand{\calC}{\mathcal{C}}
\newcommand{\calD}{\mathcal{D}}

\newcommand{\calF}{\mathcal{F}}
\newcommand{\calG}{\mathcal{G}}
\newcommand{\calH}{\mathcal{H}}

\newcommand{\calQ}{\mathcal{Q}}
\newcommand{\calR}{\mathcal{R}}

\newcommand{\bbE}{\mathbb{E}}

\newcommand{\bbH}{\mathbb{H}}

\newcommand{\bbN}{\mathbb{N}}

\newcommand{\bbR}{\mathbb{R}}

\newcommand{\bbZ}{\mathbb{Z}}

\newcommand{\bfC}{\mathbf C}

\newcommand{\ep}{\varepsilon}
\newcommand{\eps}{\varepsilon}

\newcommand{\la}{\lambda}
\newcommand{\La}{\Lambda}

\newcommand{\Ann}{{\rm Ann}}
\newcommand{\ind}{\boldsymbol 1}

\setlist[itemize]{itemsep=0pt, topsep=2pt}
\setlist[enumerate]{itemsep=1pt, topsep=4pt}

\newcommand{\rk}[1]{\bgroup\color{red}%
  \par\medskip\hrule\smallskip%
  \noindent\textbf{#1}%
  \par\smallskip\hrule\medskip\egroup}

\newcommand{\Int}{{\rm Int}}
\renewcommand{\ep}{\varepsilon}

\newcommand{\Circ}{\mathrm{Circ}}

\title{Planar random-cluster model: fractal properties of the critical phase}
\author{Hugo Duminil-Copin\footnote{Universit\'e de Gen\`eve (hugo.duminil@unige.ch) and Institut des Hautes \'Etudes Scientifiques (duminil@ihes.fr)}\ \ and Ioan Manolescu\footnote{Universit\'e de Fribourg (ioan.manolescu@unifr.ch)}\ \ and Vincent Tassion\footnote{ETHZ (vincent.tassion@math.ethz.ch)}}
\date{\today}

\begin{document}
\maketitle

\begin{center}
{\em In memory of Harry Kesten (1931-2019).}
\bigbreak
\end{center}
\begin{abstract}
  This paper is studying the critical regime of the planar random-cluster model on $\mathbb Z^2$ with cluster-weight $q\in[1,4)$. More precisely, we prove {\em crossing estimates in quads} which are uniform in their boundary conditions and depend only on their extremal lengths. They imply in particular that any fractal boundary is touched by macroscopic clusters, uniformly in its roughness or the configuration on the boundary. Additionally, they imply that {\em any sub-sequential scaling limit of the collection of interfaces between primal and dual clusters is made of loops that are non-simple}.

We also obtain a number of properties of so-called arm-events: 
{\em three universal critical exponents} (two arms in the half-plane, three arms in the half-plane and five arms in the bulk), 
{\em quasi-multiplicativity} and {\em well-separation} properties (even when arms are not alternating between primal and dual), and the fact that the {\em four-arm exponent is strictly smaller than 2}. These results were previously known only for Bernoulli percolation ($q = 1$)
and the FK-Ising model ($q = 2$). 

Finally, we prove new bounds on the one, two and four arms exponents for $q\in[1,2]$, as well as the one-arm exponent in the half-plane. These improve the previously known bounds, even for Bernoulli percolation.
\end{abstract}

\section{Introduction}
 
\subsection{Motivation}

Understanding the behaviour of physical systems undergoing a continuous phase transition at their critical point is one of the major challenges of modern statistical physics, both on the physics and the mathematical sides. In this paper, we focus on percolation systems which provide models of random subgraphs of a given lattice. Bernoulli percolation is maybe the most studied such model, and breakthroughs in the understanding of its phase transition have often served as milestones in the  history of statistical physics.  The random-cluster model (also called Fortuin-Kasteleyn percolation), another example of a percolation model, was introduced by Fortuin and Kasteleyn around 1970 \cite{For70,ForKas72} as a generalisation of Bernoulli percolation. 
It was found to be related to many other models of statistical mechanics, including the Ising and Potts models, and to exhibit a very rich critical behaviour. 

The arrival of the renormalization group (RG) formalism (see \cite{Fis98} for a historical
exposition) led to a (non-rigorous) deep physical and geometrical understanding of continuous phase transitions. The RG formalism suggests that ``coarse-graining'' renormalization transformations correspond to appropriately changing the scale and the parameters of the model under study. The large scale limit of the critical regime then arises as the fixed
point of the renormalization transformations. 
A striking consequence of the RG formalism is that, the critical fixed point being usually unique, the scaling limit at the critical point must satisfy translation, rotation
and scale invariance. In seminal papers \cite{BPZ84b,BPZ84a}, Belavin, Polyakov and
Zamolodchikov went even further by suggesting a much stronger invariance of statistical physics models at criticality: since the scaling limit
quantum field theory is a local field, it should be invariant by any map which is
{\em locally} a composition of translation, rotation and homothety, which led them to postulate full conformal invariance. These papers gave birth to Conformal Field Theory, one of the most studied domains of modern physics. Of particular importance from the point of view of physics and for the relevance of our paper is the fact that the scaling limits of different random-cluster models at criticality are expected to be related to a range of 2D conformal field theories. Existence of a conformally invariant scaling limit was rigorously proved for the random-cluster model in two special cases corresponding to cluster-weight $q=2$ \cite{Smi10,CheDumHon12a,KemSmi16} and $q=1$ (in this case the proof only applies to a related model called site percolation on the triangular lattice \cite{Smi01}) only.

For percolation models in two dimensions, conformal invariance translates into predictions for so-called crossing probabilities of topological rectangles (also called quads): as the scale of the quad increases to infinity, the crossing probability should converge to a quantity that depends only on the extremal distance of the quad.   
In this paper, we show that crossing probabilities of quads are bounded in terms of the extremal distance of the quad only, thus hinting to their conformal invariance. In addition, we prove that pivotal points are abundant, a fact which is very useful in the study of the geometry of large clusters. While we are currently unable to show existence and conformal invariance of the scaling limit for general cluster weight $q \in [1,4]$, the properties derived in this paper should serve as stepping stones towards a better understanding of the critical phase, as was the case for $q = 1$ and $q = 2$. 

\subsection{Definition of the random-cluster model}\label{sec:1.2}

As mentioned in the previous section, the model of interest in this paper is the random-cluster model, which we now define. 
For background, we direct the reader to the monograph \cite{Gri06} and the lecture notes \cite{Dum17a}.

Consider the square lattice $(\mathbb Z^2,\mathbb E)$, that is the graph with vertex-set $\mathbb Z^2=\{(n,m):n,m\in\mathbb Z\}$ and edges between nearest neighbours. In a slight abuse of notation, we will write $\bbZ^2$ for the graph itself. In this paper we will mainly work with  the random-cluster model on discrete domains that are specific subgraphs of $\mathbb Z^2$ (together with a boundary), defined as follows.

Let $\gamma=(x_0,\ldots,x_{\ell-1})$ be a simple loop on $\mathbb Z^2$, \emph{i.e.}\@ $x_0,\ldots,x_{\ell-1}$ are distinct vertices, and $x_i$ is a neighbour of $x_{i+1}$ for $0\le i<\ell$ (where $x_\ell:=x_0$). Consider the (finite) set $E$ of edges enclosed by the loop (including the edges $\{x_i,x_{i+1}\}$ of $\gamma$) and define the graph $\mathcal D=(V,E)$ induced by $E$. Any graph obtained in this way is called a \emph{(discrete) domain}. Notice that one can always reconstruct the loop $\gamma$ (up to re-ordering of the vertices) from the data of the domain $\mathcal D$, and we define the boundary $\partial \mathcal D$ to be the set of vertices on the loop $\gamma$. We point out that the  boundary of a discrete domain differs from the more standard graph-theoretical notion of boundary: in particular, a vertex of $\partial \mathcal D$ may have all its neighbours in the domain (see Fig.~\ref{fig:domain0}).

\begin{figure}[htbp]
  \centering
  \includegraphics[width=4cm]{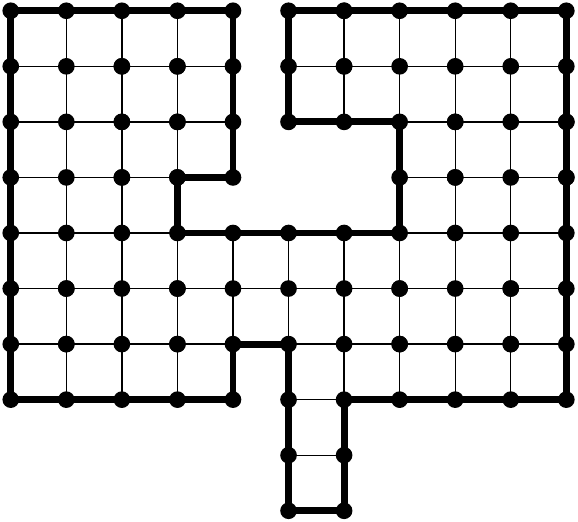}
  \caption{An example of a discrete domain $\mathcal D=(V,E)$. The loop $\gamma$ is represented by the bold line surrounding $\cal D$. The dots represent the elements of the vertex set $V$.}
  \label{fig:domain0}
\end{figure}

A percolation configuration $\omega=(\omega_e:e\in E)$ on a domain $\calD=(V,E)$ is an element of~$\{0,1\}^{E}$. An edge $e$ is said to be {\em open} (in $\omega$) if $\omega_e=1$, otherwise it is {\em closed}. A configuration  $\omega$ is identified to the subgraph of $\calD$ with vertex-set $V$ and edge-set $\{e\in E:\omega_e=1\}$. When speaking of connections in $\omega$, we view $\omega$ as a graph. For $A,B, S \subset V$, we say that $A$ is {\em connected to $B$ in $S$} if there exists a path in $\omega$ going from $A$ to $B$ and using vertices in $S$ only. We denote this event by $A\xleftrightarrow{S}   B$ (we omit $S$ when it is the full domain).
A {\em cluster} is a connected component of $\omega$. 

{\em Boundary conditions} on $\calD$ are given by a partition $\xi$ of $\partial \calD$. We say that two vertices of $\partial \calD$ are {\em wired together} if they belong to the same element of the partition $\xi$.

\begin{definition}  Let $\mathcal D=(V,E)$ be a discrete domain. The random-cluster measure on $\calD$ with edge-weight $p \in (0,1)$, cluster-weight $q>0$ and boundary conditions $\xi$ is given by
\begin{equation}\label{eq:RCM_def1}
	\phi_{\calD,p,q}^\xi[\omega]=\frac{1}{Z^\xi(\calD,p,q)} (\tfrac{p}{1-p})^{|\omega|}q^{k(\omega^\xi)} \qquad  \omega\in \{0,1\}^{E},
\end{equation}
where $|\omega|=\sum_{e\in E}\omega_e$, 
$\omega^\xi$ is the graph obtained from $\omega$ by identifying wired vertices of $\partial \calD$ together, 
$k(\omega^\xi)$ is the number of connected components of $\omega^\xi$, 
and  $Z^\xi(\calD,p,q)$ is a normalising constant called the {\em partition function} which is chosen in such a way that $\phi_{\calD,p,q}^\xi$ is a probability measure. 
\end{definition}

When $q = 1$, $\phi_{\calD,p,q}^\xi$ is a product measure, which is to say that the states of different edges are independent; this model is also called Bernoulli percolation with parameter $p$.
Two specific families of boundary conditions will be of special interest to us. On the one hand, the {\em free} boundary conditions, denoted 0, correspond to no wirings between boundary vertices. On the other hand, the {\em wired} boundary conditions, denoted 1, correspond to all boundary vertices being wired together.

For $p\in (0,1)$, $q\ge1$ and $i=0,1$, the family of measures $\phi_{\calD,p,q}^i$ converges weakly as $\calD$ tends to $\mathbb Z^2$. The limiting measure on $\{0,1\}^{\bbE}$ is denoted by $\phi_{\bbZ^2,p,q}^i$ and is called {\em infinite-volume} random-cluster measures with free or wired boundary conditions, when $i = 0$ or $i = 1$, respectively. 

The random-cluster model undergoes a phase transition at  a critical parameter $p_c=p_c(q)$  in the following sense: if $p>p_c(q)$, the $\phi_{\bbZ^2,p,q}^1$-probability $\theta(p,q)$
that 0 is connected to infinity
is strictly positive, while for $p<p_c(q)$, it is equal to 0.
In the past ten years, considerable progress has been made in the understanding of this phase transition: the critical point was proved in \cite{BefDum12} (see also \cite{DumRaoTas17,DumMan14,DCRT16}) to be equal to 
$$p_c(q)=\frac{\sqrt q}{1+\sqrt q}.$$
It was also proved in \cite{DumSidTas16,BetheAnsatz2,RaySpi19} that the phase transition is continuous (\emph{i.e.}~that $\theta(p_c,q)=0$) if $q\in[1,4]$, and discontinuous (\emph{i.e.}~that $\theta(p_c,q)>0$) for $q>4$. 
\bigbreak
{\em  As we are interested in continuous phase transitions only, in the whole paper we will fix  $q\in[1,4]$ and $p=p_c(q)$, and drop them from the notation. For this range of parameters, there is a unique infinite-volume random-cluster measure, so we omit the superscript corresponding to the boundary conditions and denote it simply  by $\phi_{\bbZ^2}$. }

\subsection{Crossing probabilities in quads}\label{sec:crossing_applications}

 A (discrete) {\em quad} $(\calD; a,b,c,d)$ is a discrete domain $\calD$ along with four vertices $a,b,c,d\in\partial \calD$ found on $\partial \calD$ in counterclockwise order. These vertices define four closed arcs $(ab)$, $(bc)$, $(cd)$, and $(da)$ corresponding to the parts of the boundary between them (here by \emph{closed} arc, we mean that the extremities $x$ and  $y$ belong to the arc $(xy)$).

 In order to define extremal distances associated to discrete quads, let us explain how the discrete domain $\calD$ can be seen as a continuous domain of the plane. First, consider the counter-clockwise loop $\gamma$ around $\mathcal D$ (up to cyclic permutation this loop is unique), and identify it  to a continuous piecewise linear curve in $\mathbb R^2$ by seeing all its edges as segments of length 1. Then, the  continuous domain associated to $\mathcal D$ is  obtained by taking the bounded connected component of $\mathbb R^2\setminus \gamma$.

 The \emph{extremal distance} $\ell_{\calD}\left[\left(ab\right),\left(cd\right)\right]$ between $\left(ab\right)$ and $\left(cd\right)$ inside $\calD$ is defined as the unique $\ell > 0$ such that there exists a conformal map from the continuous domain associated to $\mathcal D$  to the rectangle $(0,1) \times (0,\ell)$, with $a,b,c,d$ being mapped (by the continuous extension of the conformal map) to the corners of $[0,1] \times [0,\ell]$, in counterclockwise order, starting with the lower-left corner.

As mentioned in the previous section, conformal invariance of critical models exhibiting a continuous phase transition may be formulated using crossing probabilities of large quads. More precisely, $(\calD; a,b,c,d)$ is said to be {\em crossed} (from $(ab)$ to $(cd)$) in a configuration $\omega$, if it contains a path of open edges linking $(ab)$ to $(cd)$. It is expected that the probability that the blow up by $n$ of a given quad $(\calD,a,b,c,d)$ is crossed
converges as $n$ tends to infinity to a  non-degenerate limit that depends only on $\ell_{\calD}\left[\left(ab\right),\left(cd\right)\right]$.
 While we are currently unable to prove this result, we show that crossing probabilities remain bounded away from 0 and 1 uniformly in the extremal distance.

\begin{theorem}[Crossing estimates in general quads]\label{thm:sRSW}\label{thm:RSWquads}
Fix  $1\le q< 4$ and $p=p_c(q)$. For every $M>0$, there exists
$\eta=\eta(M)\in (0,1)$ such that for any discrete quad $\left(\calD,a,b,c,d\right)$ and any boundary conditions $\xi$,
\smallbreak
\begin{itemize}
\item if $\ell_{\calD}\left[\left(ab\right),\left(cd\right)\right] \leq M$\,, then $\mathbb{\phi}_{\calD}^{\xi}[\left(ab\right)\xleftrightarrow{\calD}\left(cd\right)] \geq \eta$;
\item if $\ell_{\calD}\left[\left(ab\right),\left(cd\right)\right] \geq M^{-1}$, then $\mathbb{\phi}_{\calD}^{\xi}[\left(ab\right)\xleftrightarrow{\calD}\left(cd\right)] \leq 1-\eta$.
\end{itemize}
\end{theorem}

Such crossing estimates are very useful for the study of the critical model. They initially emerged in the study of Bernoulli percolation in the late seventies under the coined name of Russo-Seymour-Welsh (RSW) theory \cite{Rus78,SeyWel78}. This theory has been instrumental in basically every result on critical Bernoulli percolation on the square lattice since then. Progress in the theory has been made in the past few years during which the RSW theorem was generalised to the random-cluster model first in \cite{BefDum12} for $q\ge1$ and specific boundary conditions, then in \cite{DumHonNol11} for $q=2$ and general boundary conditions, and finally in \cite{DumSidTas16} for $1\le q\le 4$ and arbitrary boundary conditions. 

One of the drawbacks of previous results is that the estimates were {\em restricted to rectangles} (the estimates are not expressed in terms of a conformally invariant measurement of size). This restriction is substantial in terms of applications due to the fact that boundary conditions do influence the configuration heavily in $\calD$, and that the roughness of the boundary could dictate the strength of this influence. For instance, it could a priori prevent the existence of open paths reaching the boundary of a domain, especially if the boundary is fractal (which will be the case if it is the boundary of another cluster).

In Theorem~\ref{thm:sRSW}, the crossing probability  bounds hold in \emph{arbitrary discrete quads with arbitrary boundary conditions}. In particular, they are independent
of the local geometry of the boundary. The  only other instance of such general estimates is the paper \cite{CheDumHon13} treating the specific case of $q=2$ in which much more is known thanks to discrete holomorphic observables. 

\subsection{Applications}

Theorem~\ref{thm:sRSW} has many implications for the study of the critical regime. We simply mention them briefly below, and refer to the corresponding sections for further details.
\begin{description}
\item[Tightness of interfaces:] It was recognised by Aizenman and Burchard \cite{AizBur99} that crossing estimates imply tightness when considering the scaling limit of interfaces (see Theorem~\ref{thm:tightness}). 
While tightness for random cluster interfaces was already proved \cite{KemSmi16,CheDumHon12a} using previously known crossing estimates, we would like to mention that the implication is quite straightforward when using Theorem~\ref{thm:sRSW}. 
\item[Non-simple curves in the scaling limit:] Theorem~\ref{thm:sRSW} implies that at large scales, macroscopic clusters typically touch each other and that their boundaries are non-simple (see Theorem~\ref{thm:self-touching}).
Let us mention that the family of interfaces describing boundaries of large cluster in the critical random-cluster model with cluster-weight $q\in(0,4]$ is conjectured  \cite{RohSch05} to converge to the Conformal Loop Ensemble (CLE) \cite{SW} with parameter 
$$\kappa=\kappa(q):=4\pi/\arccos(-\sqrt q/2).$$
Thus, our result rigorously excludes the possibility that the scaling limit of random-cluster models with $q \in [1,4)$ is described by a CLE with parameter $\kappa\le 4$ (as these are made of simple loops not touching each other).
\item[Quasi-multiplicativity, localization, well-separation for arm events:] While these properties were already obtained in specific cases ($q=1,2$, or general $q \in [1,4]$ but only for alternating arm-events with an even number of arms), we prove this statement for the first time in complete generality (see Propositions~\ref{prop:separation},~\ref{prop:quasimultiplicativity}, and~\ref{prop:localization}). 
\item[Universal arm exponents:] We obtain up-to-constant estimates for the probability of five alternating arms in the full plane, and two and three alternating arms in a half-plane (see Proposition~\ref{prop:universal}). 
It is noteworthy that these critical exponents do not vary for different random-cluster models despite the fact that these models belong to different universality classes.
\item[The four arm exponent is strictly smaller than 2:] 
We obtain a lower bound on the probability of four arms between scales (see Proposition~\ref{prop:four arm} for a precise statement).
This is a consequence of the value of the five arm exponent discussed above, and the strict monotonicity of arm exponents, which in turn follows from Theorem~\ref{thm:sRSW}.
Bounds on the four arm exponent have important consequences for the geometry of interfaces. 
In particular, they may be used to prove the existence of polynomially many pivotals.

When $1\le q\le 3$, we even prove a quantitative lower bound on the four arm exponent  (Proposition~\ref{prop:beta}) which is of interest when trying to prove the existence of exceptional times for the Glauber dynamics.  

\item[New bounds on the one-arm half-plane exponent:] We prove new bounds on the half plane one arm exponent when boundary conditions are free. More precisely, we show that when $q<2$ (resp.~$2<q\le 4$), this exponent is strictly smaller (resp.~larger) than $1/2$. This will be used in subsequent papers to study the effect of a defect line in the random-cluster model, and the order of the phase transition.

\item[The six-arm exponent is strictly larger than 2:] 
Another consequence of the universal value of the five arm exponent and of the strict monotonicity of arm exponents is 
an upper bound on the probability of having six alternating arms (Corollary~\ref{cor:six arm}). We mentioned it since it is very useful when studying percolation models at criticality, in particular when studying the Schramm-Smirnov topology \cite{SchSmi11} (see the detailed discussion in Section~\ref{sec:6}).

\item[New bounds for the one, two and four-arm exponents:] 
A byproduct of our proof is the following family of surprising bounds.
For $1\le q\le2$, the one-arm, two-arm and four-arm exponents can be rigorously bounded from above by $1/4$, $1/2$ and $3/2$, respectively, thus improving on the existing bounds, even in the case of Bernoulli percolation. For Bernoulli percolation, these bounds can be further improved to $1/6$, $1/3$, and $4/3$ (to be compared with the conjectured values $5/48$, $1/4$, and $5/4$). We refer to Section~\ref{sec:perco} for details.
\item[Scaling relations:] 
The existence of pivotals mentioned above is an important ingredient of the proof~\cite{DumMan20} of scaling relations connecting the different critical exponents of the random-cluster model.
\end{description}

\subsection{Idea of the proof of the main theorem}\label{sec:idea}

The starting point is the crossing estimates obtained for every $1\le q\le 4$ in \cite{DumSidTas16}. These estimates can be written under different forms. Here, we choose the following one. Write $\La_n$ for the domain  spanned by the vertex-set $\{-n,\dots, n\}^2$, and $\La_n(x)$ for its translate by $x \in \bbZ^2$. 
For a box $B:=\Lambda_r(x)$, let $\overline B:=\Lambda_{2r}(x)$ be the twice bigger box and $\Circ_B$ be the event that there exists a circuit in $\omega$ surrounding $B$ and contained in $\overline B$. The main theorem of \cite{DumSidTas16} (together with Proposition~5 in the same paper) implies the existence, for every $1\le q\le 4$, of $c_{\rm cir}>0$ such that for every domain $\calD$ and every $B$ with $\overline B\subset \calD$,
\begin{equation}\label{eq:RSW}
\phi_\calD^0[\Circ_B]\ge c_{\rm cir}.
\end{equation}
Note that the previous estimate is valid also for $q=4$ (unlike our main result), 
and that it does not require the existence of a macroscopic cluster touching the boundary. 
In fact, the main difficulty of our result consists in proving the existence of large clusters touching possibly fractal boundaries with free boundary conditions. Indeed, a statement similar to that of Theorem~\ref{thm:RSWquads} may be deduced directly from \cite{DumSidTas16} if the measure $\phi_\calD^\xi$ is replaced by the measure in a domain which is macroscopically larger than $\calD$ (see Section~\ref{sec:3.3}).

General considerations on the extremal distance together with~\eqref{eq:RSW} reduce Theorem~\ref{thm:sRSW}  to the following proposition, which will therefore be the focus of our attention. Call a domain $\calD$ $R$-{\em centred} if $\calD$ contains $\Lambda_{2R}$ but not $\Lambda_{3R}$.

\begin{proposition}\label{prop:crucial_exist}
	For $1\le q<4$, there exists $c_0>0$ such that for every $R\geq 1$ and any $R$-centred domain $\calD$, 	
	\begin{align}\label{eq:crucial_exist}
		\phi_\calD^0 [\Lambda_R\xleftrightarrow{\Lambda_{9R}}  \partial\calD]\ge  c_0.
	\end{align}
\end{proposition}

The proof of Proposition~\ref{prop:crucial_exist} is the core of the argument. Historically, results on crossing estimates are based on three different techniques. 
First, for $q=1$ or for general $q$ but specific boundary conditions, one may prove that crossing probabilities in squares are bounded away from $0$ using self-duality. Then, probabilistic arguments involving the FKG inequality enable one to extend these estimates to rectangles, see e.g.~\cite{Rus78,SeyWel78,BefDum12}. The use of self-duality relies on symmetries of the domain and of the boundary conditions, and is therefore inefficient for general quads or boundary conditions. 
A second technique based on renormalization arguments was implemented in \cite{DumSidTas16,DumTas18} for arbitrary boundary conditions but only for rectangles, and was used to prove~\eqref{eq:RSW}. While this technique treats arbitrary boundary conditions, it only applies when the boundary is at a macroscopic distance from the domain to be crossed.  
A third technique, which allows one  to prove that crossings touch the boundary, relies on the second moment method for the number of boundary vertices connected to a given set. This strategy works well when $\calD$ has a flat boundary, and was indeed used in \cite{DumSidTas16} to show crossing estimates for rectangles with free boundary conditions (see~\eqref{eq:wRSW} below), but does not extend to general quads. 
Indeed, except in the special case of the random-cluster model with $q=2$ \cite{DumHonNol11,CheDumHon13}, up-to-constant estimates on connection probabilities for vertices on the boundary are not available for general boundaries. Thus, the second moment method, as described above, becomes essentially impossible to implement.
 
The strategy used here to prove Proposition~\ref{prop:crucial_exist} will be  different than all of the above. 
It contains two different parts, and may be viewed as a combination of a first moment estimate and renormalization methods. Indeed, we start with a (sub-optimal) polynomial first moment estimate, then use a renormalization procedure to replace the second moment estimate and prove the existence of points on $\partial \calD$ connected to $\La_R$ with positive probability. 
 
For $r \geq 0$, call {\em $r$-box} any translate of $\Lambda_r$ by a vertex $x$ in $(1\vee r)\mathbb Z^2$. Notice that  a $0$-box is the same as a vertex of $\mathbb Z^d$. Consider $R\geq 1$ and a $R$-centred domain $\calD$, and let $\mathbf M_r(\calD,R)$ be the number of $r$-boxes intersecting $\partial \mathcal  D$ 
that are connected to $\Lambda_R$ in $\calD \cap \La_{7R}$ (the difference between the factors $7R$ here and $9R$ in Proposition~\ref{prop:crucial_exist} appears for technical reasons). 
In particular, $\mathbf M_0(\mathcal D,R)$ counts the number of vertices on $\partial \mathcal D$ that are connected to $\Lambda_R$ in $\mathcal D\cap \Lambda_{7R}$.  
Hence, our goal is akin to showing that there exists a uniform constant $c>0$ such that
\begin{equation}\label{eq:3}
  \phi_{\mathcal D}^0[\mathbf M_0(\mathcal D,R)\ge 1]\ge c
\end{equation}
for every $R \geq 1$ and every $R$-centred domain $\mathcal D$.  

As already mentioned, the first step towards Proposition~\ref{prop:crucial_exist} 
is to lower-bound the first moment of $\mathbf M_r(\calD,R)$, which is the object of the following proposition. Introduce 
\begin{equation}
	M(r,R):=\inf\{\phi_{\calD}^0[\mathbf M_{r}(\calD,R)]:\calD\text{ $R$-centred}\}.
\end{equation}

\begin{proposition}[non-sharp scale-to-scale lower bound on first moment]\label{prop:fund}
  For $1\le q< 4$, there exists $c_1>0$ such that for every $R\ge r\ge1$,
  \begin{align}\label{eq:fund1}
    M(r,R)\ge c_1(R/r)^{c_1}.
  \end{align}
\end{proposition}
Let us make some remarks about this result. First, we would like to emphasise that it is non-trivial, in the sense that it does not follow directly from the RSW estimates. In order to put the proposition above into perspective, let us give the estimate that we obtain if one uses only the RSW result~\eqref{eq:RSW} to estimate $\mathbf M_r(\mathcal D,R)$. By a standard scale-to-scale gluing procedure, the RSW result~\eqref{eq:RSW} gives a lower bound of $(r/R)^C$ on the probability that a $r$-box intersecting the boundary of $\mathcal D$ is connected to $\Lambda_R$, where $C>0$ is a positive constant on which we have almost no control (it is a priori very large). Since the total number of $r$-boxes intersecting the boundary of $\mathcal D$ is of order $(R/r)^d$ for some $d\in [1,2]$, we would obtain an estimate of the form
\begin{equation}
  \label{eq:4}
  \phi_{\calD}^0[\mathbf M_{r}(\calD,R)]\gtrsim \left(\frac R r\right)^{d-C}.
\end{equation}
This lower bound does not  establish the proposition above, due to the lack of control on~$C$. Another way to see that there is something subtle in the proposition above is explained in the next section: we expect that the estimate~\eqref{eq:fund1} does not hold for $q=4$ (even if the RSW-result ~\eqref{eq:RSW} does).

A second remark is that this estimate is non-sharp, in the following sense. 
For a fixed fractal domain $\mathcal D$, the expectation $\phi_{\calD}^0[\mathbf M_{r}(\calD,R)]$ is thought to behave like $(R/r)^{\eta+o(1)}$ for some $\eta>0$, but a priori the constant $c_1$ in \eqref{eq:fund1} is smaller than $\eta$. In particular, the second moment method cannot be used to deduce~\eqref{eq:3} from the first moment estimate. Instead, we use a  new renormalization technique, inspired from the theory of branching processes, and involving the following quantity:
\begin{equation}\label{eq:p(R)}
p(R):=\inf\{\phi_\calD^0 [\Lambda_R\xleftrightarrow{\Lambda_{9R}}  \partial\calD]: \calD\text{ $R$-centred}\}.
\end{equation}

\begin{proposition}[renormalization]\label{prop:renormalization}
For $1\le q\le 4$, there exists $c_2>0$ such that for every $R/20\ge r\ge 1$,
\begin{equation}\label{eq:recursive}
p(R)\ge c_2M(r,R)\min\{p(r),(\tfrac rR)^2\}.
\end{equation}
\end{proposition}

Once we have established the Propositions~\ref{prop:fund} and~\ref{prop:renormalization}, one can easily conclude the proof of Proposition~\ref{prop:crucial_exist} as follows: 

\begin{proof}[Proposition~\ref{prop:crucial_exist}] 
	Choose a constant $\lambda \geq 20$ large enough that $c_1c_2\lambda^{c_1}\geq 1$. 
	Then~\eqref{eq:fund1} and~\eqref{eq:recursive} applied with $\lambda r\le R \le \lambda^2 r$ imply  
	$$p(R) \geq \min\{\lambda^{-4}, p(r)\}.$$
	Consider the sequence $u_n:=\min\{p(R):\lambda^n\le R< \lambda^{n+1}\}$. 
	By applying the equation above to $r=\lambda^n$, we have  
	$$u_{n+1} \geq \min\{\lambda^{-4}, p(\lambda^n)\}\ge \min\{\lambda^{-4}, u_n\},$$ 
	which implies  that $u_n\geq  \min(\lambda^{-4},u_0)$ for every $n\ge 0$ by induction. 
Since $u_0>0$ (by the finite energy property), we conclude that $\inf \{p_R: R\geq 1\} > 0,$ which corresponds to the statement of the proposition. 
 \end{proof}

\begin{remark}
We will see in Proposition~\ref{prop:polynomially many} that we can prove an even stronger result, namely that with probability bounded by a universal strictly positive constant, in every $R$-centred domain $\calD$, $\Lambda_R$ is connected to polynomially many points of $\partial\calD$.
\end{remark}

\subsection{Origin of first moment bound and why is $q=4$ excluded}\label{sec:q=4}

The careful reader will have noticed that~\eqref{eq:RSW} and Proposition~\ref{prop:renormalization} are valid for every $1\le q\le 4$, while Proposition~\ref{prop:fund} and Theorem~\ref{thm:sRSW} require $q<4$ additionally. 
At this stage, it is useful to explain why $q=4$ is excluded from the latter two statements. 

Write $\phi_{\mathbb H}^0$ for the infinite-volume measure\footnote{The measure is obtained as the weak limit (as $R\rightarrow\infty$) of the measures $\phi_{[-R,R] \times [0,R]}^0$.} in the half-plane $\mathbb H:=\mathbb Z\times\mathbb Z_+$
and set 
\begin{equation}
 \pi_1^+(R):=\phi_{\bbH}^0[0\longleftrightarrow\partial\Lambda_R].
\end{equation}
The proof of Proposition~\ref{prop:fund} will crucially rely on the following lemma. 

\begin{lemma}\label{lem:mn}
    For $1\le q<4$, there exists $c_3>0$ such that for every $R\ge r\ge1$, 
    \begin{align*}
    	\pi_1^+(R)\ge c_3(r/R)^{1-c_3}\pi_1^+(r).
	\end{align*}
\end{lemma}

This lemma is a simple application (see Section~\ref{sec:5.1a}) of the following result from \cite{DumSidTas16}: 
For every $1\le q<4$ and $\rho>0$, there exists $c_{\rm bound}=c_{\rm bound}(\rho)>0$ such that 
\begin{equation}\label{eq:wRSW}
	\phi_{\calR}^0[(ab)\longleftrightarrow(cd)]>c_{\rm bound}
\end{equation}
for every rectangle $\calR:=[0,\rho R]\times[0,R]$ of aspect ratio $\rho$, where $a,b,c,d$ are the four corners of the rectangle, indexed in counter-clockwise order, starting from the bottom-right corner.
Since Lemma~\ref{lem:mn} is our starting point, we can summarise the innovation in this paper as follows:
we start from crossing estimates for domains with flat boundaries and extend these estimates to fractal domains. 

When $q=4$, we expect the scaling limit of the critical random-cluster model to be described by CLE(4). 
The probability that a macroscopic loop of CLE(4) comes within distance $\ep$ of a point on a flat boundary is of order $\ep$. 
We therefore expect that $\pi_1^+(R)$ decays like $1/R$ when $q=4$, which contradicts the conclusion of Lemma~\ref{lem:mn}.
As a consequence, the probabilities of crossing rectangles in~\eqref{eq:wRSW} are expected to tend to $0$ as $N$ increases. 
Thus, Theorem~\ref{thm:sRSW} should be {\em wrong} for $q=4$, even in the special case of flat boundaries. 
We take this opportunity to state the following question.
\begin{question}
Show that for $q=4$, $\pi_1^+(R)$ decays (up to multiplicative constants) like $1/R$, and
 probabilities  in~\eqref{eq:wRSW}  tend to 0 as $R$ tends to infinity.
\end{question}

To conclude, let us mention that Proposition~\ref{prop:fund} will be a direct consequence of the previous lemma together with the following result.

\begin{proposition}\label{prop:10}
	For $1\le q< 4$, there exists $c_4>0$ such that for every $R\ge r\ge0$,
	\begin{align}\label{eq:fund}
		M(r,R) \ge c_4\frac{R\pi_1^+(R)}{1\vee r\pi_1^+(r)}.
	\end{align}
\end{proposition}
The special form of the denominator is meant to accommodate the case $r = 0$. 
This case, albeit not important for the application of the proposition (indeed Proposition ~\ref{prop:fund} only uses $r \geq 1$), 
will serve as a stepping stone in the proof of~\eqref{eq:fund}.

The proof of this proposition uses parafermionic observables, as did that of \eqref{eq:wRSW} in \cite{DumSidTas16}. While these observables were previously used to study the critical phase of several 2D models \cite{Smi10,DumSidTas16,BefDumSmi15,DumSmi12,DumGla18}, the present use is new, and we believe that the amount of information extracted from these observables is superior to previous results dealing with general values of $q$; of course when $q=2$ much more is known due to further properties of parafermionic observables that are specific to this cluster-weight.

To conclude this section, let us show how to deduce Proposition~\ref{prop:fund} from Lemma~\ref{lem:mn} and Proposition~\ref{prop:10}.
\begin{proof}[Proposition~\ref{prop:fund}]
	Insert the bound of Lemma~\ref{lem:mn} into~\eqref{eq:fund} to obtain the desired result. 	
\end{proof}

\subsection{Organisation of the paper}

Section~\ref{sec:background} recalls some basics of the random-cluster model. 
There are three steps in the proof of Theorem~\ref{thm:RSWquads}:
\begin{itemize}
\item Proving the statements related to the non-sharp first moment estimate, namely Lemma~\ref{lem:mn} and Proposition~\ref{prop:10};
they are postponed to Sections~\ref{sec:3} and ~\ref{sec:4.1}. 
Proposition~\ref{prop:fund} was already shown to follow from these two results. 
\item Proving the renormalization procedure of Proposition~\ref{prop:renormalization}; this is done in Section~\ref{sec:2}.
\item Showing how Proposition~\ref{prop:crucial_exist} implies Theorem~\ref{thm:RSWquads}; this is done in Section~\ref{sec:3.3}.
Indeed, Proposition~\ref{prop:crucial_exist} was already shown to follow from Proposition~\ref{prop:fund} and Proposition~\ref{prop:renormalization} in Section~\ref{sec:idea}.
\end{itemize} 
Consequences of Theorem~\ref{thm:sRSW} for probabilities of arm events and properties of scaling limits are given in Sections~\ref{sec:4} (except Section~\ref{sec:4.1}) and~\ref{sec:6}, respectively. These are not necessary for the proof of Theorem~\ref{thm:sRSW}.

\paragraph{Convention regarding constants}
In this paper, $(c_i)_{i\geq0}$ denote constants specific to the statements in which they appear, and are fixed throughout the paper. 
The constants $c,c',c''$ and $C,C',C''$ denote small and large quantities, respectively, whose enumeration is restarted in each proof.

\subsection{Acknowledgments} 
The first author is supported by the ERC CriBLaM, the NCCR SwissMAP, the Swiss NSF and an IDEX Chair from Paris-Saclay. The second author is supported by the NCCR SwissMAP and the Swiss NSF. The third author is supported by the ERC grant CRISP and NCCR SwissMAP. We thank Alex Karrila for pointing out Fact~\ref{fact} to us and helping with Section~\ref{sec:3.3}.
\section{Background}\label{sec:background}

We will use standard properties of the random-cluster model. They can be found in \cite{Gri06}, and we only recall them briefly below. 
\medbreak\noindent 1. {\em FKG inequality}: Fix $q\ge 1$ and a domain $\mathcal D=(V,E)$ of $\bbZ^2$. 
An event $A$ is called {\em increasing} if for any $\omega\le\omega'$ (for the partial order on $\{0,1\}^E$), $\omega\in A$  implies that $\omega'\in A$.
For every  increasing events $A$ and $B$,
\begin{align}\label{eq:FKG} 
	\phi_{\calD}^\xi[A\cap B]&\ge \phi_{\calD}^\xi[A]\phi_{\calD}^{\xi}[B].\end{align}
\medbreak\noindent 2. {\em Comparison between boundary conditions}: For every increasing event $A$ and every $\xi'\ge\xi$, where $\xi'\ge\xi$ means that the wired vertices in $\xi$ are also wired in $\xi'$,
	\begin{align} 	\label{eq:CBC} 
	\phi_{\calD}^{\xi'}[A]&\ge \phi_{\calD}^\xi[A].
\end{align}
\medbreak\noindent 3. {\em Spatial Markov property}: for any configuration $\omega'\in\{0,1\}^E$ and any subdomain $\mathcal F=(W,F)$ with  $F\subset E$,
\begin{equation}\label{eq:SMP} 	\phi_{\calD}^\xi[\cdot_{|F}\,|\,\omega_e=\omega'_e,\forall e\notin F]\ge \phi_{\calF}^{\xi'}[\cdot],
\end{equation} where the boundary conditions $\xi'$ on $\calF$ are  defined as follows: 
$x$ and $y$ on $\partial \calF$ are wired if they are connected in $\omega_{|E\setminus F}^\xi$. 
\medbreak\noindent 4. {\em Mixing property}:
    There exists $c_{\rm mix}>0$ such that for every $R\ge1$, every $\calD\supset\Lambda_{R}$, every boundary condition $\xi$ on $\calD$ 
    and every event $A$ depending on edges in $\Lambda_{R/2}$, we have that
    \begin{align}\label{eq:mix2}
    	c_{\rm mix}\,\phi_{\calD}^0[A]\le \phi_{\calD}^\xi [A]
    	\le c_{\rm mix}^{-1}\,\phi_{\calD}^0 [A].
    \end{align}
This property is not trivial and can be obtained using~\eqref{eq:RSW}, see e.g.~\cite{Dum13}.

\section{The renormalization step: proof of Proposition~\ref{prop:renormalization}}\label{sec:2}

In this section, we fix $R/20 \ge r\ge1$. Recall that {\em $r$-boxes} are translates of $\Lambda_r$ by vertices $x\in r\mathbb Z^2$. It is worth keeping in mind that $r$-boxes, having side length $2r$, overlap.

For a $R$-centred domain $\calD$, introduce the subdomain  $\calD_r\subset \calD$ obtained as the connected component of the origin in the union of the $r$-boxes included in $\calD$ and at $L^\infty$-distance at least $10r$ of $\partial\calD$. See Fig.~\ref{fig:2} for an illustration. Notice that the condition $R \geq 20r$ ensures that $\La_R$ is contained in $\calD_r$. 

\begin{figure}[htbp]
  \centering
  \includegraphics[width=.5\textwidth]{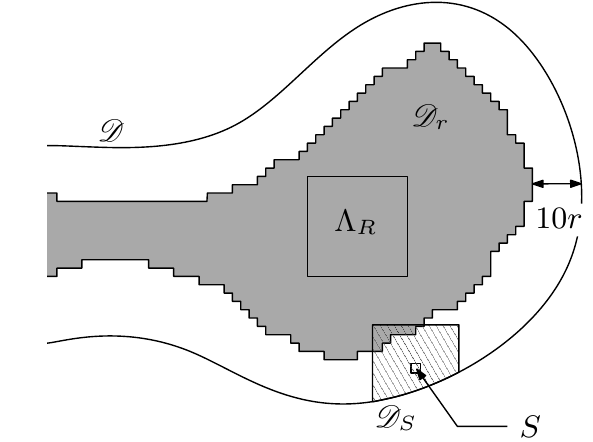}
  \caption{An illustration of the domain $\mathcal D_r$, a seed $S$, and its associated domain $\mathcal D_S$.}
  \label{fig:2}
\end{figure}

A {\em $r$-seed} of $\calD$ is a $r$-box  $S = \Lambda_r(x)$ such that $\Lambda_{2r}(x) \subset \calD$ but $\Lambda_{3r}(x) \nsubset \calD$; 
in other words, such that the translate of $\calD$ by $-x$ is $r$-centred (see Fig.~\ref{fig:2} for an example).
Let 
\begin{align*}
\calD_S&:=\Lambda_{20r}(x)\cap \calD
\end{align*}
and say that $S$ is {\em $c_\square$-activated for} a configuration $\xi$ in $\calD_r$
 if
\begin{equation}\label{eq:def square}
	\phi_{\calD_r\cup\calD_S}^0[S\xleftrightarrow{\Lambda_{7R}\cup \mathcal D_S} \Lambda_R|\,\omega_{|\calD_r}=\xi]\ge c_{\square},
\end{equation}
where $c_\square>0$ is a constant that will be selected properly in the next lemma. 
One may observe that the small domain $\calD_S$ around the seed does not necessarily intersect the domain $\calD_r$ or the box $\Lambda_{7R}$; in such cases the left-hand side of the equation above is always equal to 0, and the seed $S$ is never $c_\square$-activated.

Let $\mathbf N_r(\calD,R,c_\square)$ be the number of $r$-seeds of $\calD$ that are $c_\square$-activated for the configuration in~$\calD_r$. We emphasise that $\mathbf N_r(\calD,R,c_\square)$ is measurable with respect to the configuration restricted to $\mathcal D_r$.

Even though $\mathbf M_r(\calD,R)$ is defined in terms of boundary $r$-boxes connected to $\Lambda_R$, while $\mathbf N_r(\calD,R,c_\square)$ is defined in terms of $r$-seeds that are $c_\square$-activated (and therefore not really connected to $\Lambda_R$), 
one should consider these two quantities comparable. The following lemma provides a bound between the expectation of $\mathbf N_r(\calD,R,c_\square)$ and that of $\mathbf M_r(\calD,R)$.

\begin{lemma}\label{lem:10}
	There exist $c_5,c_\square>0$ such that for every $1\le r\le R/20$ and every $R$-centred domain $\calD$, 
		\begin{align}\label{eq:crucial_exp}
	\phi_{\calD}^0[\mathbf N_{r}(\calD,R,c_\square)]
	\ge c_5\phi_{\calD_r}^0[\mathbf M_r(\calD_r,R)].
	\end{align}
\end{lemma}

\begin{proof}
By definition,~\eqref{eq:crucial_exp} can be rewritten as
  \begin{equation}
    \label{eq:1}
 \sum_{S\text{ $r$-seed}}\phi_{\calD}^0[S\text{ is $c_\square$-activated}]\ge c_5 \sum_{\substack{B\text{ $r$-box}:\\B\cap \partial \mathcal D_r\neq\emptyset}}\phi_{\calD_r}^0[B\xleftrightarrow{\Lambda_{7R}} \Lambda_R].
 \end{equation}
 In order to prove this equation, we fix a $r$-box $B=\La_r(x)$ intersecting $\partial \calD_r$ and consider a $r$-seed $S=S(B)\subset \Lambda_{15r}(x)$; such a seed exists since $B$ is at distance between $8r$ and $12r$ from $\partial\calD$.
 For each such pair $(B,S)$ we will prove  that the inequality 
 \begin{equation}
   \label{eq:5}
   \phi_{\calD}^0[S\text{ is $c_\square$-activated}]\ge c_\square\phi_{\calD_r}^0[B\xleftrightarrow{\Lambda_{7R}} \Lambda_R]
 \end{equation}
 holds for a suitable choice of the constant $c_\square$. 
 By summing this equation over all $r$-boxes intersecting $\partial \calD_r$, 
 and using that the number of boxes $B$ corresponding to any given seed $S$ is bounded by a constant $C$, 
 this concludes the proof with $c_5=c_\square/C$. 

We now prove~\eqref{eq:5}. First, by comparison between boundary conditions~\eqref{eq:CBC} together with the fact that being $c_\square$-activated is an increasing event, we have
 \begin{align}\label{eq:6}
   \phi_{\calD}^0[S\text{ $c_\square$-activated}]&\ge \phi_{\calD_r\cup\calD_S}^0[S\text{ $c_\square$-activated}]\\
   &\ge\phi_{\calD_r\cup\calD_S}^0[S\text{ $c_\square$-activated} \:|\:  B\xleftrightarrow{\calD_r\cap\Lambda_{7R}} \Lambda_R] \phi_{\calD_r\cup\calD_S}^0[B\xleftrightarrow{\calD_r\cap\Lambda_{7R}} \Lambda_R].\notag
\end{align}

Define the random variable 
$X(\omega)=\phi_{\calD_r\cup\calD_S}^0[S\xleftrightarrow{\Lambda_{7R}\cup \mathcal D_S} \Lambda_R|\,\omega_{|\calD_r}]$,
and observe that $S$ is $c_\square$-activated if and only if $X(\omega) \geq c_\square$.
Apply the inequality $\mathbf1_{X\ge c_\square}\ge X-c_\square$ to deduce that
\begin{align}\label{eq:9}
  \phi_{\calD_r\cup\calD_S}^0[S\text{ $c_\square$-activated}\: | \: B\xleftrightarrow{\calD_r\cap\Lambda_{7R}} \Lambda_R]
 &\ge\phi_{\calD_r\cup\calD_S}^0[S \xleftrightarrow{\Lambda_{7R}\cup \calD_S}\Lambda_R\: | \: B\xleftrightarrow{\calD_r\cap\Lambda_{7R}} \Lambda_R]-c_\square.   
\end{align}
For the above, it is essential that $X$ is measurable in terms of the configuration in $\calD_r$.

Finally, we can use the RSW-estimate~\eqref{eq:RSW} to bound the first term in the lower bound above as follows.  Write $H$ for the event that all $r$-boxes $B_0$ with $\overline B_0\subset\calD_S$ satisfy $\Circ_{B_0}$.
Observe that, if $B$ is connected to $\Lambda_R$ inside $\calD_r$ and if $H$ occurs, then $S$ is connected to $\Lambda_R$ inside $(\calD_r \cap\Lambda_{7R})\cup \calD_S$ (see Fig.~\ref{fig:domains}).

\begin{figure}
\begin{center}
\includegraphics[width=0.8\textwidth]{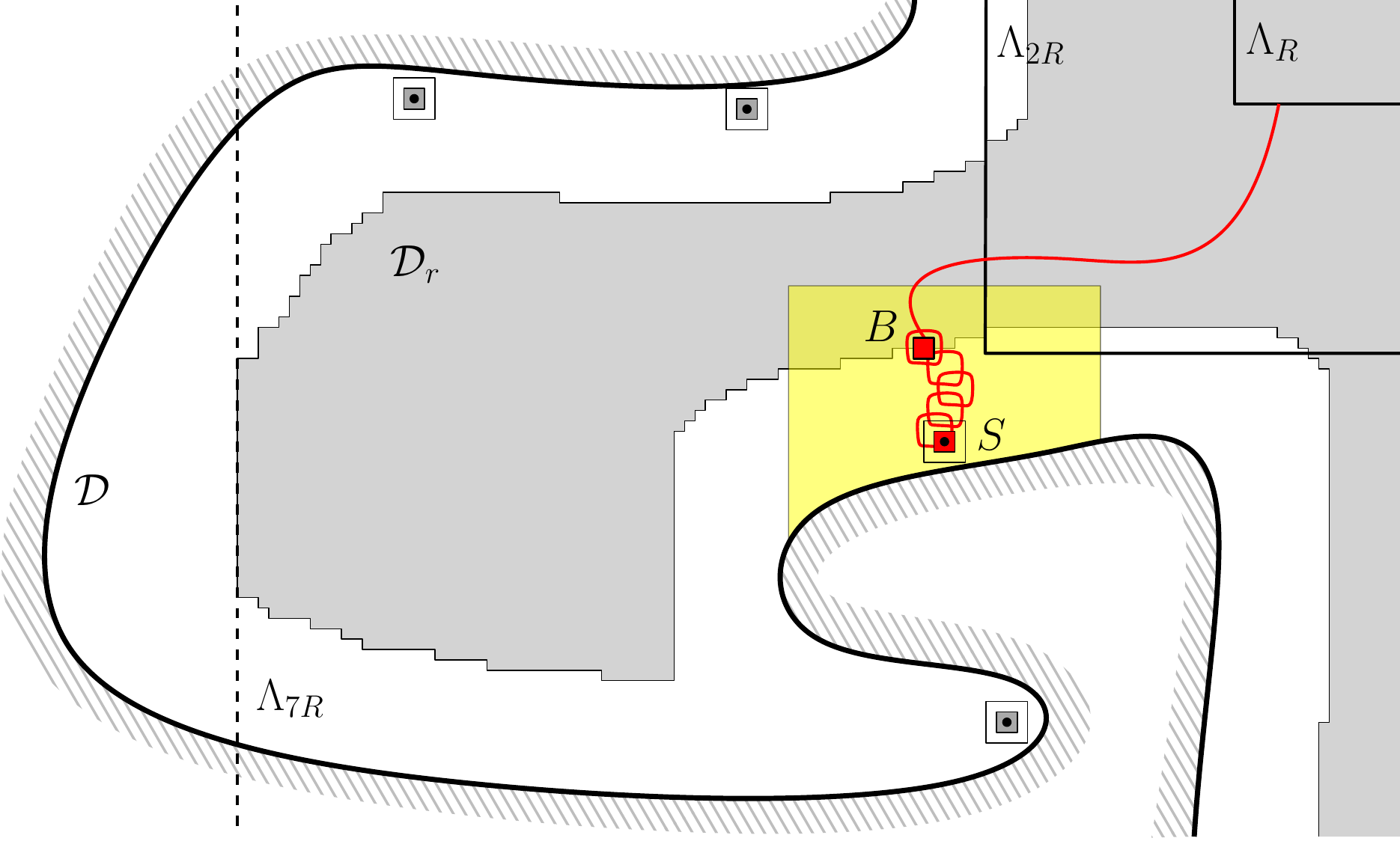}
\caption{The domain $\calD_r$ associated to $\calD$ is grey; several seeds are depicted. Notice that some seeds cannot be activated (for instance the lowest one). The box $B$ and seed $S$ are marked in red, the domain $\calD_S$ in yellow. When $H$ occurs and $B$ is connected to $\Lambda_R$ inside $\calD_r$, then $S$ is connected to $\Lambda_R$ inside $\calD_r\cup \calD_S$.}
\label{fig:domains}
\end{center}
\end{figure}

We deduce from~\eqref{eq:RSW}, the FKG inequality~\eqref{eq:FKG} and the comparison between boundary conditions~\eqref{eq:CBC} that 
$H$ occurs with probability bounded below by some constant $c > 0$. Thus
\begin{align}\label{eq:8}
 &\phi_{\calD_r\cup\calD_S}^0[S \xleftrightarrow{\Lambda_{7R}\cup\mathcal D_S}\Lambda_R\: | \: B\xleftrightarrow{\calD_r\cap\Lambda_{7R}} \Lambda_R] \ge\phi_{\calD_r\cup\calD_S}^0[H\: | \: B\xleftrightarrow{\calD_r\cap\Lambda_{7R}} \Lambda_R]\ge \phi_{\calD_r\cup\calD_{S}}^0[H]
	\ge c.
\end{align}

Equations~\eqref{eq:6},~\eqref{eq:9},~\eqref{eq:8} and the comparison between boundary conditions imply 
\begin{align*}
	\phi_{\calD}^0[S\text{ $c_\square$-activated}]
	\ge(c-c_\square)\phi_{\calD_r\cup\calD_{S}}^0[B\xleftrightarrow{\calD_r}   \Lambda_R]
	\ge(c-c_\square)\phi_{\calD_r}^0[B\xleftrightarrow{\Lambda_{7R}} \Lambda_R],
\end{align*}
which concludes the proof if we choose $c_\square=c/2$.
\end{proof}

We are ready to prove Proposition~\ref{prop:renormalization}.
\begin{proof}[Proposition~\ref{prop:renormalization}]
By Lemma~\ref{lem:10}, it suffices to prove the existence of $c > 0$ 
such that for every $1\le r\le R/20$ and every $R$-centred domain $\calD$,
\begin{equation}\label{eq:ai1}
	\phi_\calD^0 [\Lambda_R\xleftrightarrow{\Lambda_{9R}} \partial\calD]
	\ge c \phi_\calD^0[\mathbf N_r(\calD,R,c_\square)]\min\{p(r),(\tfrac rR)^2\}.
\end{equation}
Fix $r$, $R$ and $\calD$ as above. For each seed $S=\Lambda_r(x)$, introduce the events
\begin{align*}
E_S&:=\{S\xleftrightarrow{(\calD_r\cap\Lambda_{7R})\cup\calD_S}\Lambda_R\},\\
F_S&:=\{S\xleftrightarrow{\Lambda_{9r}(x)}\partial\calD\}. \end{align*}
Fix a configuration $\xi$ in $\calD_r$; it contains $\mathbf N_r(\xi)$ $c_\square$-activated seeds. Among them, we can select a subset  $\mathbf A(\xi)$ of at least $\frac 1 C \mathbf N_r(\xi)$  $c_\square$-activated seeds with disjoint corresponding domains $\mathcal D_S$, where $C$ is an absolute constant which bounds from above  the numbers of domains $\mathcal D_{S'}$  intersecting a fixed domain $\mathcal D_S$ (say $C=100^2$).

The FKG inequality~\eqref{eq:FKG} and the comparison between boundary conditions~\eqref{eq:CBC} give that for every $S\in \mathbf A(\xi)$,
\begin{align}
\phi_{\calD_r \cup \calD_S}^0[\Circ_S\cap E_S\cap F_S|\, \omega_{|\calD_r}=\xi]
& \ge \phi_{\calD_S\setminus \calD_r}^0[\Circ_S]\, \phi_{\calD_r\cup\calD_S}^0[E_S|\,\omega_{|\calD_r}=\xi]\, \phi_{\calD_S\setminus \calD_r}^0[F_S] \nonumber\\
& \ge c_{\rm cir}\,c_\square\, p(r),
\label{eq:renorm11}
\end{align}
where in the last inequality we used~\eqref{eq:RSW}, the definition of $c_\square$-activation, and the definition of $p(r)$.

If $\Circ_S\cap E_S\cap F_S$ occurs for some $S\in\mathbf A(\xi)$, then $\Lambda_R$  is connected in $\Lambda_{9R}$  to $\partial\calD$ (we use that the respective supports $\Lambda_{2r}(x)$, $(\calD_r\cap\Lambda_{7R})\cup\calD_S$, and $\Lambda_{9r}(x)$ of the three events are all subsets of  $\Lambda_{9R}$, thanks to the condition $R\ge 20r$). 
Since the seeds in $\mathbf A(\xi)$ have disjoint domains $\calD_S$, and using the comparison between boundary conditions, 
\eqref{eq:renorm11} implies that, under $\phi_{\calD}^0[.|\, \omega_{|\calD_r}=\xi]$,
the probability that $\Lambda_{R}$ is connected in $\Lambda_{9R}$ to $\partial\calD$ is larger than 
the probability that a binomial random variable with parameters $\frac1C\mathbf N_r(\xi)$ and $c_{\rm cir}c_\square p(r)$ is strictly positive.
Averaging on $\xi$ gives
\begin{align*}
\phi_\calD^0 [\Lambda_R\xleftrightarrow{\Lambda_{9R}}   \partial\calD]&\ge \phi_\calD^0[1-(1-c_{\rm cir}c_\square p(r))^{\mathbf N_r(\xi)/C}]\\
&\ge \phi_\calD^0[1-(1-c\min\{p(r),(\tfrac rR)^2\})^{\mathbf N_r(\xi)/C}]\\
&\ge \tfrac{c}{e}\phi_\calD^0[\mathbf N_r(\calD,R,c_\square)]\min\{p(r),(\tfrac rR)^2\} ,
\end{align*}
where in the second inequality we used that $x\mapsto 1-(1-x)^n$ is increasing in $x$, and in the third that $c\in(0,c_{\rm cir}c_\square)$ is chosen small enough that $c(r/R)^2\mathbf N_r(\xi)/C \leq 1$ for {\em every} realization of $\xi$, so that we can use that $1-(1-x)^n\ge nx/e$.

In conclusion,~\eqref{eq:ai1} is proved. 
\end{proof}

\section{Crossings in general quads: from Proposition~\ref{prop:crucial_exist} to Theorem~\ref{thm:sRSW}}\label{sec:3.3}

In order to prove Theorem~\ref{thm:sRSW} it suffices to show the lower bound (\emph{i.e.}~the first item) for free boundary conditions. Indeed, the lower bound for arbitrary boundary conditions then follows from the comparison between boundary conditions~\eqref{eq:CBC}. The upper bound (\emph{i.e.}~the second item) may be deduced from the lower bound by duality (see \cite{Gri06} for background on duality for the random cluster model) and the fact that $\ell_\calD[(ab),(cd)]=1/\ell_\calD[(bc),(da)]$. The rest of the section is therefore dedicated to showing the lower bound for free boundary conditions. 
The challenge here is to translate the estimates of Proposition~\ref{prop:crucial_exist} to treat crossing probabilities in general quads. 
We divide the proof in two: we first treat quads with small extremal distance, and then we generalise to quads with arbitrary extremal distance. 

\paragraph{Quads with small extremal distance} 
Let us show that there exist constants $\eta,m >0$ such that for any discrete quad $(\calD,a,b,c,d)$ 
with $\ell_{\calD}[(ab),(cd)] \leq m$, 
\begin{align*}
	\phi_{\calD}^0[(ab)\longleftrightarrow(cd)] \geq \eta.
\end{align*}

The proof will be based on the following fact shown in \cite{KemSmi12} within the proof of the implication ${\bf G2}\Rightarrow{\bf C2}$ of Proposition~2.6.
\begin{fact}\label{fact}
	There exists a constant $m>0$ such that for every quad $(\calD,a,b,c,d)$ 
	with $\ell_{\calD}[(ab),(cd)] \leq m $, 
	there exist $x\in\bbR^2$ and $R >0$ such that any crossing $\gamma$ from $(bc)$ to $(da)$ in $\calD$ contains a sub-path that 
	connects $\La_R(x)$ to $\partial\La_{2R}(x)$.
\end{fact}
From now on, fix a quad $(\calD,a,b,c,d)$ with $\ell_{\calD}[(ab),(cd)] \leq m$ and let $x$ and $R$ be given by the previous fact, with $R$ minimal for this property. To simplify notation, let us translate $\calD$ so that $x = 0$. Observe that the minimal graph distance between $(ab)$ and $(cd)$ in $\calD\cap(\Lambda_{5R/3}\setminus\Lambda_{4R/3})$ is larger than $R/3$, since otherwise one may find $x'$ and $R'<R$ satisfying the assumptions of Fact~\ref{fact}, which would contradict the minimality of $R$. 

Below we use the notation $\Circ_B$ and $F_S$ of Section~\ref{sec:2}. Set $r:=\lfloor R/60\rfloor$ and
 consider the event $H$ that 
\begin{itemize}[nolistsep]
\item $\Circ_B$ occurs for every $r$-box $B\subset\Lambda_{2R}$ with $\overline B\subset\calD$,
\item $F_S$ occurs for every $r$-seed $S\subset\Lambda_{2R}$ of $\calD$.
\end{itemize}
Then, if $H$ occurs, we claim that $(ab)$ is connected to $(cd)$ inside $\calD$. 
Indeed, by the choices of $R$ and $r$, there exists a seed $S = \La_r(x)$ 
with $\La_{9r}(x)\subset \Lambda_{5R/3}\setminus\Lambda_{4R/3}$
and such that $\La_{9r}(x)$ intersects the arc $(ab)$, but not any other part of $\partial \calD$. 
The same holds for a seed $S'$, with the arc $(ab)$ replaced by $(cd)$. 
When $H$ occurs, there exist open circuits contained in $\calD$ surrounding each of these two seeds, and connected to each other inside $\calD$. 
Moreover, since $F_{S}$ and $F_{S'}$ occur, the circuits above are connected to $(ab)$ and $(cd)$, respectively. 
See Fig.~\ref{fig:quad_annulus} for an illustration.

The FKG inequality~\eqref{eq:FKG} together with~\eqref{eq:RSW} and Proposition~\ref{prop:crucial_exist} imply that 
\begin{align*}
	\phi_{\calD}^0[(ab)\longleftrightarrow(cd)]\ge\phi_{\calD}^0[H]\ge (c_{\rm cir}c_0)^{C}=:\eta>0,
\end{align*}
where $C$ is a deterministic bound on the number of $r$-boxes in $\Lambda_{2R}$.

	\begin{figure}
	\begin{center}
	\includegraphics[width = 0.45\textwidth]{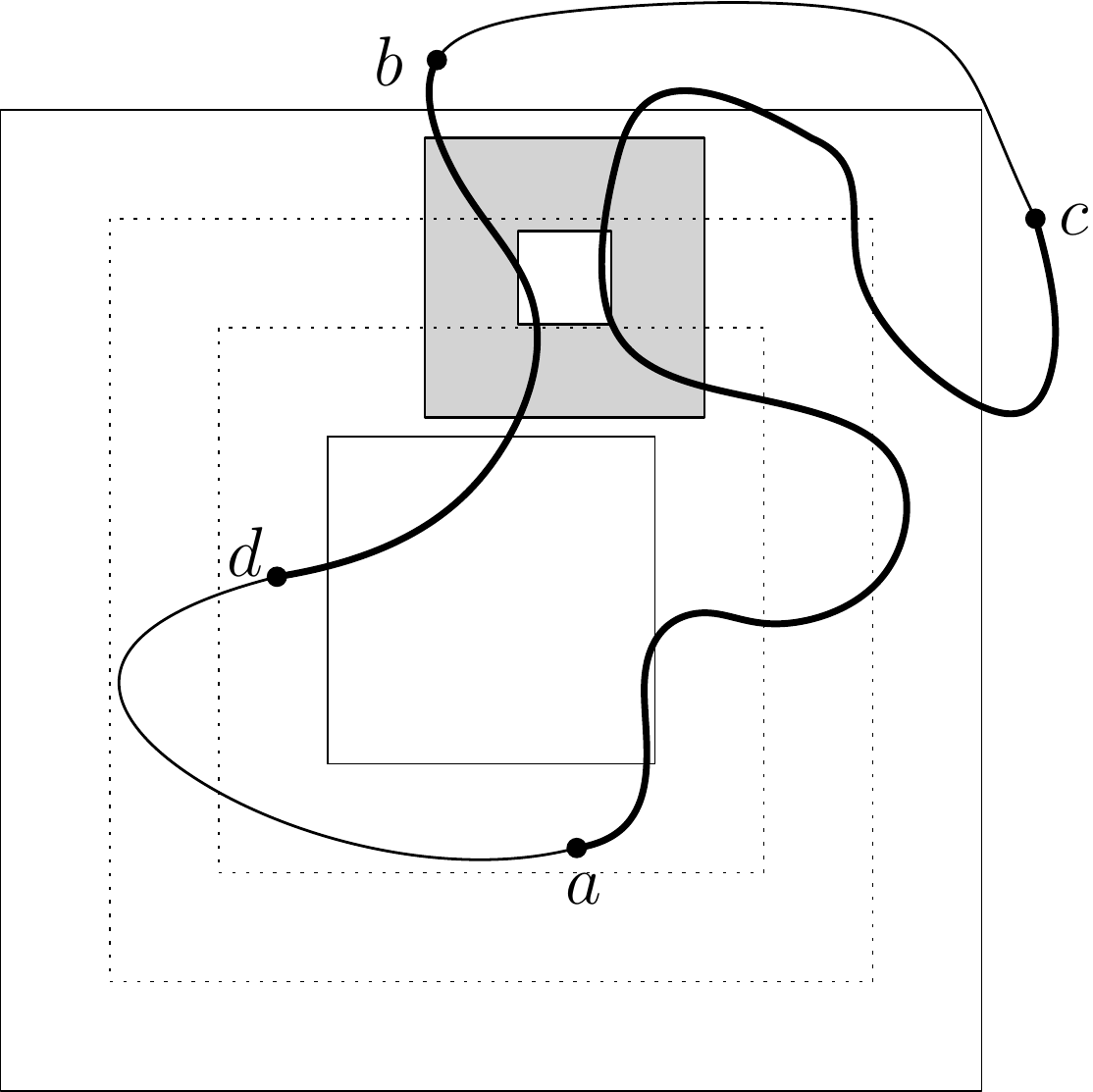}\qquad
	\includegraphics[width = 0.45\textwidth]{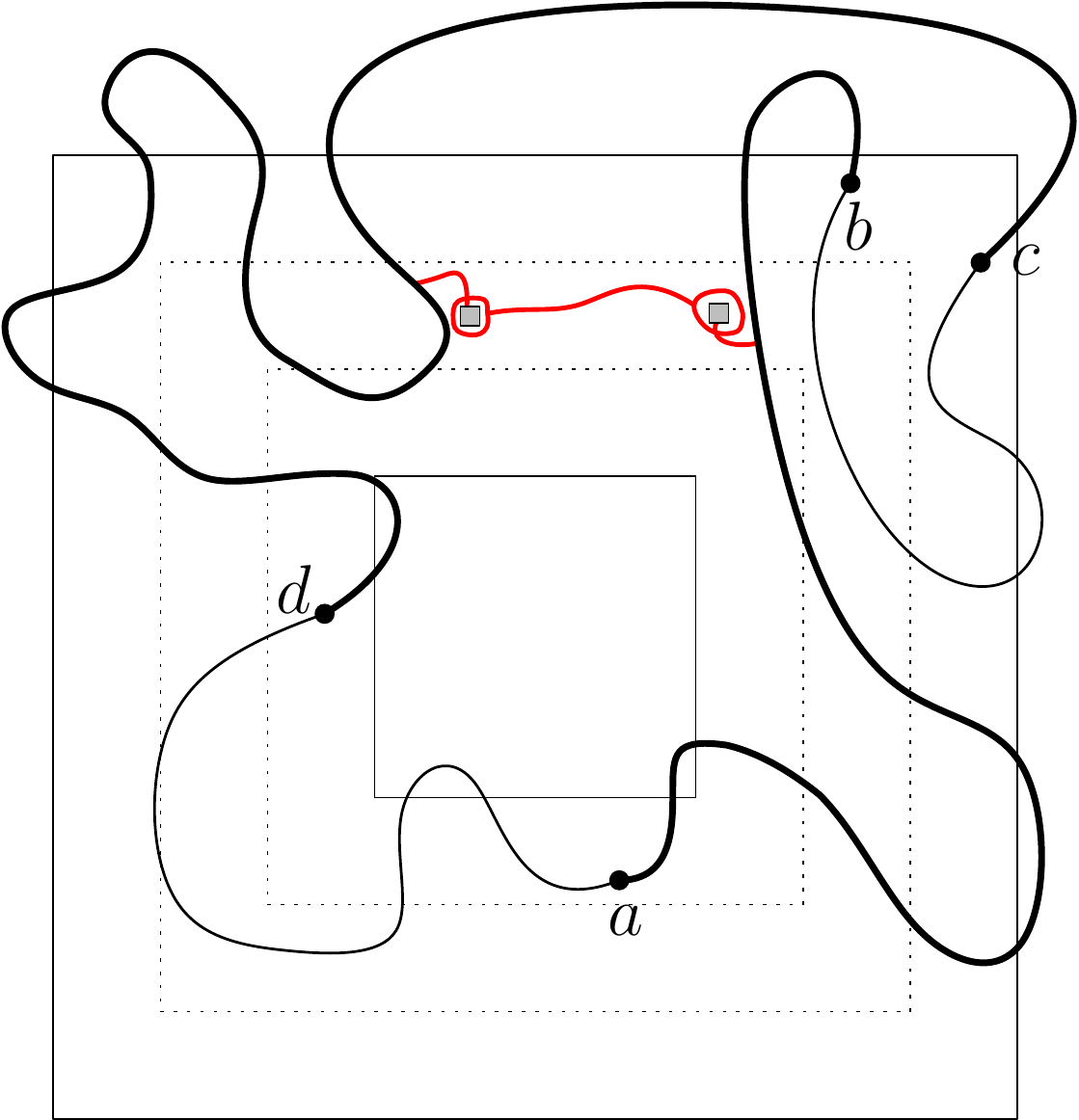}
	\caption{Quads $(\calD,a,b,c,d)$ with small extremal distance between $(ab)$ and $(cd)$. 
	{\em Left:}~The minimality of $R$ ensures that the arcs $(ab)$ and $(cd)$ are at a distance at least $R/3$ from each other in the middle annulus $\Lambda_{5R/3}\setminus\Lambda_{4R/3}$. Indeed, otherwise a smaller annulus (in grey) satisfying Fact~\ref{fact} may be found. 
	{\em Right:} When $H$ occurs, all seeds in $\La_{2R}$ are connected to each other; if in addition $F_S$ and $F_S'$ occur for particular seeds $S$ and $S'$, then $\calD$ contains a crossing from $(ab)$ to $(cd)$.}
	\label{fig:quad_annulus}
	\end{center}
	\end{figure}

\paragraph{Quads with arbitrary extremal distance}
	Fix $M >2$ and some quad $(\calD,a,b,c,d)$ with $\ell:= \ell_\calD[(ab),(cd)] \leq M$.
	By potentially restricting the crossing to a smaller quad, we may assume $\ell > 2$, which we do for simplicity. 
	Let $\Psi$ be the conformal map that maps $\calD$ to the rectangle $[-1,1] \times [0, 2\ell]$, 
	with $a,b,c,d$ being mapped to the corners $(-1,0),(1,0),(1,2\ell)$ and $(-1,2\ell)$, respectively. 
	
	For $\delta \in (0,1)$, define the simply connected domain (see Fig.~\ref{fig:short_to_long}),
	$$\calQ= \calQ(\delta):= [-1,1]^2 \setminus ([-\delta,\delta]^2 \cup \{0\}\times [-1,-\delta]).$$
	Consider four points (prime ends to be precise) $u,v,w$ and $t$ on $\partial \calQ$ that split its boundary into four arcs: 
	\begin{itemize}[noitemsep]
	\item $(uv)$ is the right side of the vertical segment $\{0\} \times[-1,-\delta]$;
	\item $(vw)$ coincides with the boundary of $[-1,1]^2$;
	\item $(wt)$ is the left side of $\{0\} \times [-1,-\delta]$;
	\item $(tu)$ coincides with the boundary of $[-\delta,\delta]^2$.
	\end{itemize}
	The quantity $\ell_\calQ[(uv),(wt)]$ can be chosen smaller than $m$ given by the first part of this section provided  $\delta$ is chosen sufficiently small.

	For $h=\delta, 2\delta,\dots, \ell$ (we assume that $\ell/\delta$ is an integer), write $\calQ_h$ for the intersection of $\calQ + (0,h)$ with $[-1,1] \times [0, 2\ell]$.
	Any such translate is a simply connected domain. 
	We will consider it with four marked prime ends $u_h,v_h,w_h,t_h$ given by
	\begin{itemize}[noitemsep]
	\item  	if $h \geq 1$, $(\calQ_h,u_h,v_h,w_h,t_h)$ is simply a translate of $(\calQ,u,v,w,t)$; 
	\item if $\delta < h < 1$, $v_h = (1,0)$, $w_h = (-1,0)$, $u_h$ and $t_h$ the translates of $u$ and $t$ by $(0,h)$;
	\item if $h= \delta $, let $v_h = (1,0)$, $w_h = (-1,0)$, $u_h = (\delta,0)$ and $t_h = (-\delta, 0)$.
	\end{itemize} 
	Notice that in all these cases 
	\begin{align}\label{eq:sfA}
		\ell_{\calQ_h}[(u_hv_h),(w_ht_h)] \leq \ell_{\calQ}[(uv),(wt)]  \leq m.
	\end{align}
	
	Consider now the pre-images $\Psi^{-1}(\calQ_h; u_h,v_h,w_h,t_h)$ of the domains $(\calQ_h; u_h,v_h,w_h,t_h)$.
	To not overburden the notation, we will consider that they are discrete quads; 
	this is not generally true, and $\Psi^{-1}(\calQ_h; u_h,v_h,w_h,t_h)$ should be replaced below by a discretisation of itself. 
	This may be done with only a limited influence on the constant $\eta (L)$ that is obtained at the end of the proof. 

	Since extremal length is preserved by conformal maps, 
	the extremal distance in $\Psi^{-1}(\calQ_h)$ between $\Psi^{-1}(u_h v_h)$ and $\Psi^{-1}(w_h t_h)$ is smaller than $m$ for any $h=\delta, 2\delta,\dots, \ell$. 
	Write $A_h$ for the event that $\Psi^{-1}(\calQ_h)$ 
	contains an open path from $\Psi^{-1}(u_h v_h)$ to $\Psi^{-1}(w_h t_h)$.
	The first part of this section implies that 
	\begin{align*}
	\phi_\calD^0 [A_h] \geq 	\phi_{\Psi^{-1}(\calQ_h)}^0 [A_h] \geq \eta.
	\end{align*}
	
	\begin{figure}
	\begin{center}
	\includegraphics[width = 0.9\textwidth]{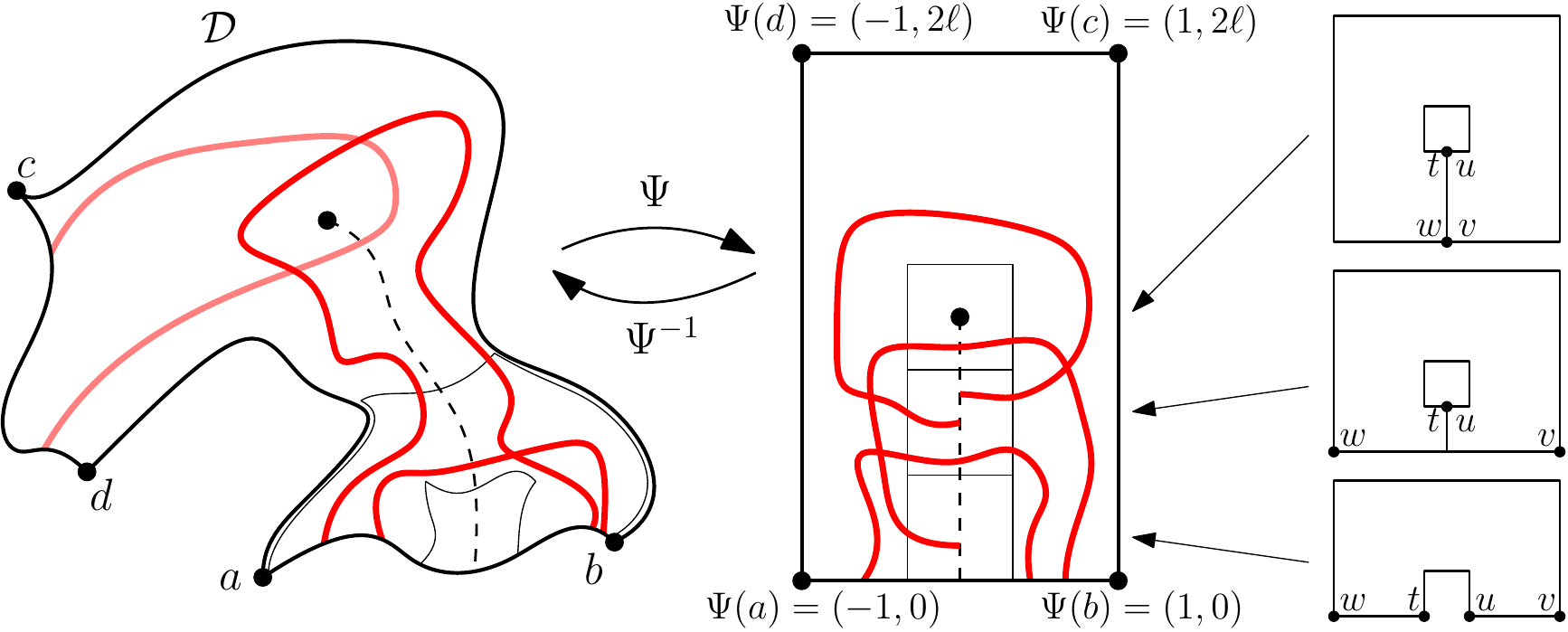}
	\caption{A domain $(\calD;a,b,c,d)$ is transformed by the conformal map $\Psi$ into the rectangle $[-1,1]\times [0,2\ell]$. 
	The domains $\calQ_{k\delta}$ (on the right) are used to pave the lower part of the rectangle; 
	if they all contain crossings, then the vertical line $\{0\}\times [0,\ell]$ is surrounded by an arc. 
	The same is true before the application of $\Psi$, that is in $\calD$ (left image). Finally, if $B$ and $T$ both occur, then $(ab)$ is connected to $(cd)$.}
	\label{fig:short_to_long}
	\end{center}
	\end{figure}
	
	We recommend to look at Fig.~\ref{fig:short_to_long} for the definitions coming next. Let $B$ be the event that there exists an open path in $\calD$ 
	with endpoints on $\Psi^{-1}([-1,0) \times \{0\})$ and $\Psi^{-1}((0,1] \times \{0\})$, respectively,  
	and which does not cross $\Psi^{-1}(\{0\}\times [0,\ell])$.
	If the events $A_h$ with $h = \delta, 2\delta,\dots, \ell$ occur simultaneously, then so does $B$ (this is easier to see after transformation by $\Psi$).
	By the FKG inequality~\eqref{eq:FKG} and the previous display,
	\begin{align*}
		\phi_\calD^0 [B]  \geq \prod_{k=1}^{\ell /\delta} \phi_\calD^0 [A_h] \geq \eta^{\ell/\delta}.
	\end{align*}
	Symmetrically, if $T$ is the event that $\calD$ contains a path connecting 
	$\Psi^{-1}([-1,0) \times \{2\ell\})$ and $\Psi^{-1}((0,1] \times \{2\ell\})$ and which avoids $\Psi^{-1}(\{0\}\times [\ell, 2\ell])$, 
	then we also have $\phi_\calD^0 [T]\geq \eta^{\ell/\delta}$.
	
	Finally, if $T$ and $B$ both occur, then $\calD$ contains a crossing from $(ab)$ to $(cd)$. 
	The FKG inequality~\eqref{eq:FKG} gives that
	$$ \phi_{\calD}^0[(ab)\longleftrightarrow(cd)] \geq \phi_\calD^0 [B\cap T]\ge \eta^{2\ell/\delta},$$
	which provides the desired conclusion with $\eta(M):= \eta^{2\ell/\delta}$.\hfill $\square$

\begin{remark}
	Note that the argument of this section also implies the following result, and here $q=4$ is not excluded. 
	For all $L > 0$, there exists $\eta(L) >0$ such that, for all~$1\le q\le 4$
	and~$(\calD, a,b,c,d)$ a discrete quad, 
	if~$\ell_\calD[(ab),(cd)] \le L$ then~
	\[\phi_{\calD}^{0/1}[\calC(\calD)] \geq\eta(L),\]
	where~$0/1$ denotes the boundary condition on~$\calD$ where the arcs~$(ab)$ and~$(cd)$ are wired and the rest of the boundary is free. 
	Indeed, a careful inspection of the proofs in this section shows the existence of 
	a family of annuli~$\Ann(x_i; r_i,2r_i) := \La_{2r_i}(x_i) \setminus  \La_{r_i}(x_i)$ with~$i = 1,\dots, k$ such that, 
	if each one contains and open path separating~$\La_{r_i}(x_i)$ from~$\La_{2r_i}(x_i)$, 
	then~$\calD$ is crossed from~$(ab)$ to~$(cd)$ by an open path 
	(for annuli intersecting~$\calD^c$, we ask for the existence of an open path in~$\calD \cap \Ann(x_i; r_i,2r_i)$ 
	that separates~$\La_{r_i}(x_i)$ from~$\La_{2r_i}(x_i)$ \emph{inside}~$\calD$). 
	Moreover,~$k$ is bounded in terms of~$\ell_\calD[(ab),(cd)]$ only and each  
	$\Ann(x_i; r_i,2r_i) \cap \calD$ intersect the boundary of~$\calD$ only along the wired arcs. 
	
	By crossing estimates from \cite{DumSidTas16} and \eqref{eq:CBC}, each annulus contains an open path separating the inside from outside with uniformly positive probability. Finally, by \eqref{eq:FKG} and the bound on~$k$, we conclude that the probability of~$\calC(\calD)$ may bounded in terms of ~$\ell_\calD[(ab),(cd)]$ only. 
\end{remark}

\section{First moment estimate: proof of Proposition~\ref{prop:10}}\label{sec:3}

The section is divided in three. 
We first show how the case for general $r$ follows from that with $r=0$. 
Then, we introduce the necessary background on parafermionic observables to prove the $r=0$ case. 
Finally, the last part is devoted to the proof of the $r=0$ case.

\subsection{Reduction to the case of $r=0$}

\begin{figure}
    \begin{center}
    	\includegraphics[width =0.4\textwidth]{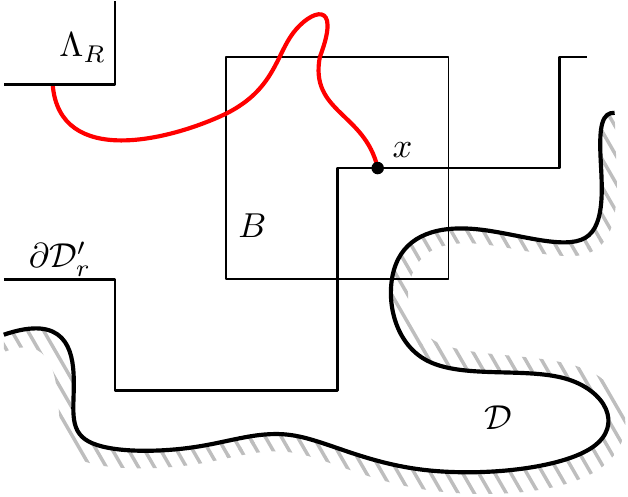}
    	\caption{When $x$ is connected to $\La_R$ inside $\calD_r'$, 
		then $B$ is connected to $\La_R$ and $x$ is connected to $\partial B$ inside $\calD_r'$. }
	    \label{fig:Bx}
    \end{center}
\end{figure}

Fix $R \geq r\ge1$ and let $\calD$ be a $R$-centred domain. 
By adapting the constant $c_4$ in~\eqref{eq:fund}, we may restrict our study to the case where $R/r$ is large enough; we make this assumption below.  
Let $\calD'_r$ be the connected component of the origin in the union of $r$-boxes included in $\calD$. 
Consider $x\in \partial\calD'_r$ and let $B = \La_r(y)$ be the $r$-box with center $y \in \partial\calD'_r$ closest to~$x$. 
When $R/r$ is large enough, $\La_R$ and $B$ do not intersect. 
Then, the comparison between boundary conditions~\eqref{eq:CBC} and the mixing property~\eqref{eq:mix2} give,
\begin{align*}
	\phi_{\calD'_r}^0[x\xleftrightarrow{\Lambda_{7R}}  \Lambda_R]
	&\le \phi_{B\cap\calD'_r}^{0/1}[x\longleftrightarrow \partial B\setminus \partial \calD'_r]\,\phi_{\calD'_r}^0[B\xleftrightarrow{\Lambda_{7R}}  \Lambda_R]\le C\pi_1^+(\|x - y\|)\phi_{\calD}^0[B\xleftrightarrow{\Lambda_{7R}}  \Lambda_R],
\end{align*}
where $\phi_{B\cap\calD'_r}^{0/1}$ denotes the measure on $B\cap\calD'_r$ with free boundary conditions on $\partial \calD'_r$ and wired on the rest of the boundary and $\|\cdot\|$ stands for the $L^\infty$-distance; see also Fig.~\ref{fig:Bx}.
Notice that any $B$ as above must intersect $\calD^c$ (otherwise its center would not lie on $\partial\calD'_r$).
By summing over all $x\in \partial\calD'_r$ we find
\begin{equation}\label{eq:hh4}
    \phi_{\calD_r'}^0[\mathbf M_0(\calD_r',R)] 
    = \sum_{x\in \partial\calD'_r}\phi_{\calD'_r}^0[x\xleftrightarrow{\Lambda_{7R}}  \Lambda_R]
    \leq C\sum_{k\le r/2}\pi_1^+(k) \sum_{B\cap \calD^c\ne\emptyset}\phi_{\calD}^0[B\xleftrightarrow{\Lambda_{7R}}  \Lambda_R].
\end{equation}
Apply now the case $r=0$ of Proposition~\ref{prop:10}  to the $R$-centred domain\footnote{Formally, $\calD'_r$ is not always $R$-centred, but is $R'$-centred for some $R'$ between $R/2$ and $R$; this suffices to apply Proposition~\ref{prop:10}.} $\calD'_r$ to bound the left-hand side from below by $c_4R\pi_1^+(R)$. 
Moreover, Lemma~\ref{lem:mn} bounds from above the first sum in the right-hand side by $r\pi_1^+(r)$. 
Dividing by the latter, we obtain~\eqref{eq:fund} for~$r \geq 1$. 

\subsection{Background on parafermionic observables}\label{sec:3.1}

The proof of Proposition~\ref{prop:10} relies heavily on parafermionic observables, that we define below.
 These definitions are now classical and we refer to \cite{Dum17a} for details. We also recommend that the reader looks at Fig.~\ref{fig:ab}.  

Let $\Omega=(V,E)$ be a discrete domain, let $a$ and $b$ be two vertices on $\partial \Omega$. The triplet $(\Omega,a,b)$ is called a {\em Dobrushin domain}. Orient $\partial \Omega$ in counterclockwise order. It is divided into two boundary arcs denoted by $(ab)$ and $(ba)$: the first one from $a$ to $b$ (excluding $a$ and $b$) and the second one from $b$ to $a$ (including the endpoints). 
The {\em Dobrushin boundary conditions} are defined to be free on $(ab)$ and wired on $(ba)$. 
Below, the measure on $(\Omega,a,b)$ with Dobrushin boundary conditions is denoted by $\phi^{0/1}_{\Omega}$.

Let $(\mathbb Z^2)^\star$ be the dual of $\mathbb Z^2$ (defined as the translate of $\mathbb Z^2$ by $(1/2,1/2)$). This way, each edge $e$  of $\mathbb Z^2$ is associated to a unique edge $e^\star$ (the one that crosses $e$) of $(\mathbb Z^2)^\star$. The dual $\Omega^\star=(V^\star,E^{\star})$ of the domain $\Omega=(V,E)$ is the subgraph of $(\mathbb Z^2)^\star$ spanned by $E^\star$, where $E^\star$ is the set of dual edges associated to $E'=E\setminus\{\text{edges of $(ba)$}\}$. 
Any configuration $\omega$ on $\Omega$ will be completed by open edges on $(ba)$ and closed edges on $\Omega^c$, which leads us to an identification between configurations $\omega$ on $\Omega'=(V,E')$ and dual configurations $\omega^\star$ on $\Omega^\star=(V^\star,E^\star)$.  Then, the dual configuration $\omega^*$ has the property that dual edges between vertices of $\partial\Omega^*$ that are bordering $(ab)$ are open (we call the set of such edges $(ab)^*$). See Fig.~\ref{fig:ab} for an illustration.

The loop representation of a configuration on $\Omega$ is supported on the {\em medial graph} of $\Omega$ defined as follows. 
Let $(\bbZ^2)^\diamond$ be the {\em medial} lattice, with vertex-set given by the midpoints of edges of $\bbZ^2$ and edges between pairs of nearest vertices (\emph{i.e.}~vertices at a distance $\sqrt 2/2$ of each other). It is a rotated and rescaled version of $\bbZ^2$.   
Let $\Omega^\diamond$ be the subgraph of $(\bbZ^2)^\diamond$ spanned by the edges of $(\bbZ^2)^\diamond$ adjacent to a face corresponding to a vertex of $\Omega\setminus (ba)$ or $\Omega^* \setminus (ab)^*$. 
Let $e_a$ and $e_b$ be the two medial edges entering and exiting $\Omega^\diamond$ between the arcs $(ba)$ and $(ab)^*$.

Let $\omega$ be a configuration on $\Omega$; recall its dual configuration $\omega^*$. 
Draw self-avoiding paths on $\Omega^\diamond$ as follows: a path arriving at a vertex of the medial lattice 
always takes a $\pm \pi/2$ turn at vertices so as not to cross the edges of $\omega$ or $\omega^*$. 
The loop configuration thus defined is formed of a path between $e_a$ and $e_b$ and disjoint loops; 
together these form a partition of the edges of $\Omega^\diamond$. 
We will not detail further the definition of the loop representation, rather direct the reader to Fig.~\ref{fig:ab} and \cite{Dum17a} and point out that
\begin{itemize}
\item any vertex of $\Omega^\diamond$ (with the exception of the endpoints of $e_a$ and $e_b$) is contained either in an edge of $\omega$ or an edge of $\omega^*$.
Therefore there is exactly one coherent way for the loop to turn at any vertex of $\Omega^\diamond$;
\item the edges of $\omega$ in $(ba)$ and the edges of $\omega^*$ in $(ab)^*$ are such that the loops, when reaching boundary vertices, turn so as to remain in $\Omega^\diamond$.
\end{itemize}
In the loop configuration, the self-avoiding curve with endpoints $e_a$ and $e_b$ is called the \emph{exploration path};
it is denoted by $\gamma=\gamma(\omega)$ and is oriented from $e_a$ to $e_b$. 
For an edge $e \in \gamma$, let $\text{W}_{\gamma}(e,e_b)$ be the winding of $\gamma$ between $e$ and $e_b$, 
that is $\pi/2$ times the number of left turns minus the number of right turns taken by $\gamma$ when going from $e$ to $e_b$.

\begin{definition}\label{def:parafermionic_observable}
    Consider a Dobrushin domain $(\Omega,a,b)$.  
     The {\em parafermionic observable} $F=F_{\Omega,a,b}$  is defined for any (medial) edge $e$ of $\Omega^\diamond$ by
    \begin{equation*}
      F(e) ~:=~\phi^{0/1}_{\Omega}[{\rm e}^{{\rm i}\sigma 
      \text{W}_{\gamma}(e,e_b)} \mathbf1_{e\in \gamma}],
    \end{equation*}
    where $\sigma\in[0,1]$ is the solution of the equation
    \begin{equation}\label{eq:hahaha}
    	\displaystyle  \sin (\sigma \tfrac\pi2) = \sqrt{q}/2.
    \end{equation}
\end{definition}

The parafermionic observable satisfies a very special property first observed in \cite{Smi10} (see also \cite[Thm.~5.16]{Dum17a}); 
it applies for any  $q>0$ when $p=\sqrt q /(1 + \sqrt q)$. 
For any Dobrushin domain $(\Omega,a,b)$ and any vertex $v$ of $\Omega^\diamond$ corresponding to an edge of $\Omega \setminus (ba)$,
\begin{equation}\label{rel_vertex}
	\sum_{i=1}^4 \eta(e_i) F(e_i) = F(e_1) - {\rm i}F(e_2) - F(e_3) + {\rm i}F(e_4) = 0,
\end{equation}
where $e_1$, $e_2$, $e_3$ and $e_4$ are the four edges incident to $v$, indexed in clockwise order,
and $\eta(e_i)$ is the complex number of norm one with same direction as $e_i$ and orientation from $v$ towards the other endpoint of $e_i$.
Write $\Int(\calC)$ for the set of vertices $v$ of the medial lattice which correspond to a primal edge of $\Omega \setminus (ba)$,
and $\calC$ for the set of medial edges of $\Omega^\diamond$ with exactly one endpoint in $\Int(\calC)$. 
Then, summing the relation above over all vertices $v \in \Int(\calC)$, we find
\begin{equation}\label{eq:rel_vertex}
	\sum_{e\in \calC}\eta(e)F(e)=0,
\end{equation}
where $\eta(e)$ is the complex number of norm one with direction given by $e$ and orientation from the endpoint of $e$ in $\Int(\calC)$ towards the outside. 

\begin{remark}
This relation should be understood as ``the contour integral of the parafermionic observable along the boundary of $\Omega^\diamond$ is 0''.
The careful reader may, however, notice that $\calC$ does not always form a closed curve (see Fig.~\ref{fig:ab}).
\end{remark}

\subsection{Proof of Proposition~\ref{prop:10} when $r=0$}\label{sec:r=0}

For technical reasons related to the parafermionic observable, we first provide an estimate in a geometry given by special Dobrushin domains. For $\ell,m\ge0$, call $(\Omega,a,b)$ a {\em $(m,\ell)$-corner} Dobrushin domain (see Fig.~\ref{fig:ab} for an illustration) if its boundary is made of 
\begin{itemize}[noitemsep]
\item the vertical segment between $(0,0)$ and $b:=(0,m)$,
\item the horizontal segment between $(0,0)$ and $a:=(\ell,0)$,
\item a self-avoiding curve $\gamma$ between $a$ and $b$ avoiding the previous two segments and going clockwise around $0$. 
\end{itemize}

\begin{figure}
\begin{center}
\includegraphics[width = 0.47\textwidth, page = 1]{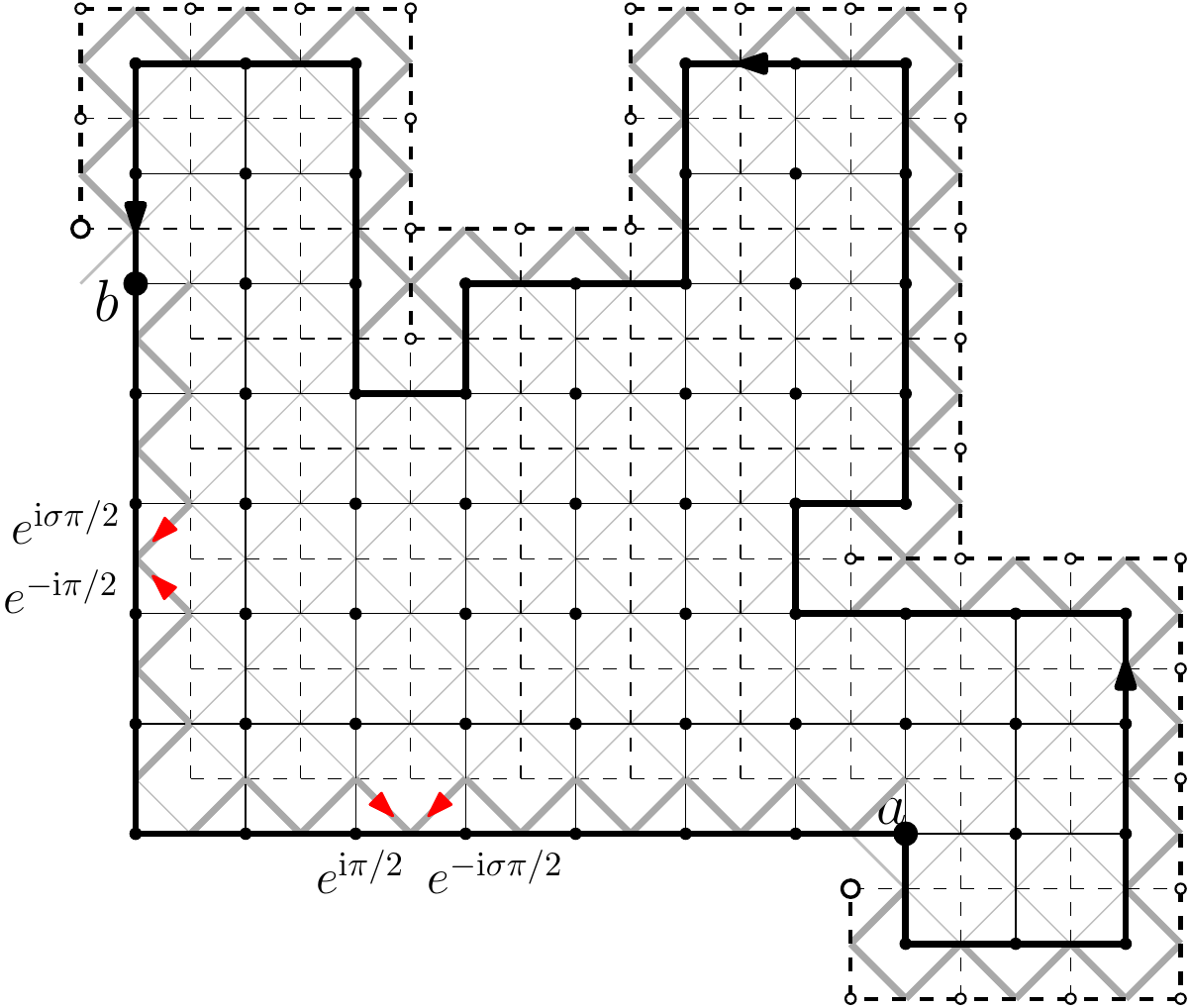}\quad
\includegraphics[width = 0.47\textwidth, page = 3]{ab2.pdf}
\caption{{\em Left:} A $(m,\ell)$-corner domain $\Omega$. The edges of $\Omega^\diamond$ are grey; those of $\alpha$ are bold.
{\em Right:} The loop representation of a configuration. Loops are drawn on the medial lattice $\Omega^\diamond$ so as not to intersect any open or dual-open edges. For a point $x \in (ab)$, the interface passes between this point and the free arc $(ab)^*$ if and only if $x\leftrightarrow (ba)$. The winding of a curve going from a medial edge
 adjacent to the primal arc $(ba)$ to $e_b$ is equal to $0$, $\pi/2$ or $\pi$.}
\label{fig:ab}
\end{center}
\end{figure}

\begin{lemma}\label{lem:op}
	There exists $c_5>0$ such that for every $(m,\ell)$-corner Dobrushin domain $\Omega$ with $m\ge\ell$, 
		 $$\sum_{x\in (ab)} \phi^{0/1}_\Omega[x\longleftrightarrow (ba)]\ge c_5\,m\pi_1^+(m).$$
\end{lemma}

\begin{proof}
Consider the parafermionic observable on $(\Omega,a,b)$. Note that the edges of $\calC$ are of three kinds:
    \begin{itemize}[noitemsep]
    \item the medial edges $e_a$ and $e_b$;
    \item the medial edges of $\calC$ incident to $(ab)^*$, we call the set of such edges $\alpha$; 
    \item the medial edges of $\calC$ incident to $(ba)$, we call the set of such edges $\beta$.
    \end{itemize} 
   Equation~\eqref{eq:rel_vertex} applied to $\Omega$ implies that 
    \begin{align*}	    \big|\sum_{e\in \alpha}\eta(e)F(e)\big| 
	    = \big|\eta(e_a)F(e_a)+\eta(e_b)F(e_b)+\sum_{e\in \beta}\eta(e)F(e)\big|.
    \end{align*}
    Notice that, for any medial edge $e \in \alpha$, if the interface passes through $e$, then the vertex $x \in (ab)$ which is adjacent to $e$ 
	is connected inside $\Omega$ to the arc $(ba)$ (see also Fig.~\ref{fig:ab}). 
	Since there are at most two edges of $\alpha$ adjacent to any one vertex of $(ab)$, 
	we deduce that 
    \begin{align}\label{eq:h11}
    	2 \sum_{x\in(ab)}\phi_\Omega^{0/1}[x\longleftrightarrow (ba)]\ge  \sum_{e\in \alpha}|F(e)|\ge \big|\sum_{e\in\beta}\eta(e)F(e)\big|-2,
	\end{align}
	where the second inequality uses the triangular inequality and the fact that $|F(e_a)|=|F(e_b)|=1$

	Similarly, for any edge $e \in \beta$, 
    $\gamma$ passes through $e$ if and only if 
    the unique vertex $y \in \Omega^*$ adjacent to $e$ is connected by a dual-open path to $(ab)^*$ in $\Omega^*$;
    write $y\xleftrightarrow{*} (ab)^*$ for the latter event. 
	When $\gamma$ contributes to $\eta(e) F(e)$, the argument of its contribution is determined by the orientation of $e$.
    There are four possible arguments: $e^{{\rm i}\sigma \pi/2}$ and $e^{-{\rm i} \pi/2}$ for edges on the vertical section of $(ba)$ 
    and $e^{-{\rm i}\sigma \pi/2}$ and $e^{{\rm i} \pi/2}$ for edges on the horizontal section of $(ba)$ 
    (up to a fixed phase that depends on the geometry of $\Omega$ around $b$). 
    Additionally, observe that any edge $e$ whose contribution has argument $e^{\pm{\rm i} \pi/2}$ 
    may be paired with the edge $f \in \beta$ to the right or above it, 
    which has contribution of same absolute value as that of $e$, and argument $e^{\mp{\rm i}\sigma \pi/2}$.	
    Indeed, $\gamma$ passes through $e$ if and only if it also passes through $f$. 
    Thus,
    \begin{equation}\label{eq:h12}
    	\big|\sum_{e\in\beta}\eta(e)F(e)\big| \geq  \cos \big[(1+\sigma)\tfrac{\pi}{4}\big]\sum_{y\in(ba)^*}\phi_\Omega^{0/1}[y\xleftrightarrow{*}(ab)^*].
    \end{equation}
      
    By self-duality, the crossing estimates~\eqref{eq:wRSW}, and the mixing property~\eqref{eq:mix2}, we may deduce that for every $y\in(ab)^*$ between distance $m/3$ and $2m/3$ of $b$, 
    \begin{equation}\label{eq:h13}
    \phi_\Omega^{0/1}[y\xleftrightarrow{*}    (ab)^*]\ge c\pi_1^+(m)
    \end{equation}
	for some constant $c>0$ independent of $m$.  
	Putting~\eqref{eq:h11}--\eqref{eq:h13} together, we find that 
    \begin{align*}
    	\sum_{x\in(ab)}\phi_\Omega^{0/1}[x\longleftrightarrow (ba)] 
	 	\geq \tfrac12 \cos \big[(1+\sigma)\tfrac{\pi}{4}\big] \sum_{y\in(ba)^*}\phi_\Omega^{0/1}[y\xleftrightarrow{*}    (ab)^*] - 1
		\geq c' m\pi_1^+(m),
	\end{align*}
	for some $c' > 0$. In the last inequality, we used that $\cos \big[(1+\sigma)\tfrac{\pi}{4}\big]>0$ when $q < 4$ (see~\eqref{eq:hahaha}) and that 
	$m\pi_1^+(m)$ tends to $\infty$ (see Lemma~\ref{lem:mn}).
\end{proof}

We now proceed with the proof of Proposition~\ref{prop:10} in the case $r = 0$, which is slightly technical. Throughout this proof, we will assume $R$ to be large; small values of $R$ may be incorporated by adjusting the constant in~\eqref{eq:fund}.  Fix a small quantity $\delta > 0$ (we will see below how small $\delta$ needs to be, and that it does not depend on $R$) and assume for simplicity that $\delta R$ and $\delta^2 R$ are integers (we may do this as we may take $R$ large and adjust $\delta$). In the proof below, it will be important to keep track of the dependencies in $\delta$: the constants $c,c',C,C'$, etc are independent of $\delta$, while $c(\delta), c'(\delta)$, etc do depend on $\delta$.

The idea is to construct a $(m,\ell)$-corner Dobrushin domain $(\Omega,a,b)$ associated with each $R$-centred domain $\calD$ in order to apply the previous lemma. We refer to Fig.~\ref{fig:domain} for an illustration of the construction. 

\begin{proof}[Proposition~\ref{prop:10} for $r=0$]
Fix a $R$-centred domain $\calD$ and consider the largest Euclidean open ball $B$ centred at the origin and which is included in $\calD$. Let $x=(x_1,x_2)$ be a vertex of $\partial\calD\cap\partial B$ and assume with no loss of generality that $x$ is in the wedge $\{(u,v):u\ge v\ge 0\}$. Let $y = (y_1,y_2) $ be the rightmost vertex of the half-line $\bbZ_+ \times \{x_2 + \delta R\}$ that is contained in $B$. Finally, let $\tau$ be the translation mapping $x$ to $y$.

\begin{definition}[corner domain associated with $\calD$]
Let $a:=x-(8\delta^2 R,0)$, $b:=\tau(a)$ and 
$\mathrm{Rect}$ be the rectangle with edges parallel to the axis, top-left corner $b$ and bottom-right corner $x$. 
Finally, let $\Omega$ be the connected component of $a$ in the subgraph $(\calD\cap \Lambda_{7R})\cap\tau(\calD\cap \Lambda_{7R})\cap (\mathrm{Rect} \cup B^c)$; see the shaded region in Fig.~\ref{fig:domain}.
\end{definition}
\begin{figure}
\begin{center}
\includegraphics[width = 0.51\textwidth]{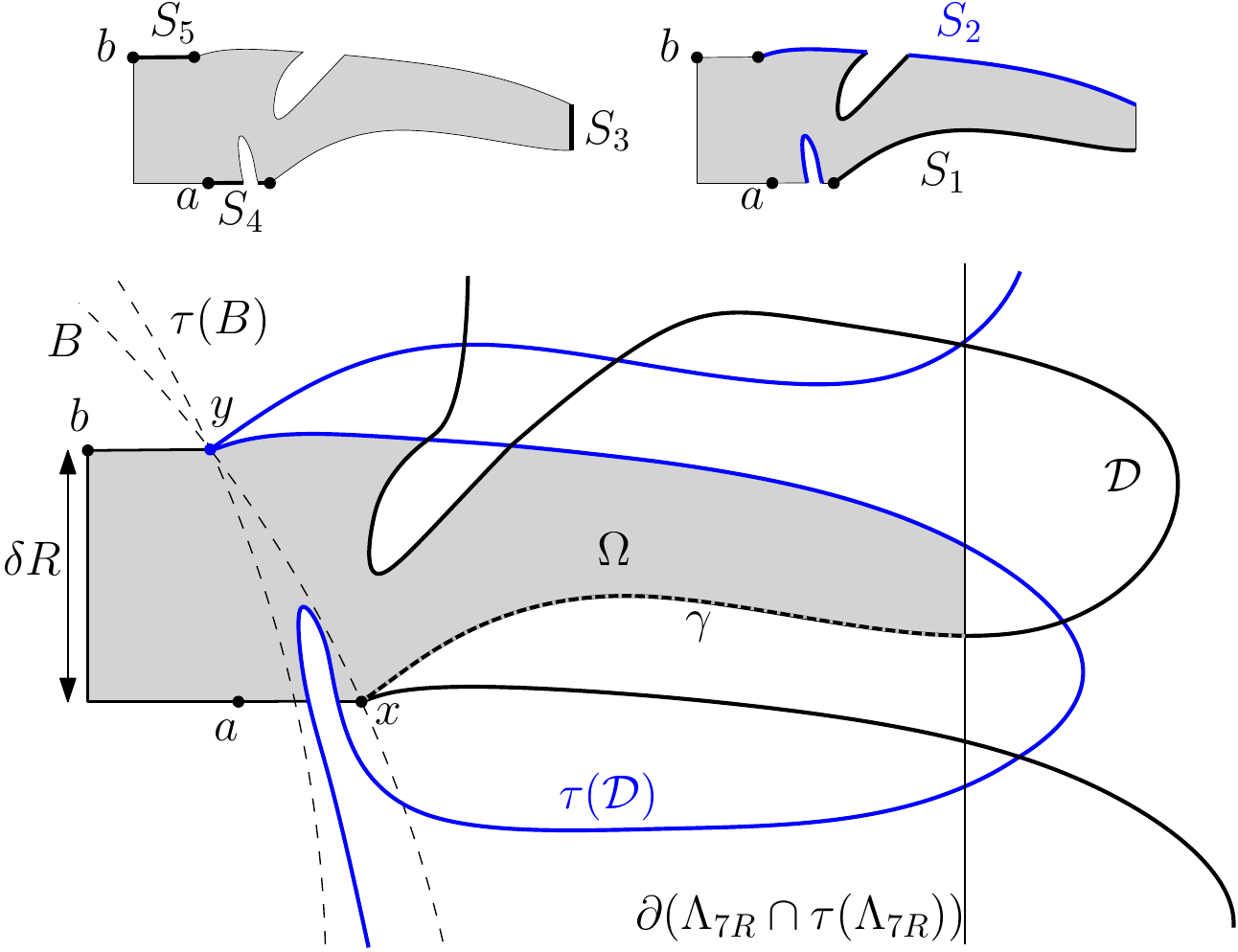}\,\,
\includegraphics[width = 0.47\textwidth]{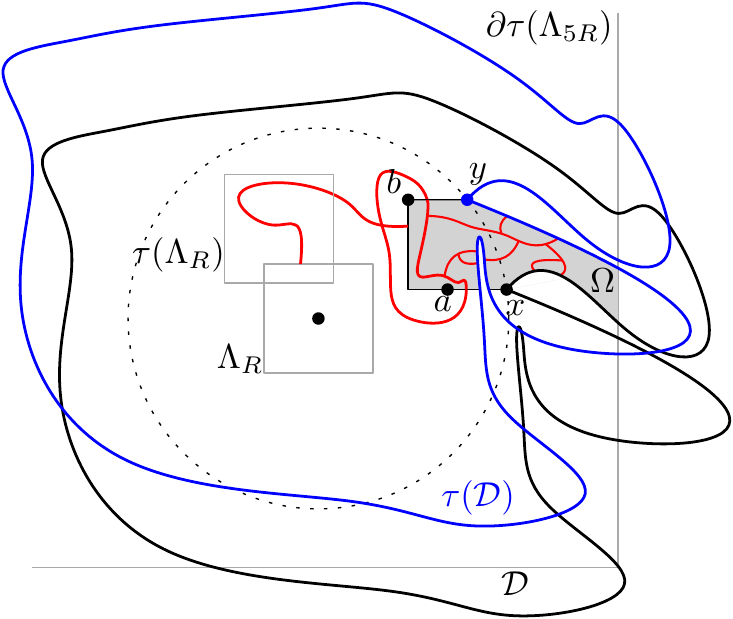}
\caption{{\em Left:} The domain $\Omega$; part of its boundary is contained in $\partial \calD$ (black) or in $\partial\tau(\calD)$ (blue). The arc $(ab)$ is divided into five sets $S_1,\dots, S_5$ as described above. 
{\em Right:} When the event $E$ occurs, the measure induced inside $\Omega$ dominates that with Dobrushin boundary conditions. Then,~\eqref{eq:k5} provides a bound on the number of points connected to $\La_{R}\cap\tau(\La_{R})$.}
\label{fig:domain} 
\end{center}
\end{figure}

Assuming $\delta > 0$ is small enough, a simple trigonometric computation shows that the distance between $x$ and $y$ is at most $2\delta R$. 
This further shows that the distance along the horizontal line passing through $x$ between $x$ 
and the translated ball $\tau(B)$ is smaller than $4\delta^2 R$.
In particular, $a$ is then contained in $B \cap \tau(B)$, and so is $b$. 
Thus, if we assume $\delta$ to be small and $R$ large enough, 
the whole arc $(ba)$ of the boundary of $\mathrm{Rect}$ is also contained in $B \cap \tau(B)$, hence it is part of~$\partial \Omega$. 
As such $(\Omega, a,b)$ is the translate of a $(m ,\ell)$-corner Dobrushin domain with $m = \delta R$ and $\ell:= x_1 - y_1$. 
Lemma~\ref{lem:op} applied to $(\Omega, a,b)$ gives
\begin{equation}\label{eq:kkkk}
	\sum_{z \in (ab)} \phi^{0/1}_\Omega[z\longleftrightarrow (ba)]\ge c_5\delta R\,\pi_1^+(\delta R).
\end{equation} 
We are now going to harvest this inequality by splitting the boundary arc $(ab)$ into five types of vertices and estimating the contribution of each of them in order to get a more useful inequality, namely~\eqref{eq:k5} below. Divide $(ab)$ into five sets: 
\begin{itemize}[noitemsep]
\item[$S_1=$]  vertices of $\partial \calD$;
\item[$S_2=$]  vertices of $\tau(\partial \calD)$;
\item[$S_3=$]  vertices of $\partial (\La_{7R} \cap \tau(\La_{7R}))$;
\item[$S_4=$]  vertices of the horizontal segment between $a$ and $x$ -- call this set $[ax]$;
\item[$S_5=$]  vertices of the horizontal segment between $y$ and $b$ -- call this set $[yb]$.
\end{itemize} 
Next, we analyse the contribution of each of the sets $(S_i)_{i=1\dots5}$ to the right-hand side of~\eqref{eq:kkkk}. 
Our goal is to show that $S_1$ and $S_2$ contribute significantly to~\eqref{eq:kkkk},
and thus that $S_3$, $S_4$ and $S_5$ contribute only partially. This will be valid for $\delta> 0$ small enough, but independent of $N$. 
\smallskip 

\noindent
{\em Contribution  of $S_4\cup S_5$}. 
There are at most $16 \delta^2 R$ vertices in $S_4\cup S_5$ and
the crossing estimates~\eqref{eq:wRSW}, the mixing property~\eqref{eq:mix2} and Lemma~\ref{lem:mn} give that 
\begin{equation}\label{eq:kkk}
	\sum_{z\in S_4 \cup S_5}\phi^{0/1}_\Omega[z\longleftrightarrow (ba)]
	\le C\,\delta^2 R\,\pi_1^+(\delta^2 R).
\end{equation}
Choosing $\delta > 0$ small enough, we may suppose that the contribution of these vertices is smaller than a quarter of the right-hand side of~\eqref{eq:kkkk}.
\bigbreak

\noindent {\em Contribution  of $S_3$}.
Any open path linking a vertex of $S_3$ to $(ba)$ needs to traverse a long thin corridor with free boundary conditions on its sides, something which occurs with very small probability. Formalising this is technical, but not surprising. 

For any vertex $z \in S_3$, the mixing property~\eqref{eq:mix2} gives
\begin{align*}
	\phi^{0/1}_\Omega[z\longleftrightarrow (ba)]
	\le C' \pi_1^+(\delta R) \,\phi^{0/1}_\Omega[(ba)\longleftrightarrow\partial\Lambda_{6R}].
\end{align*} 
Let us bound from above the last term on the right-hand side. 
For this term to be positive, $\Omega$ needs to contain vertices of $\partial\Lambda_{6R}$, hence we may restrict ourselves to this case. 

Recall that $x$ is the closest point of $\partial \calD$ to $0$ in euclidian distance, and that we assumed that $\La_{3R} \nsubset\calD$.
Thus $x$ is contained in $\La_{3\sqrt2 R} \subset \La_{5R}$.
Moreover $\Omega \subset \calD$, and therefore $\partial\calD$ does intersect $\partial \La_{6R}$.
Let $\tilde x$ be the first vertex of $\partial \La_{6R}$ when going around $\partial \calD$ in counter-clockwise order starting from $x$;
let $\gamma$ be the arc of  $\partial \calD$ between $x$ and $\tilde x$ (see Fig.~\ref{fig:domain}).
Then $\gamma$ has length at least $R$.
Choose a family of points $x_1,\dots,x_s$ on $\gamma$, at a distance at least $10\delta R$ from each other, from $\partial \La_{6R}$ and from ${\rm Rect}$. Due to the length of $\gamma$, one may choose $s \geq c / \delta$ for some small constant $c > 0$. 

Notice that, for any $1 \leq j \leq s$, any circuit of dual edges contained in $\La_{5\delta R }(x_j)$ and surrounding both $x_j$ and $\tau(x_j)$ separates $(ba)$ from $\partial \La_{6R}$ inside $\Omega$. 
The crossing estimates~\eqref{eq:wRSW}, the fact that $\tau$ is a translation by at most $2\delta R$ and the comparison of boundary conditions~\eqref{eq:CBC} imply the existence of a universal positive constant 
that bounds from below the $\phi^{0/1}_\Omega$-probability of existence of a dual-open path contained in $\La_{5\delta R }(x_j) \cap \Omega$ that 
separates $(ba)$ from $\partial \La_{6R}$ inside $\Omega$ for each $j = 1,\dots, s$. 
The mixing property~\eqref{eq:mix2} and the lower bound on $s$ yield 
\begin{align*}
	\phi^{0/1}_\Omega[(ba)\longleftrightarrow\partial\Lambda_{6R}]
	\leq e^{-c'/\delta}
\end{align*}
for some $c' > 0$. 
In conclusion 
\begin{align*}
	\sum_{z  \in S_3}	\phi^{0/1}_\Omega[z\longleftrightarrow (ba)]
	\le C'' e^{-c'/\delta} R \,\pi_1^+(\delta R),
\end{align*}
where the factor $R$ is an upper bound comes from the number of terms in the sum (recall that $S_3 \subset \partial \La_{7R} \cup \partial \tau(\La_{7R})$).
By choosing $\delta$ smaller than some universal constant, the above may be rendered smaller than a quarter of the right-hand side of~\eqref{eq:kkkk}.

\bigbreak\noindent
{\em Contribution  of $S_1$ and $S_2$}.
Overall, considering the bounds on the contributions of vertices in $S_3\cup S_4\cup S_5$, we find that~\eqref{eq:kkkk} implies
\begin{equation}\label{eq:k5}
	\sum_{z\in S_1 \cup S_2}\phi^{0/1}_\Omega[z\longleftrightarrow (ba)]
	\geq \tfrac12 c_5\, \delta R \, \pi_1^+(\delta R)\ge c(\delta) R\pi_1^+(R).
\end{equation}
\bigbreak
We are now in a position to conclude the proof. We have a large $\phi^{0/1}_\Omega$-expectation of the number of $z$ on the boundary that are connected to $(ba)$ and we now need to convert it to an estimate on the $\phi^0_\calD$-expectation of the number of $z$ on the boundary that are connected to $\Lambda_R$.

Consider the event $E$ that there exists an open circuit surrounding $(ba)$ in $B\cap\tau (B)$, 
and an open path from $\Lambda_{R}\cap \tau(\Lambda_{R})$ to $(ba)$ in $B\cap\tau (B)\setminus \mathrm{Rect}$. 
Recall that, by construction, the arc $(ba)$ is at a distance at least $\delta^2 R$ from $[B\cap\tau (B)]^c$, and therefore of $[\calD\cap\tau(\calD)]^c$. Set $\calD':=(\calD\cap \Lambda_{7R})\cap\tau(\calD\cap \Lambda_{7R})$.
Using the FKG inequality~\eqref{eq:FKG} and the crossing estimates ~\eqref{eq:wRSW}, we find that
\begin{equation}\label{eq:k6}
	\phi_{\calD'}^0[E]\ge c'(\delta)>0.
\end{equation}
If $E$ occurs, let $\Gamma$ be the inner-most open circuit surrounding $(ba)$ in $B\cap\tau (B)$,
and let~$\Omega'$ be the set of edges of $\Omega$ that lie outside $\Gamma$; notice that $\Gamma$ may be explored from inside and that, by the definition of $E$, $\Gamma$ is connected to $\Lambda_{R}\cap \tau(\Lambda_{R})$.
By ~\eqref{eq:CBC} and~\eqref{eq:SMP}, conditioning $\phi_{\calD'}^0$ on $E$ and on the realisation of $\Gamma$ induces
a measure on $\Omega'$ that dominates $\phi^{0/1}_\Omega$. Thus,~\eqref{eq:k5} and~\eqref{eq:k6} together give
\begin{equation*}
	\sum_{z\in  S_1 \cup S_2} \phi^{0}_{\calD'}[z \xleftrightarrow{\calD'} \Lambda_R\cap \tau(\Lambda_R)] 
	\ge c''(\delta) R\pi_1^+(R).
\end{equation*}
Observe that $\phi^{0}_{\calD'}$ is dominated by both $\phi^{0}_{\calD}$ and $\phi^{0}_{\tau(\calD)}$
and that $S_1 \subset \partial \calD \cap \La_{7R}$ and $S_2 \subset \tau(\partial \calD \cap \La_{7R})$.
Thus, the above implies 
\begin{equation*}
	\sum_{z\in  \partial \calD} \phi^{0}_{\calD}[z \xleftrightarrow{\Lambda_{7R}}   \Lambda_R\cap\tau(\Lambda_R)] 
	+ 	\sum_{z\in  \tau(\partial \calD)} \phi^{0}_{\tau(\calD)}[z \xleftrightarrow{\tau(\Lambda_{7R})}   \Lambda_R\cap \tau(\Lambda_R)]
	\ge c''(\delta)R\pi_1^+(R).
\end{equation*}
We conclude by observing that both terms are smaller than $\sum_{z\in  \partial \calD } \phi^{0}_{\calD}[z \xleftrightarrow{\Lambda_{7R}}\Lambda_R]$.
\end{proof}

\section{Applications to arm events}\label{sec:4}

In this section, we gather  results concerning arm-events. Section~\ref{sec:5.1a} proves a lower bound for the probability of the one arm event in the half-space (Lemma~\ref{lem:mn}), which is necessary for the proof of Theorem~\ref{thm:RSWquads}. The next sections contain applications of Theorem~\ref{thm:RSWquads} to other arm events. 

For $r \leq R$ consider the annulus $\La_R\setminus \La_r$ with inner boundary $\partial \La_r$ and outer boundary $\partial \La_R$. 
A self-avoiding path of $\bbZ^2$ or $(\bbZ^2)^*$ connecting the inner to the outer boundaries of the annulus is called an {\em arm}. 
We say that an arm is {\em of type $1$} if it is composed of primal edges that are all open, and {\em of type $0$} if it is composed of dual edges that are all dual-open. 
For $k\ge1$ and $\sigma\in\{0,1\}^k$\,, define $A_{\sigma}(r,R)$ to be the event that there exist $k$ \emph{disjoint} arms from $\partial\Lambda_r$ to $\partial\Lambda_R$ which are of type $\sigma_1,\dots, \sigma_k$, when indexed in counterclockwise order. 
To avoid annuli with inner radii too small for arm events to occur, define $r_\sigma$ be the smallest $r$ such that $A_{\sigma}(r,R)$ is non-empty for every $R\ge r$.
We also introduce $A_{\sigma}^+(r,R)$ to be the same event as $A_{\sigma}(r,R)$, except that the paths must lie in the upper half-plane $\mathbb H$ and are indexed starting from the right-most. 

\subsection{Proof of Lemma~\ref{lem:mn}}\label{sec:4.1}\label{sec:5.1a}

We  insist on the fact that this part  only relies on the crossing estimates~\eqref{eq:wRSW}, not on Theorem~\ref{thm:sRSW}. 
With the notation of this section, we have $\pi_1^+(R)=\phi_{\bbH}^0[A^+_1(0,R)]$, and we will use the latter notation in this part.

Let $E_r$ be the event that $\Lambda_{2r}\setminus\Lambda_{r}$ contains an open path from $\partial\bbH$ to itself disconnecting 0 from infinity in $\bbH$.
Combining crossings in three rectangles, the FKG inequality together with the crossing estimates~\eqref{eq:wRSW} give that $\phi_{\bbH}^0[E_r]\ge c$ for every $r\ge1$. As a consequence,~\eqref{eq:FKG} implies that 
\begin{equation*}
  	\frac{\phi_{\bbH}^0[A^+_1(0,R)]}{\phi_{\bbH}^0[A^+_1(0,r)]}
	\ge \phi_\bbH^0[E_{r/2}]\phi_\bbH^0[E_{r}]\phi_\bbH^0[\Lambda_{r/2} \longleftrightarrow\partial\Lambda_{2r}]\phi_{\bbH}^0[A^+_1(r,R)]
	\ge c'\phi_{\bbH}^0[A^+_1(r,R)],
\end{equation*} 
so that we may focus on bounding the right-hand side from below.

For $|s|\le R/(2r)$, let $F_s$ be the event that there exist a path in $\omega$ and a path in $\omega^*$ going from the translate by $sr$ of $\Lambda_r$ to $\partial\Lambda_R$.  Then, if $[0,R/2]\times[0,R]$ is crossed vertically by a path in $\omega$, 
and $[-R/2,0]\times[0,R]$ is crossed vertically by a path in $\omega^*$, at least one event $F_s$ occurs. 
Moreover, due to~\eqref{eq:wRSW}, the two crossings mentioned above occur simultaneously with probability at least $c''>0$.
The FKG inequality~\eqref{eq:FKG} and the union bound give
\begin{equation}\label{eq:hhg}
	\phi_{\bbH}^0[A^+_1(r,R)]\phi_{\bbH}^0[A^+_0(r,R)]\ge \phi_{\bbH}^0[A^+_{10}(r,R)]\ge\max_s\phi_{\bbH}^0[F_s]\ge c'''\tfrac rR.
\end{equation}
Successive applications of the bound on the probability of $E_{2^k}$ with $r\le 2^k< R/2$ give
\begin{equation}\label{eq:hhgg}
	\phi_{\bbH}^0[A^+_0(r,R)]\le (1-c)^{\lfloor \log[R/(2r)]\rfloor}.
\end{equation}
Dividing~\eqref{eq:hhg} by~\eqref{eq:hhgg} concludes the proof.

\subsection{Quasi-multiplicativity, localization and well-separation}\label{sec:4.2}

Let us start with the classical notion of {\em well-separated arms}. In what is next, let $x_i$ and $x'_i$ be the end-points of the arm $\gamma_j$ on the inner and outer boundary of $\Lambda_R\setminus\Lambda_r$ respectively.
\begin{definition} Fix $\delta>0$. The arms $\gamma_1,\dots,\gamma_k$ are said to be $\delta$-\emph{well-separated} if
\begin{itemize}[noitemsep]
\item $x_1,\dots,x_k$ are at a distance larger than $2\delta r$ from each other.
\item $x_1',\dots,x_k'$ are at a distance larger than $2\delta R$ from each other.
\item For every $1\le i\le k$, $x_i$ is $\sigma_i$-connected to distance $\delta r$ of $\partial\Lambda_r$ in $\Lambda_{\delta r}(x_i)$.
\item For every $1\le i\le k$, $x_i'$ is $\sigma_i$-connected to distance $\delta R$ of $\partial\Lambda_R$ in $\Lambda_{\delta R}(x_i')$.
\end{itemize}
\end{definition}
Let $A_{\sigma}^{\rm sep}(r,R)$ be the event
that $A_{\sigma}(r,R)$ occurs and there exist arms realizing
$A_{\sigma}(r,R)$ which are $\delta$-well-separated. While it is not explicit in the notation, $A_{\sigma}^{\rm sep}(r,R)$ depends on~$\delta$.

\begin{proposition}[Well-separation]\label{prop:separation}
Fix $1\le q<4$ and $\sigma \in \{0,1\}^k$ for some $k$. Then for all $\delta >0$ small enough, there exists $c_6=c_6(\sigma,\delta,q)>0$ such that for every $R\ge r\ge r_\sigma$,
\begin{equation}
c_6 \phi_{\bbZ^2}[A_{\sigma}(r,R)]\le \phi_{\bbZ^2}[A_{\sigma}^{\rm sep}(r,R)]\le  \phi_{\bbZ^2}[A_{\sigma}(r,R)].
\end{equation}
\end{proposition}

\begin{proof}
With the help of Theorem~\ref{thm:sRSW}, the proof follows the same lines as for the random-cluster model with $q=2$. We refer to \cite{CheDumHon13} for details.
\end{proof}

As a first application of this result, we obtain the following.

\begin{proposition}[Quasimultiplicativity]\label{prop:quasimultiplicativity}
  Fix  $1\le q< 4$ and $\sigma$. There exist $c_7=c_7(\sigma,q)>0$ and $C_7=C_7(\sigma,q)>0$ such that for every $R\ge r\ge r_\sigma$,
 \begin{equation}\label{eq:MULTIPLICATIVITY}
c_7\,\phi_{\bbZ^2}[A_{\sigma}(r,R)] \le \phi_{\bbZ^2}[A_{\sigma}(r,\rho)]\phi_{\bbZ^2}[A_{\sigma}(\rho,R)] \le C_7\,\phi_{\bbZ^2}[A_{\sigma}(r,R)].
  \end{equation}
  \end{proposition}

\begin{proof}
With the help of Theorem~\ref{thm:sRSW}, the proof follows the same lines as for the random-cluster model with $q=2$. We refer to \cite{CheDumHon13} for details.
\end{proof}

For sequences $\sigma$ with no two 0 or 1 following each other when going cyclically, the previously available crossing probability estimates~\eqref{eq:RSW} are sufficient to derive the quasi-multiplicativity for $A_{\sigma}$, see \cite{Wu}. Nevertheless, obtaining the same result for other sequences relies crucially on 
Theorem~\ref{thm:sRSW}.  This is particularly important for the arm sequence $10101$ that is used repeatedly later on, see for instance the discussions of~\eqref{eq:UNIVERSAL5},~\eqref{eq:SIX-ARM}, and~\eqref{eq:FOUR-ARM}.

Another classical consequence of well-separation is the possibility of {\em localizing} the endpoints of arms.

\begin{definition}
Let $I=(I_i)_{1\leq i\leq k}$ and $J=(J_i)_{1\le i\le k}$ be two collections of disjoint intervals on the boundary of
the square $[-1,1]^2$, distributed in counterclockwise order. For a sequence $\sigma$ of length $k$, let $A^{I,J}_{\sigma}(r,R)$ be the event that $A_{\sigma}(r,R)$ occurs and the arms $\gamma_i$ can be chosen in such a way that $\gamma_i$ starts on $rI_i$ and ends on $RJ_i$ for every $1\leq i\leq k$.
\end{definition}

Since Theorem~\ref{thm:sRSW} generalises \cite{CheDumHon13} to every $1\le q<4$, we refer to the corresponding paper for the proof of the following result.
\begin{proposition}[Localization]\label{prop:localization}
Fix $1\le q<4$. For every $k\ge1$, every $I$ and $J$ as above, and every $\sigma$ of length $k$, there exists $c_8=c_8(\sigma,I,J,q)>0$ such that for every $R\ge r\ge r_\sigma$,
\begin{equation}
c_8\phi_{\bbZ^2}[A_{\sigma}(r,R)]\le \phi_{\bbZ^2}[A_{\sigma}^{I,J}(r,R)]\le\phi_{\bbZ^2}[A_{\sigma}(r,R)].
\end{equation}
\end{proposition}

\subsection{Bounds on the probability of arm events}\label{sec:proof_arm_events}

We begin with deriving up-to-constant estimates on three specific arm-events whose probabilities do not really vary when changing $q$.

\begin{proposition}[Universal arm-exponents]\label{prop:universal}
    Let $1\le q<4$. There exist $c_9,C_9>0$ such that for every $R\ge r\ge 1$,
    \begin{align}
    	c_9\left(r/R\right)^{2}\le\,&\phi_{\bbZ^2}[A_{10101}(r,R)]\le C_9\left(r/R\right)^{2},\label{eq:UNIVERSAL5}\\
    	c_9\left(r/R\right)^{2}\le\,&\phi_{\bbZ^2}[A_{101}^+(r,R)]\le C_9\left(r/R\right)^{2},\label{eq:UNIVERSAL3}\\
    	c_9\,r/R\,\le\,&\phi_{\bbZ^2}[A_{10}^+(r,R)]\le C_9\,r/R.\label{eq:UNIVERSAL2}  
    \end{align}
\end{proposition}

\begin{proof}The proof of \cite{CheDumHon13} extends trivially  to our setting using~\eqref{eq:MULTIPLICATIVITY} and Proposition~\ref{prop:localization}. 
\end{proof}

Next we study two consequences of~\eqref{eq:UNIVERSAL5}. The first concerns the six arm event. 

\begin{corollary}\label{cor:six arm}
  Fix  $1\le q< 4$. There exist $c_{10},C_{10}>0$ such that for every $R\ge r\ge 1$,
  \begin{align} &\phi_{\bbZ^2}[A_{101010}(r,R)] \le C_{10}\left(r/R\right)^{2+c_{10}}.\label{eq:SIX-ARM}
   \end{align}
\end{corollary} 

\begin{proof}
	This is a standard argument that we only sketch. 
	Conditionally on the first five arms, the probability that an additional dual arm exists decays at least as fast as  $(r/R)^c$ due to Theorem~\ref{thm:sRSW}. 
	Since the occurrence of the first five arms has a probability of  order $(r/R)^2$ as stated by~\eqref{eq:UNIVERSAL5},~\eqref{eq:SIX-ARM} follows. 
\end{proof} 

We now turn to an estimate on the probability of the four-arm event.
\begin{proposition}\label{prop:four arm}
  Fix  $1\le q< 4$. There exist $c_{11},c_{12}>0$ such that for every $R\ge r\ge 1$,
  \begin{align}
 &\phi_{\bbZ^2}[A_{1010}(r,R)]\ge   c_{11}\tfrac{r\, \pi_1^+(R)}{R\, \pi_1^+(r)}\ge c_{12}(r/R)^{2-c_{12}} \label{eq:FOUR-ARM}.
  \end{align}
\end{proposition} 

That the probability of the four-arm event is polynomially larger than that of the five-arm event, that is than $(r/R)^2$, is a standard consequence of Theorem~\ref{thm:sRSW}. 
It is noteworthy that~\eqref{eq:FOUR-ARM} aditionally provides an explicit bound for the probability of the four-arm event 
in terms of the probability of the half-plane one-arm event.

\begin{proof}
Fix $1\le r\le R$. 
Let $E$ be the event that $\La_{3R}$ contains both an open circuit and a dual open circuit surrounding $\La_{2R}$, with the open one being connected to $\partial \La_{4R}$. 
By the crossing estimates~\eqref{eq:wRSW}, $\phi_{\bbZ^2}[E]\ge c>0$.
Let $\calD$ be the connected component of $0$ in the set of vertices not connected to $\partial\Lambda_{3R}$ (to be more precise the largest subdomain containing 0). Observe that when $E$ occurs, $\calD$ is $R$-centred. 
Moreover, conditionally on $\calD$ and on the configuration outside of it, the measure inside $\calD$ is $\phi_{\calD}^0$.

Let $A_{1010}(x,r,R)$ be  the translate by the vector $x \in \bbZ^2$ of the event $A_{1010}(r,R)$.
Then, for any $r$-box $\La_r(x)$ that intersects $\calD^c$ and is connected to $\Lambda_R$ in $\calD$, $A_{1010}(x,r,R)$ occurs. 
Indeed, the two arms of type 0 are given by $\partial \calD$, 
one arm of type 1 is given by the fact that any vertex of $\calD^c$ neighbouring a vertex of $\calD$ is connected to $\partial \La_{4R}$ outside of $\calD$, 
and the second arm of type 1 is given by the connection between $\La_r(x)$ and $\La_R$.
Thus
\begin{equation}\label{eq:pro}
    \sum_{x\in r\mathbb Z^2 \cap \La_{3R}}\phi_{\bbZ^2}[A_{1010}(x,r,R)]
    \ge \phi_{\bbZ^2}\big[\phi_{\calD}[\mathbf M_r(\calD,R)]\,\ind_E\big]
    \ge c\,M(r,R).
\end{equation}
Proposition~\ref{prop:10} and Lemma~\ref{lem:mn}  conclude the proof since there are $C(R/r)^2$ terms in the sum above, all equal to $\phi_{\bbZ^2}[A_{1010}(r,R)]$. 
\end{proof}

Using the parafermionic observable, when $q\in [1,3]$, 
the previous lower bound on the probability of the four arm event may be transformed as follows. 

\begin{proposition}\label{prop:beta}
	For every $1\le q\le 3$,  there exists $c_{13}>0$ such that for every $R\ge1$,
    \begin{align}\label{eq:bbb}
    	\phi_{\bbZ^2}[A_{1010}(1,R)]&\ge c_{13}R^{-2 + c_{13}}\tfrac1{\phi_{\bbZ^2}[A_1(0,R)]}.
	\end{align}
\end{proposition}

The above inequality is interesting from two points of view. 
First, it can be used to prove that the density of the infinite cluster $\theta(p)$ is not Lipschitz near $p_c$ (see \cite{DumMan20}).
Second, it is a necessary condition fo the Glauber dynamics to have exceptional times (we refer to \cite{GPS} and references therein for details). 
Let us mention that~\eqref{eq:bbb} is expected to fail for $q$ close to~$4$.

\begin{proof}
By Proposition~\ref{prop:four arm}, it suffices to prove the existence of $c>0$ such that for every $R\ge1$,
\begin{equation}\label{eq:h21}
	R\, \phi_{\bbH}^0[A^+_1(0,R)]\, \phi_{\bbZ^2}[A_1(0,R)]\ge c\,R^c.
\end{equation}

In order to do so, we use the parafermionic observable. 
Consider the Dobrushin domain $\Omega_R$ obtained from $\Lambda_{3R}$ by removing the vertices $(x,0)$ with $x\ge1$ (call this the {\em slit}), 
with $a=b=0$ (in this case $e_a$ and $e_b$ are the medial edges right of the origin, 
and the exploration path is simply the loop passing through $e_a$ and $e_b$); see Fig.~\ref{fig:Omega_R}. 
We now apply~\eqref{eq:rel_vertex} to $\Omega_R$.
The set $\calC$ of boundary medial edges is split into three parts: 
the set $\{e_a,e_b\}$, the set $\alpha$ of edges that are on the boundary of $\Lambda_{3R}^\diamond$, and the set $\beta$ of remaining edges, which are above and below the slit. Proceeding as in Lemma~\ref{lem:op}, we have
\begin{equation}\label{eq:5.15}
	2\sum_{x\in \partial\Lambda_{3R}}\phi_{\Omega_R}^0[0\longleftrightarrow x]
	\ge\sum_{e\in\alpha}|F(e)|\ge \big|\sum_{e\in\beta}\eta(e)F(e)+\eta(e_a)F(e_a)+\eta(e_b)F(e_b)\big|.
\end{equation}
A careful computation (along the same lines as that leading to~\eqref{eq:h12}, and using the vertical symmetry of $\Omega_R$)
shows that the two complex numbers $\sum_{e\in\beta}\eta(e)F(e)$ and $\eta(e_a)F(e_a)+\eta(e_b)F(e_b)$ are collinear; see also the explanation of Fig.~\ref{fig:Omega_R}. 
Moreover, when $1 < q\leq 3$ (which is to say $\sigma \in (1/3,2/3]$), they also have the same direction, 
which implies 
\begin{equation*}
	\big|\sum_{e\in\beta}\eta(e)F(e)+\eta(e_a)F(e_a)+\eta(e_b)F(e_b)\big|
	\ge |\eta(e_a)F(e_a)+\eta(e_b)F(e_b)|= 2\cos \tfrac{\pi}4(3\sigma-1)  >0.
\end{equation*}
The two last displayed inequalities imply
\begin{equation}\label{eq:5.17}
	\sum_{x\in\partial\Lambda_{3R}}\phi_{\Omega_R}^0[0\longleftrightarrow x]\ge \cos \tfrac{\pi}4(3\sigma-1).
\end{equation}

Now, for 0 to be connected to $x \in \partial \La_{3R}$, 0 must be connected to $\partial\Lambda_R$ and $x$ to $\partial\Lambda_R(x)$. 
Thus, using the mixing property~\eqref{eq:mix2},
\begin{equation*}
	\phi_{\Omega_R}^0[0\longleftrightarrow x]
	\le C\,\phi_{\Omega_R}^0[0\longleftrightarrow\partial\Lambda_R]\,\phi_{\Omega_R}^0[x\longleftrightarrow\partial\Lambda_R(x)]
	\le C'\,\phi_{\bbZ^2}[0\xleftrightarrow{\Omega_R}\partial\Lambda_R] \, \phi_{\bbH}^0[A^+_1(0,R)].
\end{equation*}
The second inequality holds since $x$ is on $\partial\Lambda_{3R}$, and therefore the boundary conditions induced by $\Omega_R$ are dominated by the free boundary conditions on a half-plane with $x$ on its boundary. 
Plugging the above into~\eqref{eq:5.17} yields
\begin{equation*}
	R \, \phi_{\bbZ^2}[0\xleftrightarrow{\Omega_R}\partial\Lambda_R] \, \phi_{\bbH}^0[A^+_1(0,R)] \geq c'.
\end{equation*}
Thus, in order to prove ~\eqref{eq:h21}, it suffices to show that 
\begin{equation}\label{eq:yt}
	\phi_{\bbZ^2}[0\xleftrightarrow{\Omega_R}  \partial\Lambda_R]\le R^{-c''}\phi_{\bbZ^2}[A_1(0,R)],
\end{equation}
or in words, having one arm in a slit box is polynomially harder than having one arm in a box.
This is an easy consequence of the crossing estimates~\eqref{eq:wRSW}. 
Indeed, let $A_{k}$ be the event that for every even integer $\ell\le k$, 
there exists no dual path in $\La_{2^{\ell+1}}\setminus \La_{2^\ell}$ that disconnects $0$ from $\partial\Lambda_R$ inside $\Omega_R$. 
If $0$ is connected to $\partial \La_R$ in $\Omega_R$, then $A_{\lfloor \log_2R\rfloor}$ necessarily occurs. Thus,  
\begin{equation}\label{eq:yyt}
	\phi_{\bbZ^2}[0\xleftrightarrow{\Omega_R}  \partial\Lambda_R]\le \phi_{\bbZ^2}[0\longleftrightarrow\partial\Lambda_R,A_{\lfloor \log_2R\rfloor}].
\end{equation}
Now, the crossing estimates~\eqref{eq:wRSW} imply that 
$$\phi_{\bbZ^2}[0\longleftrightarrow\partial\Lambda_R,A_{k+1}]\le (1-c''')\phi_{\bbZ^2}[0\longleftrightarrow\partial\Lambda_R,A_{k}].$$
Inequality~\eqref{eq:yt} follows readily by applying the above $\lfloor\log_2 R\rfloor$ times and using~\eqref{eq:yyt}. This concludes the proof.
\end{proof}

\subsection{New bounds on the one, two and four-arm exponents}\label{sec:perco}

\begin{proposition}\label{lem:perco}	
    Fix $1\le q<4$. There exists $c_{14}>0$ such that
    \begin{align}
        \phi_{\bbZ^2}[A_1(0,R)] &\geq c_{14}\,\pi_1^+(R)^{1/2},\label{eq:Bernoulli_exponents1}\\
        \phi_{\bbZ^2}[A_{10}(0,R)] &\geq c_{14}\,\pi_1^+(R),\label{eq:Bernoulli_exponents2} \\
        \phi_{\bbZ^2}[A_{1010}(1,R)] &\geq c_{14}\,\pi_1^+(R)/R.\label{eq:Bernoulli_exponents4}
    \end{align}
\end{proposition}

\begin{proof}
	The first inequality follows from the second one using the FKG inequality~\eqref{eq:FKG}. 
	The third is the conclusion of Proposition~\ref{prop:four arm} with $r = 1$. 
	Therefore, it only remains to prove~\eqref{eq:Bernoulli_exponents2}.

	Consider the Dobrushin domain $\Lambda_R$, with $a$ and $b$ being the bottom right corner and top left corner of $\La_R$, respectively; see Fig.~\ref{fig:Omega_R}.
    We will proceed as in the proof of Lemma~\ref{lem:op} (and therefore only sketch the proof).
    Instead of working with the contour $\calC$ which runs along the boundary of $\La_R$, 
    we will work with the contour $\calC'$ that surrounds the vertices of the medial lattice which lie below the diagonal $x = -y$. 
    The medial edges of $\calC'$ may be split into those adjacent to the diagonal (call this set $\alpha$) and those adjacent to $\partial \La_R$ (call this set $\beta$).
    By summing~\eqref{rel_vertex} over every vertex of the medial lattice which lies strictly inside $\calC'$, we find
    $$\sum_{e\in \calC'}\eta(e)F(e)=0.$$
	Using the triangular inequality, we obtain that
    $$\sum_{e\in \alpha}|F(e)|\ge |\sum_{e\in \beta}\eta(e)F(e)| \geq  c\,R\,\pi_1^+(R),$$
    where the second inequality was already proved in Lemma~\ref{lem:op}.
    
    Notice now that, for any $e \in \alpha$, a configuration contributes to $F(e)$ only when the primal and dual vertices separated by $e$ are connected inside $\La_{R}$ 
    to $(ba)$ and $(ab)^*$ by primal and dual paths, respectively. 
    Thus, due to the mixing property~\eqref{eq:mix2}, if $r$ denotes the distance from $e$ to $\partial \La_R$, then $|F(e)|\leq C \phi_{\bbZ^2}[A_{10}(0,r)]$. 
	In conclusion, 
	\begin{align*}
		\sum_{r = 1}^R \phi_{\bbZ^2}[A_{10}(0,r)]\geq c' \sum_{e\in \alpha}|F(e)|\ge c\, c'\, R\,\pi_1^+(R).
	\end{align*}
	Finally, it is a classic consequence of the quasi-multiplicativity~\eqref{eq:MULTIPLICATIVITY}  and the bound $\phi_{\bbZ^2}[A_{10}(r,R)]\ge (r/R)^{1-c''}$
	(which follows from~\eqref{eq:hhg} by standard arguments)
	that the left-hand side of the above is bounded from above by $c''' R\,\phi_{\bbZ^2}[A_{10}(0,R)]$. 
	Plugging this estimate in the last displayed equation gives~\eqref{eq:Bernoulli_exponents2}.\end{proof}

This proposition implies that one deduces bounds on the left-hand sides of the three inequalities from bounds on $\pi_1^+(R)$. For $q=1$, \cite{PonIkh} showed that $\pi_1^+(R)R^{1/3}$ is bounded away from 0 and $\infty $ uniformly in $R$ so that 
\begin{align*}
   \phi_{\bbZ^2}[A_1(0,R)] &\geq c/R^{1/6},\\
   \phi_{\bbZ^2}[A_{10}(0,R)] &\geq c/R^{1/3},\\
   \phi_{\bbZ^2}[A_{1010}(1,R)] &\geq c/R^{4/3}.\\
\end{align*}
While the result of \cite{PonIkh} is sharp, the bound we obtain are not  (see \cite{BefDum13} for references on the case of site percolation on the triangular lattice to compare to the following bounds).
Note also that it is  elementary to show from~\eqref{eq:hhg} and~\eqref{eq:UNIVERSAL2} that for $q=1$, $\pi_1^+(R)\ge c/R^{1/2}$,
so that
    \begin{align*}
        \phi_{\bbZ^2}[A_1(0,R)] &\geq c/R^{1/4},\\
        \phi_{\bbZ^2}[A_{10}(0,R)] &\geq c/R^{1/2},\\
        \phi_{\bbZ^2}[A_{1010}(1,R)] &\geq c/R^{3/2}.
    \end{align*}
We conclude this section by a proof that the inequality $\pi_1^+(R)\ge c/R^{1/2}$ is in fact valid for every $q\in[1,2]$, thus extending the previous bounds to this context.

\begin{proposition}\label{prop:one_arm1/2}
    For $q\in[1,2]$, there exists $c_{15}=c_{15}(q)>0$ such that for every $R\ge1$,
    \begin{equation*}
    	\pi_1^+(R)\ge c_{15}R^{-1/2}.
    \end{equation*}
\end{proposition}

\begin{proof}
We apply the parafermionic observable to the graph $\Omega_R:=\bbZ\times[0,2R]$ with $a=b=0$. 
Using that the contour integral on the boundary vanishes, 
and following the same lines as when going from~\eqref{eq:5.15} to~\eqref{eq:5.17}, we find that
\begin{equation}\label{eq:5.30}
\sum_{x\in\bbZ\times\{2R\}}\phi_{\Omega_R}^0[0\longleftrightarrow x]
\ge \tfrac12|\eta(e_a)F(e_a)+\eta(e_b)F(e_b)|= \cos \tfrac{\pi}4(3\sigma - 1)  >0.
\end{equation}
At this stage we used that $1 \leq q \leq 2$ and the horizontal symmetry of the strip to show that the contribution to the contour integral of medial edges on the bottom of $\Omega_R$ is positively proportional to that of $e_a$ and $e_b$. 

Now, the mixing property~\eqref{eq:mix2} and crossing estimates~\eqref{eq:wRSW} easily lead to the existence of $c,C\in (0,1)$ such that 
\begin{equation*}
	\phi_{\Omega_R}^0[0\longleftrightarrow x]\le C \pi_1^+(R)^2 c^{|x|/R},
\end{equation*}
where the second term accounts for vertices $x\in\bbZ\times\{2R\}$ that are far on the left or right. Plugging this estimate in~\eqref{eq:5.30} gives
\begin{equation*}
	C'R\pi_1^+(R)^2\ge c',
\end{equation*}
and therefore the claim.
\end{proof}

\begin{figure}
	\begin{center}
		\includegraphics[width = 0.45\textwidth]{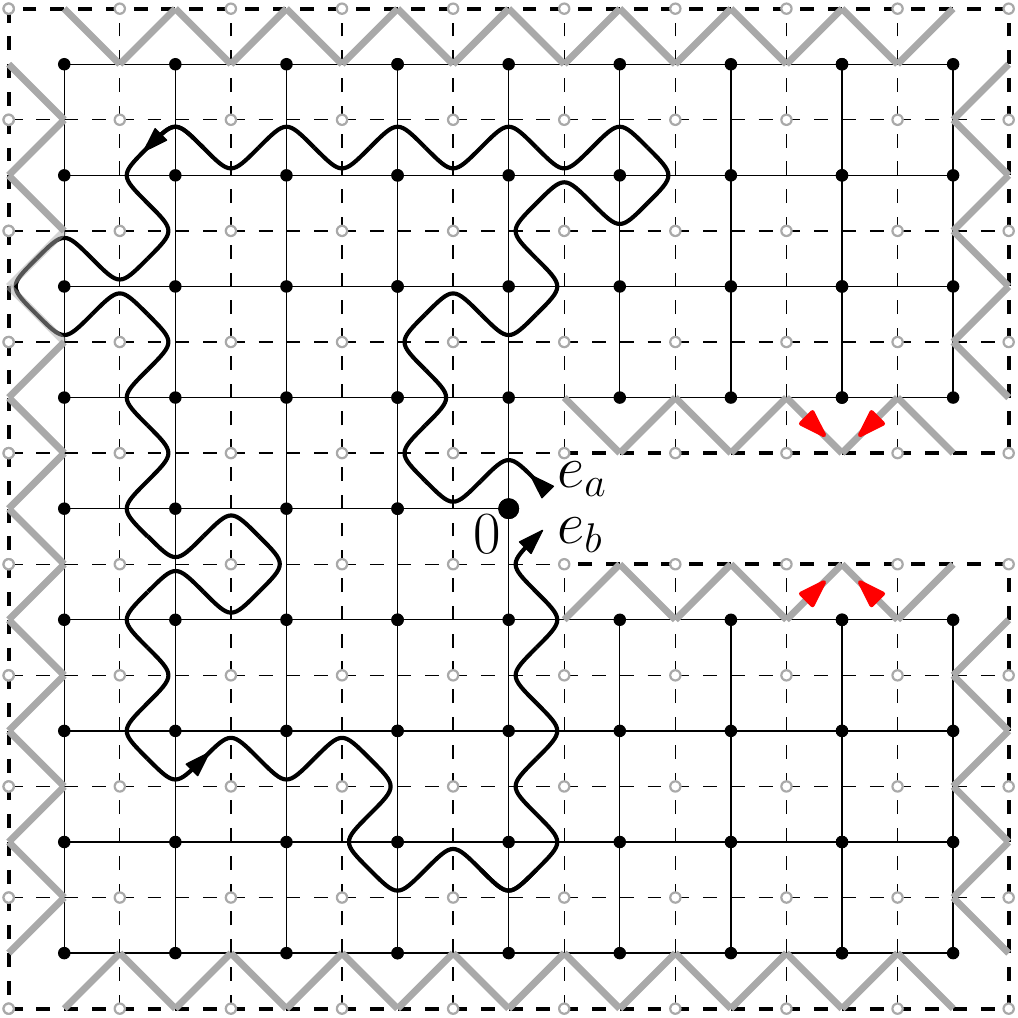}\hspace{0.05\textwidth}
		\includegraphics[width = 0.45\textwidth]{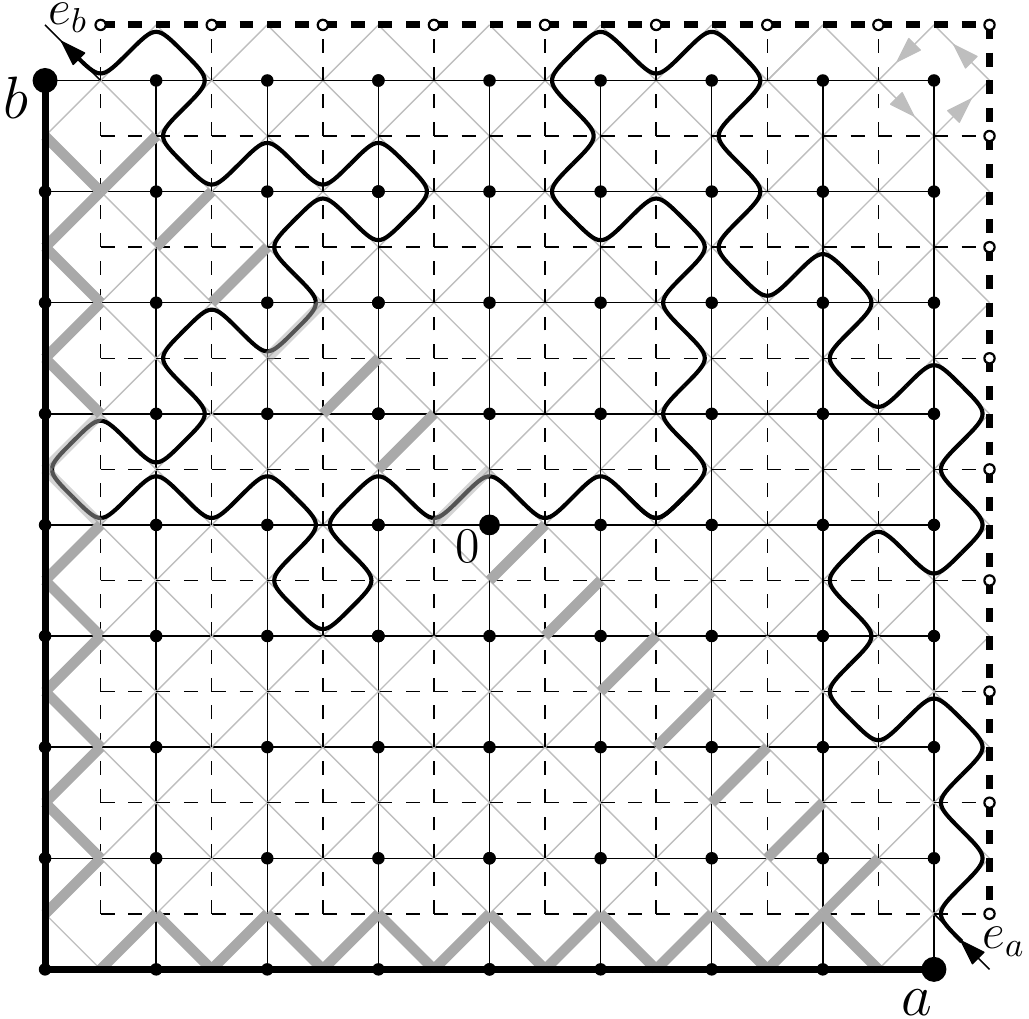}
	\caption{{\em Left:} The domain $\Omega_R$ used in the proof of Proposition~\ref{prop:beta} is obtained from $\La_{3R}$ by removing the slit right of $0$. 
	The medial edges $e_a$ and $e_b$ are right of $0$; all other edges of $\calC$ are bold. 
	The red arrows indicate the orientation of $\eta(e)$ for four edges of $\beta$. 
	Due to the symmetry of $\Omega_R$, the absolute value of $F(e)$ for these four edges is equal;
	their complex arguments are
	${ \frac{7\pi}4(\sigma+1)}$, ${- \frac{7\pi}4(\sigma + 1)}$, 
	${ \frac{5\pi}4(\sigma+1)}$ and ${- \frac{5\pi}4(\sigma+1)}$, respectively (up to an additive constant).
	{\em Right:} The Dobrushin domain used in the proof of Proposition~\ref{lem:perco}; the medial edges of the contour $\calC'$ are bold. }
	\label{fig:Omega_R}
	\end{center}
\end{figure}

\subsection{Bounds on the scale-to-scale connection probability  in a half-plane with free boundary conditions}

The goal of this section is the following result, which is a refinement and an extension of the previous proposition.

\begin{proposition}
    For every $1\le q<2$, there exists $c=c(q)>0$ such that for every $r\le R$,
    \begin{equation}
    	\phi_\bbH^0[A_1(r,R)]\ge c(r/R)^{1/2-c}.
    \end{equation}
    For every $2<q<4$, there exists $c=c(q)>0$ such that for every $r\le R$,
    \begin{equation}
	    \phi_\bbH^0[A_1(r,R)]\le \tfrac1c(r/R)^{1/2+c}.
    \end{equation}
\end{proposition}

\begin{figure}
	\begin{center}
		\includegraphics[width = 0.85\textwidth]{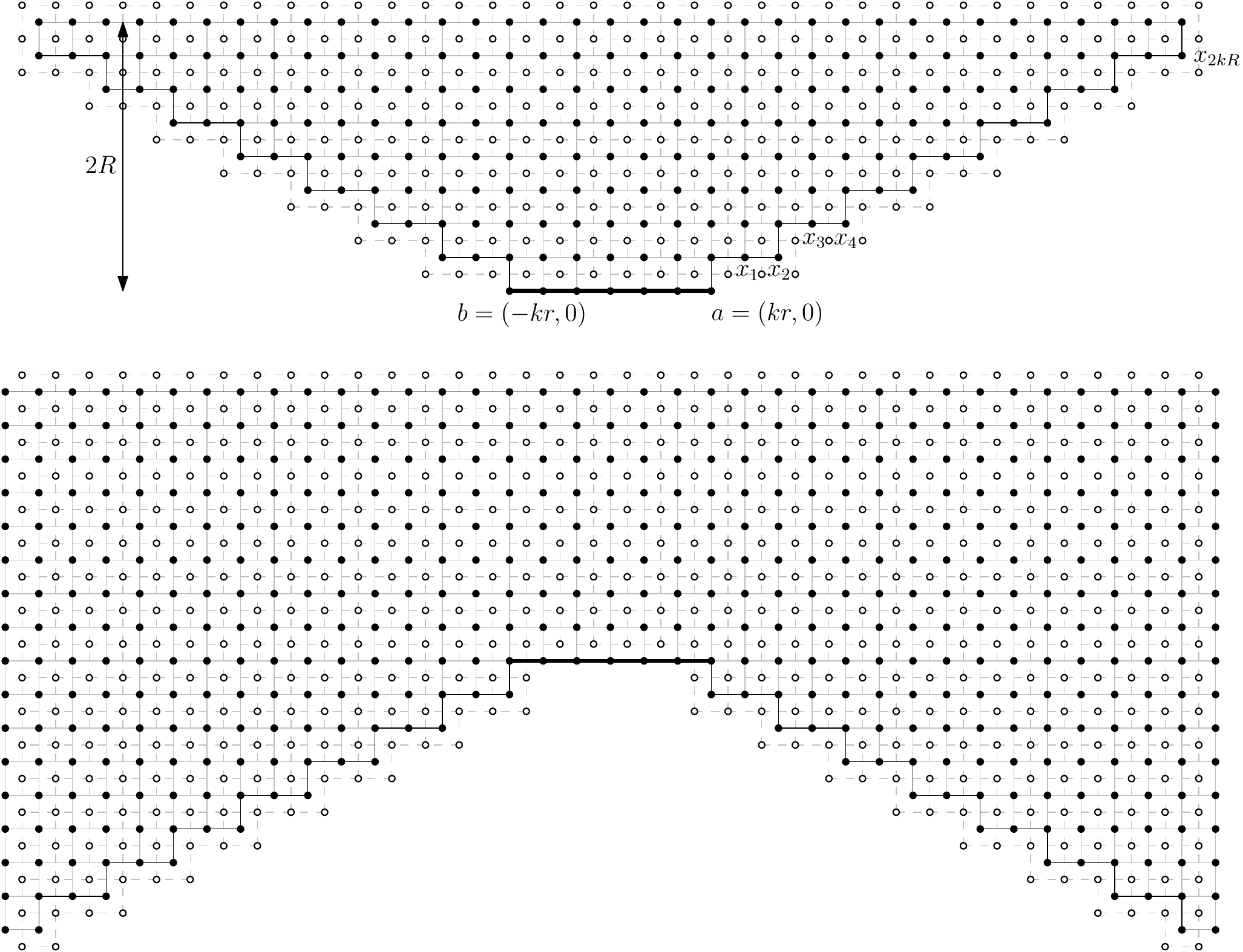}
		\caption{The domains $\Omega$ (above) and $\Omega'$ (below) used for $q<2$ and $q>2$, respectively, with $k = 2$.  
		Outside of the depicted section, $\Omega'$ is equal to the strip $\bbZ \times [-2R,2R]$. }	
		\label{fig:stair}
	\end{center}
\end{figure}

\begin{proof}
We start with the case $1\le q<2$. We improve on the proof of Proposition~\ref{prop:one_arm1/2} by choosing a better domain. 
Fix some integer $k \ge 1$ (which will be chosen later independently of $r,R$) and integers $r \le R$.
Consider the domain $\Omega_{R,r,k}=\Omega$ defined as the subset of the strip $\mathbb Z\times [0,2R]$ composed of vertices above the two half-lines $(0,-r )+ie^{\pm i\alpha}\bbR_+$ where $\tan\alpha:=1/k$; see Fig.~\ref{fig:stair}. 
We consider the Dobrushin boundary conditions $0/1$ with $a=(kr,0)$ and $b=(-kr,0)$, which are wired on the segment $(ba):=\bbZ\cap\partial\Omega$, and free elsewhere.  

This domain is an approximation of a (horizontal symmetry of a) trapeze. 
Index the vertices of $\Omega$ that lie strictly inside $\bbZ_+\times[0,2R]$ and that have at most three neighbours in the $\Omega$ as $x_1,\dots,x_{2kR}$, from bottom to top, as in Fig.~\ref{fig:stair}. 
This indexing ensures that for each $n =0,\dots, 2R-1$, $x_{kn+1},\dots,x_{k(n+1)}$ form a horizontal segment at height $n+1$.

Using \eqref{eq:rel_vertex} for $\Omega$ and its vertical symmetry, after an appropriate change of phase, we find 
\begin{align*}
	\sum_{x\in\bbZ\times\{2R\}}\phi_{\Omega}^{0/1}[x\longleftrightarrow (ba)]
	=c_q^{(ba)}\Big(1+\!\!\!\sum_{u\in (ba)^*}\phi_{\Omega}^{0/1}[u\stackrel{*}\longleftrightarrow (ab)^*]\Big)
	+\sum_{i=1}^{2kR}c_q(i)\phi_{\Omega}^{0/1}[x_i\longleftrightarrow (ba)],
\end{align*}
with constants $c_q^{(ba)} >0$ and $c_q(i)$ for $i \geq 1$ given by 
\begin{equation}
    c_q(i):=\begin{cases}
	    c_q^+&\text{ if }k\text{ does not divide }i,\\
    	c_q^-&\text{ if }k\text{ divides }i,
	    \end{cases}
\end{equation}
with $c_q^+>0$ and $c_q^-<0$.
We omit the details of this computation as it is similar to those previously performed. Note however that the signs of $c_q^+$ and $c_q^-$ are due $q$ being strictly smaller than $2$. 
Now, choose $k=k(q)$ large enough (independent of $r$ and $R$) that for every $n$,
\begin{align*}
	\sum_{i=kn+1}^{k(n+1)}c_q(i)\phi_{\Omega}^{0/1}[x_i\leftrightarrow (ba)]
	\ge 0.
\end{align*}
This may be done since, due to Theorem~\ref{thm:sRSW}, 
the probability that $x_{k(n+1)}$ is connected to $(ba)$ is much smaller than the probability that one of the $x_i$ with $kn<i<k(n+1)$ is.

Using this fact and Theorem~\ref{thm:sRSW}, we obtain that
\begin{align*}
	CR\pi_1^+(R) \phi_{\Omega}^{0/1}[(ba)\longleftrightarrow \bbZ\times\{R\}]
	&\ge  \sum_{x\in\bbZ\times\{2R\}}\phi_{\Omega}^{0/1}[x\longleftrightarrow (ba)]\\
	&\ge c_q^{(ba)}\Big(1+\!\!\!\sum_{u\in (ba)^*}\phi_{\Omega}^{0/1}[u\stackrel{*}\longleftrightarrow (ab)^*]\Big)\\
	&\ge c \sum_{\ell=1}^r\pi_1^+(\ell),
\end{align*}
for constants $c,C >0$ independent of $r$ and $R$. 
Finally, a standard application of Theorem~\ref{thm:sRSW} implies the existence of $c' = c'(k)>0$ such that 
\[
\phi_{\Omega}^{0/1}[(ba)\leftrightarrow \bbZ\times\{R\}]\le c'(r/R)^{c'}\pi_1^+(r,R).
\]
Combined with the above, this implies that
\[
C'R\pi_1^+(R)\pi_1^+(r,R)(r/R)^{c'}\ge  \sum_{\ell=1}^r\pi_1^+(\ell)\ge r\pi_1^+(r).
\]
The claim follows by quasi-multiplicativity.
\bigbreak

We now turn to the case $2<q< 4$. 
We proceed in a similar fashion, but using a slightly different domain. 
Define the domain $\Omega_{R,r,k}'=\Omega'$ obtained as the reflection with respect to the horizontal axis of $\bbZ\times[-2R,2R]\setminus\Omega_{R,r,k}$; see Fig.~\ref{fig:stair}. 
A similar reasoning to before, this time using that the contribution on the lateral sides is negative for $k= k(q)$ large enough, implies that 
\begin{align*}
    \sum_{x\in\bbZ \times \{2R\}}\phi_{\Omega'}^{0/1}[x\longleftrightarrow (ba)]
    \le  c_q^{(ba)}\Big(1+\!\!\!\sum_{u\in (ba)^*}\phi_{\Omega'}^{0/1}[u\longleftrightarrow (ab)^*]\Big).
\end{align*}
In this new domain, Theorem~\ref{thm:sRSW} provides a constant $c = c(k) > 0$ such that 
$$ \phi_{\Omega'}^{0/1}[(ba)\leftrightarrow \bbZ\times\{R\}]\ge c(R/r)^{c}\pi_1^+(r,R).$$
The same reasoning as above yields positive constants $c', C', C''$ such that
\begin{align*}
	c'R\pi_1^+(R) (R/r)^{c}\pi_1^+(r,R)
	\le \sum_{x\in\bbZ \times \{2R\}}\phi_{\Omega'}^{0/1}[x\longleftrightarrow (ba)]
	\le C'  \sum_{k=1}^r\pi_1^+(k)
	\le C'' r\pi_1^+(r),
\end{align*}
with the last inequality due to Lemma~\ref{lem:mn}.
The proof follows by quasi-multiplicativity.
\end{proof}

\section{Properties of any sub-sequential limit}\label{sec:6}

In this section, we describe properties of sub-sequential limits of the family of cluster boundaries of the critical random-cluster model.

\subsection{Existence of sub-sequential limits}

To start, we recall the tightness criterion of \cite{AizBur99} for families of interfaces, formulated here for the random-cluster measure. 
The criterion may be shown to hold using the pre-existent crossing estimate~\eqref{eq:RSW}.

Let $\Omega$ be an open subset of the plane, and define $\Omega_\delta$ as the subgraph of $\delta\bbZ^2$ induced by the edges included in $\Omega$. 
Let $\phi_{\Omega_\delta}^0$ and $\phi_{\Omega_\delta}^1$ be the critical random-cluster measures on $\Omega_\delta$ with free and wired boundary conditions respectively. Also, let $\mathcal F_\delta$ be the collection of interfaces between the primal and dual clusters in $\Omega_\delta$. 

\begin{theorem}[\cite{AizBur99}]\label{thm:tightness}
	Fix $i \in \{0,1\}$ and $q \in [1,4]$. 
	Suppose that for each $k\geq 2$ there exist constants $C(k)>0$ and $\lambda(k)>0$, with $\la(k)$ tending to infinity with~$k$ such that, 
	for all $\delta > 0$ and any annulus $\Lambda_R(x)\setminus\Lambda_r(x) \subset \Omega$ with $\delta\le r\le R\le 1$,
	\begin{equation}\tag{{\bf H1}}\label{eq:H1}
		\phi_{\Omega_\delta}^i[\Lambda_R(x)\setminus\Lambda_r(x)\text{ is traversed by $k$ separate paths of }\mathcal F_\delta]\le C(k)(\tfrac rR)^{\lambda(k)}.
	\end{equation}
	Then the random variables $(\mathcal F_\delta)_{\delta > 0}$ form a tight family for the Hausdorff metric on collections of loops (see \cite{AizBur99} for a definition).
	Moreover, there exists $c_{16}>0$ such that any sub-sequential limit of the variables above (for the convergence in distribution) 
    is supported on collections of loops which have Hausdorff dimension between $1+c_{16}$ and $2-c_{16}$.
\end{theorem}

As mentioned above, it is a consequence of~\eqref{eq:RSW} that~\eqref{eq:H1} is satisfied for both $i=0$ and $i = 1$, and any $q \in [1,4]$.
Indeed, for $\Lambda_R(x)\setminus\Lambda_r(x)$ to be traversed by $k$ separate crossings of $\mathcal F_\delta$, 
a $k$-alternating arm event needs to occur in the annulus around $x$. 
It is standard to deduce from~\eqref{eq:RSW} that the probability of such an event is bounded as required, 
uniformly in $r,R$ and the boundary conditions on the annulus.

\begin{remark}
A close inspection of the proof of \cite{AizBur99} shows that it suffices to have the existence of $k$ such that $\lambda(k)>2$. 
We deduce from~\eqref{eq:SIX-ARM} that $k=6$ works in our setting. 
See \cite{KemSmi12} for alternative criteria that are implied by Theorem~\ref{thm:sRSW}.
\end{remark}

\begin{remark}
	The argument of \cite{AizBur99} shows that the interfaces are naturally fractal. 
	One implication of this fact is the existence of $c_{17}>0$ such that 
	\begin{equation}\label{eq:2-arm bound}
		\phi_{\bbZ^2}[A_{10}(r,R)]\ge c_{17}(r/R)^{1-c_{17}}.
	\end{equation}
\end{remark}

In the next sections we derive from Theorem~\ref{thm:sRSW} properties relating to sub-sequential limits for the family of interfaces.

\subsection{Large clusters touch the boundary many times}

In this section, we improve on Theorem~\ref{thm:sRSW} to show that large clusters touch the boundary of a domain at all scales and in many places. 
To simplify the statements and illustrate this slightly informal claim, we choose the context of $R$-centred domains and formulate the result as follows. 

\begin{proposition}\label{prop:polynomially many}
    For $1\le q<4$, there exists $c_{18}>0$ such that for every $R\ge r\ge 1$ and every $R$-centred domain, 
    \begin{equation}\label{eq:polynomially many}
    	\phi_\calD^0[\Lambda_R\text{ is connected in $\Lambda_{9R}$ to $(R/r)^{c_{18}}$ $r$-boxes intersecting }\partial\calD]\ge c_{18}.
    \end{equation}
\end{proposition}

There are several ways of obtaining this result. 
One is to use an argument involving exploration, along with the fact that $p(R)$ is uniformly bounded away from $0$ (see Section~\ref{sec:idea}).
Here, we take a more direct approach based on our proof of Proposition~\ref{prop:renormalization}.

\begin{figure}
    \begin{center}
    \includegraphics[width = 0.5\textwidth]{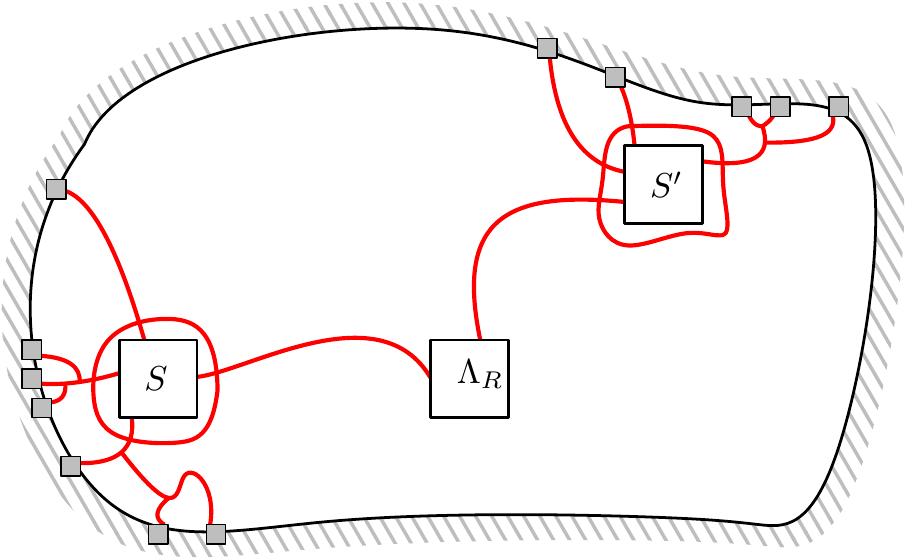}
    \caption{When the $R$-seeds $S$ and $S'$ are each connected to $k$ $r  $-boxes intersecting $\partial\calD$ within the box of radius $9R$ around them, 
    and when $\La_R$ is connected to circuits around both $S$ and $S'$, then $\La_{R}$ is connected to at least $2k$ $r  $-boxes intersecting $\partial\calD$. 
    Indeed, the boxes associated to $S$ are disjoint from those associated to $S'$ because the two seeds were chosen far from each other. }
    \label{fig:polynomially_many}
    \end{center}
\end{figure}

\begin{proof}
We will use the same notation as in the proof of Proposition~\ref{prop:renormalization}.
First, we claim that the quantity
$$
p_{r  ,k}(R):=\inf_{\calD \text{ $R$-centred}}\phi_\calD^0[\Lambda_R\text{ is connected in $\Lambda_{9R}$ to $k$  $r  $-boxes intersecting }\partial\calD]
$$
satisfies
\begin{equation}\label{eq:ahgg}
	p_{r  ,k}(R)\ge c_2M(r,R)\min\{p_{r  ,k}(r),(\tfrac rR)^2\} \qquad \text{for all $R \geq 20r$}.
\end{equation}
Indeed, the same proof as for $p(R)$ applies, 
with $F_S$ replaced by the event $F_S(r  ,k)$ that $S$ is connected in $\La_{9R}$ to $k$ $r  $-boxes intersecting $\partial\calD$.

Second, we claim that there exists $c > 0$ independent of $R$, $r  $ and $k$, such that
\begin{align}\label{eq:ahg}
	p_{r  ,2k}(20R)\ge c\, p_{r  ,k}(R)^2.
\end{align}
To see this, consider a $20R$-centred domain $\calD$ and two $R$-seeds $S$ and $S'$ that are at a distance at least $20R$ of each other, 
but within distance $30R$ of $\La_R$.  
If the events $\Circ_S\cap E_S\cap F_S(r  ,k)\cap \Circ_{S'}\cap E_{S'}\cap F_{S'}(r  ,k)$ occur, 
then $\Lambda_{20R}$ is connected in $\Lambda_{120R}$ to at least $2k$ $r  $-boxes intersecting $\partial\calD$
(see Fig.~\ref{fig:polynomially_many} and its caption for more details). 
Using the FKG inequality~\eqref{eq:FKG}, the crossing estimates~\eqref{eq:wRSW}, the definition of $p_{r  ,k}(R)$, 
and then taking the infimum over $20R$-centred domains, the above leads to~\eqref{eq:ahg}. 

Proceeding as in the proof of Proposition~\ref{prop:crucial_exist}, 
we may choose a constant $\lambda$ independent of $R$, $r  $ and $k$,
large enough that 
\begin{align}\label{eq:Delta2}
	p_{r  ,k}(\lambda R)\ge 2\min\{ p_{r  ,k}(R), \lambda^{-2}\}\qquad \text{ for all $R \geq 20r$}.
\end{align}
Moreover, due to~\eqref{eq:wRSW},
 we may assume that $\lambda$ is also such that $\inf_r   p_{r  ,1}(20r  ) \geq \lambda^{-2}$.

Suppose now that 
$p_{r  ,k}(R) \geq \lambda^{-2}$ for some $r$, $R$ and $k$. 
Then $j:= \lceil 2 \log_2 \lambda/c\rceil $ applications of~\eqref{eq:Delta2} followed by one application of~\eqref{eq:ahg} yield
\begin{align*}
	p_{r  ,2k}(20 \lambda^j R)\ge \min\{2^{j} p_{r  ,2k}(20 R), 2\lambda^{-2}\} \ge 2^j \, c \lambda^{-4}\ge \lambda^{-2}.
\end{align*}
The above together with the bound on $p_{r  ,1}(20r)$ implies the existence of $c' > 0$ 
such that $\displaystyle \inf_{r, R} p_{r  , (R/r  )^{c'}}(R) > 0$, which is the desired conclusion. 
\end{proof}

 \subsection{Large clusters touch each other}

Theorem~\ref{thm:tightness} implies that sub-sequential limits $\calF$ of collections of loops $\calF_\delta$ exist, 
but does not guarantee that the loops of $\calF$ touch each other (as is expected).
If the macroscopic loops of $\calF_\delta$ are shown to touch each other, then the same follows for those of $\calF$. 
However, the opposite is not true; it may be that the loops of $\calF$ touch each other, 
while the macroscopic loops of $\calF^\delta$ come within a mesoscopic distance of one other (as is expected when $q = 4$). 
The self-touching property for $\calF_\delta$ is useful for
\begin{itemize}
\item obtaining the full scaling limit of discrete interfaces (we refer to \cite{AizBur99,KemSmi16} for examples);
\item applying \cite{MilSheWer17} to derive, for instance for $q$ equal to 2 or 3, the convergence of interfaces in the $q$-state Potts model from the convergence of the interfaces in the random-cluster model by using a continuous version of the Edwards-Sokal coupling where clusters of the CLE($\kappa$) are colored in one of $q$ colors.
\end{itemize}

Below we show that macroscopic clusters (or equivalently macroscopic loops of $\calF^\delta$) do touch each other with high probability.
We illustrate this informal statement by three results that we believe could prove useful.
Other similar results may be obtained from Theorem~\ref{thm:sRSW} if needed. 
We insist on the fact that $q < 4$ is necessary here (see Section~\ref{sec:q=4}). 

Call a {\em chain of clusters} any sequence of distinct cluster $\mathbf C_1,\dots, \mathbf C_k$ 
with the property that, for each $1\leq j < k$, there exists a closed edge connecting $\bfC_j$ to $\bfC_{j+1}$. 
We say that such a chain {\em connects} two sets of vertices $A$ and $B$ if $\bfC_1$ intersects $A$ and $\bfC_k$ intersects $B$. 
\newcommand{\Rect}{{\rm Rect}}
For $N \geq 1$ and $\ell >0$, let $\Rect = \Rect(N,\ell) = [0,\ell N] \times [0,N]$ be the rectangle of aspect ratio $\ell$ and size $N$.

\begin{theorem}[Crossings of rectangles by chains of large clusters]\label{thm:self-touching}
	For $\alpha \geq 0 $ and $K\geq 1$, let $\calG(K,\alpha,N,\ell)$ be the event 
	that there exists a chain of at most $K$ clusters of $\omega \cap \Rect(N,\ell)$ 
	connecting the left and right sides of $\Rect(N,\ell)$, all of which have a diameter at least~$\alpha N$. 
	Then, for every $\ep,\ell > 0$, there exist $K \geq 1$ and $\alpha >0$ such that for every $N\ge1$,
	$$ \phi_{\bbZ^2}[\calG(K,\alpha,N,\ell)] \geq 1 - \eps.$$ 
\end{theorem}

This theorem is expected to fail for $q=4$: 
in this case the shortest chain crossing the rectangle should contain either one or a logarithmic number of clusters.
Below, we give two consequences of the theorem, closer to the informal statements announced. 

\begin{corollary}[Large clusters are connected by chains of large clusters]\label{cor:chain}~
	Write $\calH(K, \alpha, \delta ,N)$ for the event that any two clusters $\bfC$, $\bfC'$ of $\omega\cap \La_N$
	of  diameter at least $\delta N$ are connected by a chain of at most $K$ clusters, each of diameter at least $\alpha N$.
Then, for every $\ep,\delta > 0$, there exist $K \geq 1$ and $\alpha >0$ such that for every $N\ge1$,
	$$ \phi_{\bbZ^2}[\calH(K,\alpha,\delta,N)] \geq 1 - \eps.$$ 
\end{corollary}

\begin{corollary}[Neighbouring large clusters touch each other]\label{cor:clusters_touch}~
	Write $\calF(\alpha, \delta ,N)$
	for the event that  there exist clusters $\bfC$, $\bfC'$ of $\omega \cap \La_N$ 
	of diameter at least $\delta N$ and such that $1< {\rm dist}(\bfC,\bfC') \leq \alpha N$.
	Then, for every $\ep,\delta > 0$, there exists $\alpha >0$ such that for every $N\ge1$,
	\begin{align}\label{eq:clusters_touch}
		\phi_{\bbZ^2}[\calF(\alpha,\delta,N)] < \eps.
	\end{align}
\end{corollary}

We start with the proof of the theorem, which is based on the following two steps. 
First, we will show that the event $\calG(K,0,N,\ell)$, that there exists a chain of at most $K$ clusters crossing $\Rect$, 
regardless of their diameter, occurs with high probability. 
Then, we will show that in any such chain, the clusters are actually large. 
The first step is contained in the following lemma. 

\begin{lemma}\label{lem:bounded hamming}
	For every $\ep>0$ and $\ell > 0$, there exists $ K \geq1$ such that for every $N \ge 1$,
	$$ \phi_{\bbZ^2}[\calG(K,0,N,\ell)] \geq 1 - \eps.$$ 
\end{lemma}

In other words, the above states that the Hamming distance to the crossing of $\Rect$ is bounded by $K$ with high probability. 

\begin{proof}
	Fix $\eps$, $\ell$ and $N$. 
	For $k\geq 0$ let $\calR_k$ be the set of vertices connected to the left side of $\Rect$ by a path containing at most $k$ closed edges. 
	If $\calR_k$ does not intersect the right side of $\Rect$, 
	let $\calD_k$ be the connected component of $\Rect \setminus \calR_k$ that contains the right side of $\Rect$. 
	Since the configuration inside $\calD_k$ does not depend on the states of edges in $\calR_k$, 
	the measure induced by $\phi_{\bbZ^2}[\cdot| \calR_k]$ inside $\calD_k$ dominates $\phi_{\calD_k}^0$. 
	View $\calD_k$ as a quad with the arc $(ab)$ being the right side of $\Rect$ 
	and the arc $(cd)$ being the boundary of $\calD_k$ (see Fig.~\ref{fig:hamming}). 
	Using Theorem~\ref{thm:sRSW} and the fact that $\ell_{\calD_k}[(ab),(cd)]\le \ell$, 
	we find 
	\begin{align*}
		\phi_{\bbZ^2}[ \calC(\calD_k) \,| \, \calR_k] \geq \phi_{\calD_k}^0[ \calC(\calD_k)]\geq \eta(\ell).
	\end{align*}
	Finally, observe that if $\calC(\calD_k)$ occurs, then so does $\calG(k+1,0,N,\ell)$. Thus
	\begin{align*}
		\phi_{\bbZ^2}[\calG(k+1,0,N,\ell)\,|\,\calG(k,0,N,\ell)^c] \geq \eta(\ell).
	\end{align*}
	We may therefore fix $K \geq 0$ depending only on $\eps$ and $\eta(\ell)$ 
	such that $\phi_{\bbZ^2}[ \calG(K,0,N,\ell)] \geq 1- \eps $ for every $N\ge1$. 
\end{proof}

\begin{figure}
	\begin{center}
	\includegraphics[width = 0.49\textwidth]{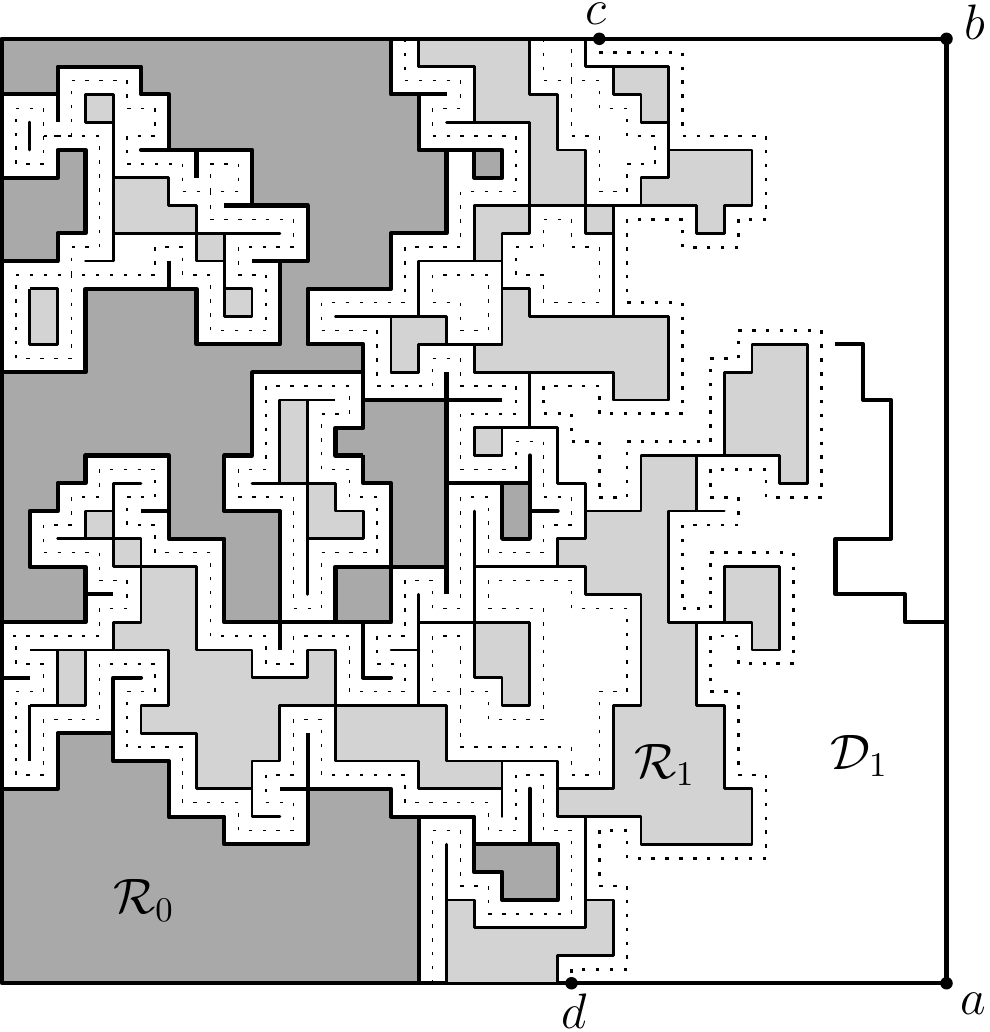}\hspace{0.01\textwidth}
	\includegraphics[width = 0.49\textwidth]{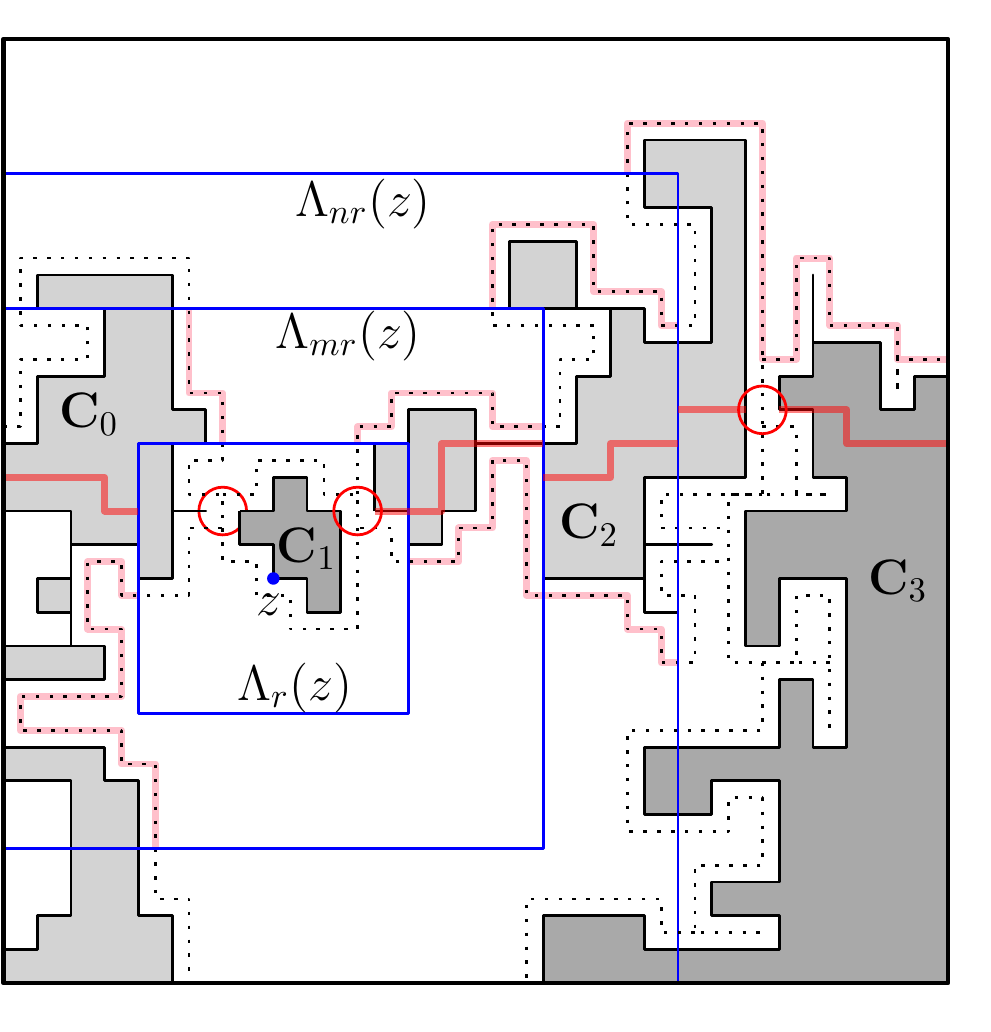}
	\caption{{\em Left:} The sets $\calR_0$ and $\calR_1$ in dark and light grey, respectively. 
	The connected component of the complement containing the right side of $\Rect$ is $\calD_1$. 
	Viewed as a quad, there is a positive probability that it is crossed horizontally, which would induce $\calG(2,0,N,\ell)$. 
	{\em Right:} a configuration with $k=3$; the clusters $\bfC_0,\dots,\bfC_3$ are depicted. 
	The edges $e_1,e_2, e_3$ are marked by red circles (their choice is not unique). 
	The boxes $\La_{r}(z)$, $\La_{mr}(z)$ and $\La_{nr}(z)$ are marked in blue; notice the arms of different types between them in red.}
	\label{fig:hamming}
	\end{center}
\end{figure}

The second step in the proof of Theorem~\ref{thm:self-touching} is provided by the lemma below. Some notation is required. 
For $R\ge r\ge1$ and $\sigma \in \{0,1\}^j$, the quarter plane arm event between radii $r$ and $R$, denoted $A_{\sigma}^{++}(r,R)$, 
is defined as $A_{\sigma}^+(r,R)$, with the arms being restricted to the quarter plane $\bbZ_+^2$.
For $K\geq 0$, an arm in $\Lambda_R\setminus\Lambda_r$ of type $1$ (resp.~type $0$) with $K$ defects 
is a path crossing $\Lambda_R\setminus\Lambda_r$ which contains at most $K$ closed (resp.~open) edges. 
For $K\geq 1$ and $\sigma$ of length $j$, define $A_{\sigma}(K;r,R)$ as the event that there exist 
$j$ disjoint arms with $K$ defects in total, from the inner to the outer boundary of $\Lambda_R\setminus\Lambda_r$, 
which are of type $\sigma_1,\dots, \sigma_j$, when indexed in counterclockwise order.
Define the half-plane and quarter-plane arm events with defects, written $A^+_\sigma(K;r,R)$ and $A^{++}_\sigma(K;r,R)$, in the same way.

The following lemma is a straightforward consequence of~\eqref{eq:SIX-ARM}, 
the crossing estimates~\eqref{eq:wRSW} and the quasi-multiplicativity~\eqref{eq:MULTIPLICATIVITY};
see \cite[Prop.~17]{Nol08} for a proof in the case of percolation which adapts readily to our case. 
Note that the sequence in the first equation is not the same as in~\eqref{eq:SIX-ARM}, but the same bound may be proved without difficulty.
	
\begin{lemma}\label{lem:arm_defects}
	There exist $c_{19},C_{19}>0$ such that for every $K$ and every $R\ge r\ge1$, 
	\begin{align*}
		\phi_{\bbZ^2}[A_{100100}(K;r,R)]  & \leq C_{19}[1 + \log(R/r)]^K \cdot (r/R)^{2+c_{19}},\\
		\phi_{\bbZ^2}[A_{10}^+(K;r,R)] 	  & \leq C_{19}[1 + \log(R/r)]^K \cdot (r/R)^2,\\
		\phi_{\bbZ^2}[A_{10}^{++}(K;r,R)] & \leq C_{19}[1 + \log(R/r)]^K \cdot (r/R)^{1+c_{19}}.
	\end{align*}
\end{lemma}

We are now in a position to prove Theorem~\ref{thm:self-touching}.
\begin{proof}[Theorem~\ref{thm:self-touching}]
	Fix $\eps,\ell > 0$ and $N$. For clarity, we will assume $\ell \geq 1$; the same proof applies for $\ell <1$. 
	By Lemma~\ref{lem:bounded hamming}, we may choose $K = K(\eps,\ell) \geq 1$ such that
	$\calG(K,0,N,\ell)$ occurs with probability at least $1 - \eps$.
	Henceforth we assume $K$ fixed as above, and focus on configurations $\omega$ in $\calG(K,0,N,\ell)$.
	Our goal is to prove that $\omega \in \calG(K,\alpha,N,\ell)$ with high probability for some $\alpha >0$. 
	
	If $\omega$ contains a horizontal crossing of $\Rect$, then the cluster containing the crossing has diameter at least $\ell N$ 
	and $\omega\in \calG(K,\alpha,N,\ell)$ for any $0 < \alpha \leq \ell$. 
	
	Next, we focus on the situation where $\omega$ does not contain a horizontal crossing of $\Rect$. 
	Let $e_1,\dots, e_k$ be a minimal set of edges such that $\omega\cup\{e_1,\dots, e_k\}$ contains a horizontal crossing of $\Rect$.
	Write $e_i = (u_i,v_i)$ with $v_i$ connected to $u_{i+1}$ in $\omega \cap \Rect$ for all $1 \leq i < k$, 
	and $u_0$ and $v_k$ connected to the left and right sides of $\Rect$, respectively. 
	Write $\bfC_i$ for the cluster of $v_i$ in $\omega \cap \Rect$ and $\bfC_{0}$ for the cluster of $u_1$.
	By the minimality of $e_1,\dots, e_k$, the clusters $\bfC_0,\dots, \bfC_k$ are all distinct. 
	
	We start by analysing $\bfC_j$ with $1 \leq j < k$. 
	Fix such a value $j$ and let us assume that $\|v_j - u_{j+1}\| \leq \alpha N$ for some small constant $\alpha >0$ to be chosen later
	(here $\|\cdot\|$ denotes the $L^\infty$ norm). 
	Write $r = \lfloor \alpha N \rfloor$ and assume for convenience that $N /r \in \bbN$ and $\ell N /r \in \bbN$.
	
	By our assumption on the distance between $v_j$ and $u_{j+1}$, there exists a $r$-box $\La_r(z)$ containing both $v_j$ and $u_{j+1}$. 
	Let $m r$ be the distance between $z$ and $\partial \Rect$ 
	(recall that $z \in r\bbZ^2$ and that $N/r \in \bbN$, hence this distance is indeed an integer multiple of $r$, as the notation suggests). 
	Let $n r$ be the distance from $z$ to the three sides of $\partial \Rect$ furthest from $z$. 
	Then, we claim that in $\tilde\omega:=\omega\cup\{e_1,\dots, e_{j-1}, e_{j+2}, \dots, e_k\}$ there exist\footnote{
		Let us justify (a)--(c).
    	Let $C_{\text{left}}$ and $C_{\text{right}}$ be the clusters in $\tilde\omega$ of $u_{j}$ and $v_{j+1}$, respectively.
    	Then $C_{\text{left}}$ intersects the left side of $\Rect$ and $C_{\text{right}}$ the right side. 
    	Write $\partial C_{\text{left}}$ for the paths of dual-open edges separating  $C_{\text{left}}$ from $\Rect \setminus  C_{\text{left}}$.
    	Define $\partial C_{\text{right}}$ in the same way. 
    	By the minimality of $e_1,\dots, e_k$, the boundaries $\partial C_{\text{left}}$ and $\partial C_{\text{right}}$ of the two clusters are disjoint 
		(since we are considering the edge-boundaries of ${\bf C}_j$ and ${\bf C}_{j+2}$ and that an intersection would imply the existence of a shortcut). 
    	
    	We start by explaining (a): both clusters $C_{\text{left}}$ and $C_{\text{right}}$ and their boundaries intersect $\La_r$ and $\partial \La_{mr}$. 
    	It follows that each cluster contains a primal open-path between $\La_r$ and $\partial \La_{mr}$
    	and the boundaries of the two clusters contain two dual-open paths between   $\La_r$ and $\partial \La_{mr}$ each. 
    
    	We move on to (b): suppose (as in Fig.~\ref{fig:hamming}) that $z$ is closest to the left side of $\Rect$. 
    	Then $C_{\text{right}}$ and its boundary intersect both $\La_{mr}$ and $\partial \La_{nr}$; 
    	as a consequence $C_{\text{right}}$ contains a path of open edges and its boundary contains two disjoint paths 
		of dual-open edges between $\La_{mr}$ and $\partial \La_{nr}$.
    	The same holds when $z$ is closest to the right side of $\Rect$.
    	If $z$ is closest to the top or bottom of $\Rect$, then $C_{\text{left}}$ and $C_{\text{right}}$ contain one arm of type $1$ each, 
		and their boundaries an arm of type $0$ each. Hence we deduce that the half-plane arm event with arms of types $1001$ occurs. 
		This is obviously contained in the half-plane arm event with arms of types $101$. 
    
    	Finally we show (c): 
    	The two sides closest to $z$ are necessarily adjacent. Hence we may suppose (as in Fig.~\ref{fig:hamming}) that they are the left and bottom ones. 
    	Then, $C_{\text{right}}$ contains an arm of type $0$ from $\La_{nr}$ to the right side of $\Rect$, contained in the quarter-plane $\bbZ_+^2$, 
    	and its boundary contains an arm of type $0$ from $\La_{nr}$ to either the top or the right side of $\Rect$, contained in the same quarter-plane. 
    
		We deduce that arms as in (a)--(c) also exist in $\omega$, but with $K$ defects at most.
	}:
	\begin{itemize}
    	\item[(a)] six arms in the annulus $\La_{mr}(z)\setminus \La_r(z)$ of types $100100$;
    	\item[(b)] three arms in the half-space annulus $(\La_{nr}(z)\setminus \La_{mr}(z)) \cap \Rect$ 
    	of types $101$ (if the side of $\Rect$ closest to $z$ is either the left or right one) or $010$ (otherwise);
    	\item[(c)] two arms in the quarter-plane annulus $\Rect \setminus \La_{nr}(z)$ of types $10$ or $01$, 
    	from $\La_{nr}(z)$ to the two sides of $\Rect$ furthest from $z$. 
	\end{itemize}
	Write $E(z)$ for the event that arms as in (a)-(c) exists, with at most $K$ defects.	The mixing property ~\eqref{eq:mix2} and Lemma~\ref{lem:arm_defects} give that 
	\begin{align*}
		\phi_{\bbZ^2}[E(z)] 
		&\leq C \phi_{\bbZ^2}[A_{100100}(K;r,mr)] \phi_{\bbZ^2}[A_{101}^{+}(K;mr,nr)]\, \phi_{\bbZ^2}[A_{10}^{++}(K;nr,N)]	\\
		&\leq C' [1+\log (N/r)]^{3K}(1/m)^{2 + c_{19}} (m/n)^{2}  (nr/N)^{1 + c_{19}} .
	\end{align*}
	Now, observe that for any $m\leq n\leq N/r$, 
	there exist at most eight points $z \in r \bbZ^2 \cap \Rect$ at a distance $mr$ from $\partial \Rect$ 
	and a distance $nr$ to the three sides of $\Rect$ furthest from $z$. 
	Applying a union bound, it follows that
	\begin{align*}
		\phi_{\bbZ^2}\big[\bigcup_{z}E(z)\big]
		& \leq 8C'  [1+\log (N/r)]^{3K} \sum_{1 \leq m \leq n \leq N/r} (1/m)^{2 +c_{19}} (m/n)^{2}  (nr/N)^{1 + c_{19}}\\
		&\leq 8C' [1+\log (N/r)]^{3K} (r/N)^{c_{19}},
	\end{align*}
	where the union if over all $z \in r \bbZ^2 \cap \Rect$ and the last inequality is obtained through straightforward computation. 
	Suppose now that $\alpha$ (and hence $r/N$) is chosen small enough such that the above is smaller than $\eps$; 
	notice that the choice of $\alpha$ depends on $K$, but that $K$ only depends on $\eps$ and $\ell$, not on $N$. 
	Then,
	\begin{align}\label{eq:HK}
		\phi_{\bbZ^2}[\calG(K,0,N,\ell) \cap \bigcap_{z}E(z)^c ] \geq 1- 2\eps.
	\end{align}
	Moreover, on the event above, each ${\bf C}_j$ with $1 \leq j < k$ has diameter at least $\alpha N$, 
	since it contains two points $v_j$ and $u_{j+1}$ at a distance at least $\alpha N$ from each other.
	
	At this stage, we should also exclude that the diameters of ${\bf C}_0$ and ${\bf C}_k$ are small. 
	We do this below through a similar argument as for the diameters of ${\bf C}_1,\dots, {\bf C}_{k-1}$.
	Suppose that the diameter of ${\bf C}_0$ is smaller than $r$. 
	Then, there exists $z \in \{0\}\times r\bbZ$ so that $u_1 \in \La_r(z)$. 
	Let $nr$ denote the distance from $z$ to the top and bottom of $\Rect$. 
	The same analysis as above shows that $\omega$ contains 
	\begin{itemize}
    	\item[(b')] three arms in the half-space annulus $(\La_{nr}(z)\setminus \La_{r}(z)) \cap \Rect$ of types $010$ with at most $K$ defects;
    	\item[(c')] two arms, with at most $K$ defects, in the quarter-plane annulus $\Rect \setminus \La_{nr}(z)$ of types $01$ or $10$, 
    	from $\La_{nr}(z)$ to the two sides of $\Rect$ furthest from $z$. 
	\end{itemize} 
	We recognise above the event $E(z)$ for $z$ on the left boundary of $\Rect$.
	The same analysis applies  when the diameter of ${\bf C}_k$ is smaller than $r$. 
	In conclusion, the event in~\eqref{eq:HK} guarantees the occurrence of $\calG(K,\alpha,N,\ell)$, and the bound in~\eqref{eq:HK} implies the result. 	
\end{proof}

Finally, we prove the two corollaries.

\begin{proof}[Corollary~\ref{cor:chain}]
	Fix $\eps, \delta > 0$ and $N$. 
	We may assume $N$ larger than some threshold, and for simplicity we will consider $\delta N$ and $1/\delta$ to be integers. 

	Partition $\La_N$ into strips $S_j = [-N,N] \times [j\delta N, (j+1)\delta N]$ with $-1/\delta \leq j <1/\delta$. 
	Each strip $S_j$ is a translate of $\Rect(\delta N, 2/\delta)$. 
	Let $\alpha$ and $K$ be such that 
	$$\phi_{\bbZ^2}[\calG(K,\alpha/\delta ,\delta N,2/\delta)] \geq 1 - \eps\delta/4.$$
	By Theorem~\ref{thm:self-touching}, such values of $\alpha > 0$ and $K\geq 0$ exist, and only depend on $\eps$ and $\delta$. 
	
	Write $\calG_h$ for the event that $\calG(K,\alpha/\delta,\delta N,2/\delta)$ occurs in every strip $S_j$, and $\calG_v$ for the rotation by $\pi/2$ of $\calG_h$. 
	Then, due to our choice of $\alpha$ and $K$,
	\begin{align}\label{eq:Ghv}
		\phi_{\bbZ^2}[\calG_h \cap \calG_v] \geq 1 - \eps.
	\end{align}
	Moreover, we claim that if $\calG_h \cap \calG_v$ occurs, then any two clusters in $\La_N$ of diameter at least $2 \delta N$ 
	are connected by a chain of at most $2K$ clusters of diameter at least $\alpha N$. 
	
	Indeed, any cluster of diameter at least $2 \delta N$ contains a vertical crossing of a strip $S_j$, 
	or a horizontal crossing of the rotation by $\pi/2$ of a strip $S_j$. 
	As such, it is contained in one of the chains of clusters crossing horizontally the strips $S_j$, or vertically their rotations. 
	Finally, if $\bfC$ and $\bfC'$ are members of two such chains, we can exhibit a chain of clusters connecting $\bfC$ to $\bfC'$ 
	by following a chain crossing some $S_j$ horizontally, then one crossing vertically the rotation of some $S_{j'}$. 
	By construction, the chain thus obtained contains only clusters of diameter at least $\alpha N$ and at most $2K$ of them. 
	In conclusion,~\eqref{eq:Ghv} implies that
	$\phi_{\bbZ^2}[\calH(2K,\alpha,2\delta,N)] \geq 1 - \eps. $
\end{proof}

\begin{proof}[Corollary~\ref{cor:clusters_touch}]	
	Fix some $\delta >\alpha > 0$ and $N \geq 1$. Assume for simplicity that $r := \alpha N$ and $N/r=1/\alpha$ are integers.
	By the same analysis as in the proof of  Theorem~\ref{thm:self-touching}, 
	if $\calF(\alpha, \delta,N)$ occurs, then there exists a $r$-box $\La_r(z)$ that intersects two clusters $\bfC$ and $\bfC'$
	of diameters at least $\delta N$ and at a distance at least $2$ from each other. 
	Write $m r$ for the distance between $z$ and $\partial \La_N$.
	Then we claim that $\omega$ contains
	\begin{itemize}
    	\item[(a)] six arms in the annulus $\La_{mr}(z)\setminus \La_r(z)$ of types $100100$;
    	\item[(b)] four arms in the half-space annulus $(\La_{\delta N}(z)\setminus \La_{mr}(z)) \cap \La_N$ of types $1001$.
	\end{itemize}
	Indeed, the two primal arms are provided by $\bfC$ and $\bfC'$ and the dual ones by their disjoint boundaries. 
	Following the proof of Theorem~\ref{thm:self-touching}, there exists $C > 0$ independent of $\alpha,\delta$ or $N$ such that 
	\begin{align*}
		\phi_{\bbZ^2}[\calF(\alpha, \delta,N)]
		&\leq C (\alpha/\delta)^{c_{19}}.
	\end{align*}
	Thus, for $\eps,\delta > 0$, in order to obtain~\eqref{eq:clusters_touch},
	it suffices to choose $\alpha$ small enough for the above to be smaller than $\eps$.
\end{proof}

\newcommand{\etalchar}[1]{$^{#1}$}
\providecommand{\bysame}{\leavevmode\hbox to3em{\hrulefill}\thinspace}
\providecommand{\MR}{\relax\ifhmode\unskip\space\fi MR }
\providecommand{\MRhref}[2]{%
  \href{http://www.ams.org/mathscinet-getitem?mr=#1}{#2}
}
\providecommand{\href}[2]{#2}

\end{document}